\documentclass[12pt]{article}
\pdfoutput=1
\usepackage{Max}
\usepackage{tikz-cd}

\usepackage[shortlabels]{enumitem}

\usepackage{mathrsfs}
\usepackage{tcolorbox}
\definecolor{lightgrey}{rgb}{0.9,0.9,0.9}
\definecolor{lightblue}{rgb}{.8,.9,.95}
\definecolor{lightgreen}{rgb}{.7,1,.8}
\definecolor{lightorange}{rgb}{1,.9,.7}
\usepackage{empheq}

\title{Hyper-K\"ahler manifolds from Riemann-Hilbert problems I: Ooguri-Vafa-like model  geometries}
\author[1]{Laura Fredrickson\thanks{lfredric@uoregon.edu}}
\author[2]{Max Zimet\thanks{maxzimet@gmail.com}}
\affil[1]{Department of Mathematics,

University of Oregon, Eugene, OR 97403 USA}

\affil[2]{The Voleon Group,

Berkeley, CA 94704 USA}

\date{}

\definecolor{coolcolor}{rgb}{0.6,0,1}

\definecolor{greencolor}{rgb}{0,0.6,0.5}

\newcommand{\hqnoxi}{\mathord{/ \!\! / \! \! /}}

  \theoremstyle{plain}
\newtheorem{thm}{Theorem}[section]
\newinttheorem{cor}{Corollary}
\newinttheorem{lem}{Lemma}
\newinttheorem{prop}{Proposition}
\newinttheorem{conj}{Conjecture}
 \newinttheorem{axiom}{Axiom}
 \newinttheorem{ass}{Assumption}
 \theoremstyle{definition}
 \newinttheorem{defn}{Definition}
 \newinttheorem{notation}{Notation}
 \newinttheorem{ex}{Example}
 \newinttheorem{constr}{Construction} 
 \newtheorem*{thm*}{Theorem}
 \theoremstyle{remark}
 \newinttheorem{rmk}{Remark}
 \newinttheorem{rem}{Remark}
 
\def\beq{\begin{eqnarray}}
\def\eeq{\end{eqnarray}}
 \newcommand{\bp}{\begin{proof}[Proof]}
 \newcommand{\ep}{\end{proof}}

\newcommand\nc{\newcommand}
\nc\mc\mathcal
\nc\mb\mathbb
\nc\Up\Upsilon

 \DeclareFontFamily{U}{wncy}{}
    \DeclareFontShape{U}{wncy}{m}{n}{<->wncyr10}{}
    \DeclareSymbolFont{mcy}{U}{wncy}{m}{n}
    \DeclareMathSymbol{\Sh}{\mathord}{mcy}{"58} 

\begin{document}

\maketitle

\begin{abstract}
We construct model hyper-K\"ahler geometries that include and generalize the multi-Ooguri-Vafa model using the formalism of Gaitto, Moore, and Neitzke. 

This is the first paper in a series of papers making rigorous Gaiotto--Moore--Neitzke's formalism for constructing hyper-K\"ahler metrics near semi-flat limits. In that context, this paper describes the assumptions we will make on a sequence of lattices $0 \to \Gamma_{f} \to \widehat{\Gamma} \to \Gamma \to 0$ over a complex manifold $\B'=\B - \B''$ near the singular locus, $\B''$, in order to define a smooth manifold $\M \to \B$ and hyper-K\"ahler model geometries on neighborhoods of points of the singular locus. In follow-up papers, we will use a modified version of Gaiotto--Moore--Neitzke's iteration scheme starting at these model geometries to produce true global hyper-K\"ahler metrics on $\M$.
\end{abstract}

\hypersetup{linkcolor=blue}

\section{Introduction}\label{sec:introduction}

This paper is the first in a series of papers making rigorous (a generalization and modification of) the formalism of Gaiotto, Moore, and Neitzke \cite{GMN:walls} for constructing hyper-K\"ahler manifolds near semi-flat limits.  
From the differential geometric standpoint, this formalism is a machine for translating enumerative data (as well as data describing the semi-flat limit) into explicit smooth hyper-K\"ahler manifolds near semi-flat limits. Ultimately, this  provides a general framework for proving results about Gromov-Hausdorff collapse of hyper-K\"ahler manifolds to semi-flat limits and completes the Strominger-Yau-Zaslow \cite{strominger:mirrorT} conception of mirror symmetry for hyper-K\"ahler manifolds at the level of hyper-K\"ahler geometry (as opposed to only constructing one complex structure).

\bigskip

This first paper is largely preparatory in nature. The main result is a construction of local hyper-K\"ahler model geometries that are multi-Ooguri-Vafa-like via the Gaiotto, Moore, Neitzke formalism and twistor theorem.  This paper expands on \cite[\S4]{GMN:walls}. There, Gaiotto--Moore--Neitzke describe the Ooguri-Vafa model geometry in their formalism. However, rather than showing \emph{directly} via the twistor theorem \cite{hitchin:hkSUSY} that their formalism for making a $\CC^\times$ family 
\be \varpi(\zeta) = - \frac{i}{2\zeta} \omega_+ + \omega_3 - \frac{i}{2} \zeta \omega_- \ee
of closed $2$-forms gives a hyper-K\"ahler metric, they show that the shape of $\varpi^{\rm model}(\zeta)$ matches known expression for the Ooguri-Vafa model in the Gibbons-Hawking Ansatz\cite{gross:OV}. In \cite[\S4.7]{GMN:walls}---only 2 pages!---, Gaiotto--Moore--Neitzke write down higher rank generalizations of Ooguri-Vafa and multi-Ooguri-Vafa geometries, but only in the Gibbons-Hawking Ansatz. We note that for these higher rank models they never prove that there is a corresponding hyper-K\"ahler structure, and in particular do not mention nondegeneracy. We describe these same local models in their formalism and use a concrete variant of the twistor theorem \cite{FZ:twistor} to prove their formalism gives rise to a hyper-K\"ahler structure.  (Technically, our models are slightly more general, since Gaiotto--Moore--Neitzke restrict to unimodular lattices (see \S\ref{sec:lattices}). However, they coincide for unimodular lattices.)

\bigskip
\noindent\textbf{Models. } In four real dimensions, these local models are simply the Ooguri-Vafa and multi-Ooguri-Vafa models (with both a single $I_N$ singular fiber or perturbations thereof $I_N = I_{N_1} + \cdots I_{N_k}$ for $N=N_1 + \cdots + N_k$), with the option to turn on some background local holomorphic function on $\B$. 

Multi-Ooguri-Vafa models are ubiquitious in degenerations of 4d hyper-K\"ahler manifolds to semi-flat limits. Ooguri-Vafa model is a key part of Gross--Wilson's description in \cite{gross:OV}  of K3 metrics near semi-flat metrics, under the assumption that $\pi: K3 \to \B=\mathbb{P}^1$ has twenty-four $I_1$ singular fibers.  In \cite{chen:k3sing}, Chen--Viaclovsky--Zhang complete the story by allowing all singular fibers, including $I_N$; their paper includes a detailed treatment of multi-Ooguri-Vafa models with exactly one $I_N$ singular fiber. However, under small perturbations the single $I_N$ fiber can decompose to $I_{N_1} + \cdots + I_{N_k}$ for $N=N_1 + \cdots + N_k$.  In their paper, it was sufficient to treat $I_{N_1} + \cdots + I_{N_k}$ as $k$ distinct $I_{N_1}, \cdots, I_{N_k}$ models. However, once one starts talking about small fiber degenerations in the moduli space of $K3$s in which the fibration $\pi:K3 \to \B$ changes, it is  necessary to treat $I_N$ and its perturbations in a uniform way. The Gaiotto--Moore--Neitzke formalism gives some complicated geometric package that produces these 4d multi-Ooguri-Vafa metrics, and naturally treats the above degenerations or collisions of singular fibers uniformly.  Consequently, our estimates are sharper than previous estimates.

Moreover, in Gaiotto--Moore--Neitzke's formalism, there is a natural higher rank generalization of multi-Ooguri-Vafa metrics. Like all hyper-K\"ahler metrics conjecturally produced by Gaiotto--Moore--Neitzke's formalism, these manifolds are fibered over a half-dimensional complex base, $\B$: \be\pi: \M \to \B,\ee
and the fibers are generically abelian varieties. 
 When the lattice $\widehat{\Gamma} \to \B$ in Gaiotto--Moore--Neitzke's formulation is unimodular, these higher rank generalizations include products of multi-Ooguri-Vafa models and flat $\mathbb{R}^4$, with the option, again, to turn on some background holomorphic data. Local models associated to non-unimodular lattices are some quotient-like generalization thereof.  While these are not qualitatively new hyper-K\"ahler manifolds, for any singular fiber $\pi^{-1}(u)$ and choice of fiber length scale $\frac{1}{\pi}$ we are able to \emph{quantity} the size of the neighborhood of $u$ in $\B$ admitting a model hyper-K\"ahler metric. This will be useful for us in future papers in this series, and hopefully will be useful for work of others as well.

More broadly, just as the four-dimensional Ooguri-Vafa model has a triholomorphic $S^1$-action hence is a hypertoric manifold\footnote{We note that the $\mathbb{Z}$-cover of Ooguri-Vafa model is more standard in algebraic geometry than Ooguri-Vafa itself. It is of infinite-topological type with a chain of $\mathbb{P}^1$s.  It does not arise via a finite-dimensional hyper-K\"ahler quotient.}, some of these higher multi-Ooguri-Vafa models are also hypertoric manifolds\footnote{Our local models of real dimension $4r$ have a $T^s$ action, but $s$ is not always equal to $r$.}.  Some higher-dimenisonal verions of Ooguri-Vafa are discussed in \cite[\S7]{Dancer}, whose $\mathbb{Z}^n$ cover is described---in the higher-dimensional Gibbons-Hawking framework---via a $T^n$ bundle over $\mathbb{R}^{3n}$ degenerating over a union of real codimension planes invariant under some $\mathbb{Z}^n$ action. Our models include all of these.

 \bigskip

 \noindent\textbf{Method.} 
 We choose to construct these multi-Ooguri-Vafa-like models via Gaiotto--Moore--Neitzke's integral relation, rather than via the simpler approach via the higher-dimensional analogue of the Gibbons-Hawking ansatz (though we do give this perspective in \S\ref{sec:GH}), because it relates better to our ultimate aim of making rigorous the formalism of Gaiotto, Moore, and Neitzke. After appropriate modification, Gaiotto--Moore--Neitzke's formalism produces a $\mathbb{P}^1_\zeta$-family of closed $2$-forms on $\M'$, an open dense subset of the desired manifold $\M$. For these model geometries, no modification is required. We will appeal to a concrete variant of Hitchin--Karlhede--Lindstr\"om--Ro\v{c}ek's twistor theorem (see \cite[Theorem 3.16b]{FZ:twistor}) in order to produce a pseudo-hyper-K\"ahler structure.
 \begin{theorem}[Theorem 3.16b of \cite{FZ:twistor}]\label{thm:twistor}
    Let \be \varpi(\zeta) = - \frac{i}{2\zeta} \omega_+ + \omega_3 - \frac{i}{2} \zeta \omega_- \ee be a family of holomorphic symplectic forms on a real manifold $\M$ such that $\omega_- = \bar \omega_+$ and $\omega_3=\bar \omega_3$. Moreover, assume that $\omega_+$ is a holomorphic symplectic form on $\M$.
    Then,
    $\M$ is a pseudo-hyper-K\"ahler manifold with pseudo-K\"ahler forms $(\omega_1,\omega_2,\omega_3)$ corresponding to the complex structures $(J_1,J_2,J_3)$, where $\omega_\pm = \omega_1 \pm i \omega_2$.
    \end{theorem}
At the end, we will check the signature. Here, 
\begin{definition}\label{def:holoSympDef}
A \emph{holomorphic symplectic form on a real manifold} $\M$ is a closed 2-form $\Omega$, which we regard as a map $\Omega: T_\CC\M\to T^*_\CC\M$, such that $T_\CC \M = \ker\Omega \oplus \ker \bar\Omega$.
\end{definition}
Such a holomorphic symplectic form on a real manifold $\M$ uniquely determines a complex structure $J$ on $\M$ such that $\Omega$ is a holomorphic symplectic form on the complex manifold $(\M,J)$. A convenient means to prove that a $2$-form is holomorphic symplectic is given by 
\begin{proposition}[Proposition 3.4 of \cite{FZ:twistor}] \label{prop:holoVol}
    A closed 2-form $\Omega$ on a real $4r$-manifold $\M$ is a holomorphic symplectic form if and only if $\Omega^r\wedge\bar\Omega^r$ vanishes nowhere and $\dim_\CC \ker\Omega \ge 2r$.
    \end{proposition}
    The condition $\dim_\CC \ker \Omega\ge 2r$ is particularly convenient because the nullity of $\Omega$ is upper semicontinuous.

    The nondegeneracy of $\Omega^r \wedge \bar \Omega^r$ is the biggest challenge throughout this paper. 

    We must to separately prove that $\varpi(\zeta)$ is a holomorphic symplectic form for $\zeta \in \CC^\times$ and that $\omega_+$ is a holomorphic symplectic form. When $\zeta \in \CC^\times$, $\varpi(\zeta)$ will naturally be written
\begin{equation}
\varpi(\zeta) = \sum_{i=1}^r d \mathcal{Y}_{e_i}(\zeta) \wedge d \mathcal{Y}_{m_i}(\zeta)
\end{equation}
in certain open set $O$ contained in $\M'$, an open dense set of $\M$. By analogy with usual symplectic forms, we refer to these coordinates as holomorphic Darboux coordinates. For $\omega_+$, we will have to work a little harder to produce holomorphic Darboux coordinates.  These holomorphic Darboux coordinates provide us with a convenient way to prove that $\varpi(\zeta), \omega_+$ are holomorphic symplectic forms on $O$: 
\begin{corollary}\label{cor:twistor}
The $2$-form \be \Omega = \sum_{i=1}^r d \alpha_i \wedge d \beta_i \ee is a holomorphic symplectic form an open set $O$ if 
the Jacobian determinant of \be \{\Real \alpha_i, \Imag \alpha_i, \Real \beta_i, \Imag \beta_i\}_{i=1}^r\ee with respect to the natural smooth coordinates on $O$ is non-vanishing.
\end{corollary}
\begin{proof}
 Our proof is based on Proposition \ref{prop:holoVol}.
The functions $\{\alpha_i, \beta_i\}_{i=1}^r$ are non-degenerate complex coordinates on $O$ if the determinant of the Jacobian of the coordinates $$\{\Real \alpha_i, \Imag \alpha_i, \Real \beta_i, \Imag \beta_i\}_{i=1}^r$$ with respect to smooth coordinates on $\M'$ is non-vanishing. Correspondingly, $ \dim_{\CC}\ker \Omega =2r$ on $O$ since it contains the linearly independent vectors 
\be  
\Big\{
\frac{\partial}{\partial \overline{\alpha_i}}, \frac{\partial}{\partial \overline{\beta_i}} 
\Big\}_{i=1}^r.
\ee

 We compute that 
\begin{align}
\Omega^r \wedge \overline{\Omega}^r &=r!r! (d\alpha_1  \wedge d\beta_1) \wedge \cdots \wedge  (d\alpha_r \wedge d\beta_r)  \wedge (d\bar \alpha_1 \wedge d\bar \beta_1)  \wedge \cdots \wedge (d\bar\alpha_r \wedge d\bar\beta_r) \nonumber \\
&= r!r!(d\alpha_1  \wedge d\bar \alpha_1) \wedge \cdots \wedge  (d\alpha_r \wedge \bar d\alpha_r)  \wedge (d\beta_1 \wedge d\bar \beta_1)  \wedge \cdots  \wedge (d\beta_r \wedge d\bar\beta_r) \nonumber\\
&= (-2i)^{2r} r!r!(\Real d\alpha_1  \wedge \Imag d\alpha_1) \wedge \cdots \wedge  (\Real d\alpha_r \wedge \Imag d\alpha_r) \nonumber \\
& \qquad    \wedge (\Real d\beta_1 \wedge \Imag d\beta_1) \wedge \cdots  \wedge (\Real d\beta_r \wedge \Imag d\beta_r)
\end{align}
Thus, we see that this is non-vanishing if the Jacobian determinant above is non-vanishing.
\end{proof}

In the cases in this paper it is straightforward to prove that $\varpi(\zeta), \omega_+$ is a holomorphic symplectic on $\bar{O}$ by first showing that $\varpi(\zeta), \omega_+$ are smooth when extended to $\bar{O}$ and then showing that $\varpi(\zeta)^r \wedge \overline{\varpi(\zeta)}^r$ and $\omega_+^r \wedge \bar \omega_+^r$ are non-vanishing using the limit of the Jacobian with respect to the smooth coordinates on $\M$. (See the proof of Theorem \ref{thm:modelgeom}.)

\bigskip

\noindent\textbf{Fixing $R$.} The Gaiotto--Moore--Neitzke formalism \cite{GMN:walls} has a parameter $R>0$, a negative power of which governs the area of a torus fiber in the manifold of interest. The semi-flat limit is then described as the $R\to\infty$ limit. However, $R$ is a dimensionful quantity, whereas mathematics must be conducted with dimensionless ratios. (Skeptical readers are invited to consider the expression ``$\sin(7\, {\rm meters})$".) And, indeed, this is reflected in the fact that $R$ only ever appears multiplying other quantities, such as the central charge, in such a way that one may rescale all of them in order to set $R$ equal to any value of interest. We choose to use this freedom to set $R=1/\pi$, as this simplifies many expressions. (Physicists would say that we work in units where $R=1/\pi$.) Fixing $R$ is related to our earlier discussion about treating $I_N$ and its perturbations $I_{N_1}+ \cdots I_{N_k}$ uniformly. This helps clarify that what is relevant for the model geometry near a point $u \in \B$ is not the area of the torus fiber, but rather the ratio of (an appropriate) power of this area to a measure of the distance to the nearest ``conflicting'' singular fiber. 

\subsubsection*{Outline/Key Results}

An outline of the remainder of this paper is as follows. 
\begin{itemize} 
\item In \S\ref{sec:first}, we explain all of the data in that is required as input 
for the construction of hyper-K\"ahler structures in Gaiotto--Moore--Neitzke's formalism and construct from this data many additional 
ingredients that will be needed. 
\begin{itemize}
\item [(\S\ref{sec:semiflat})] The basic data includes a local system of 
lattices $0 \to \Gamma_f \to \widehat{\Gamma} \to \Gamma \to 0$  over the regular locus of the Coulomb branch 
$\B'=\B\backslash \B''$, a connected complex manifold of complex dimension $r$, 
equipped with an integer-valued antisymmetric pairing $\avg{, }$  and 
map $Z: \widehat{\Gamma} \to \CC$ that is constant on $\Gamma_f$.  
 In Theorem \ref{prop:semiflathyperkahler} we construct a hyper-K\"ahler structure $\omega^{\semif}(\zeta)$ on $\mathcal{M}'$.

  An additional piece of data is a map 
 $\Omega: \widehat{\Gamma} \to \mathrm{Q}$, though this is not important for the semi-flat geometry. 
 \item[(\S\ref{sec:strategy})] Lastly, we are in the position to explain Gaiotto--Moore--Neitzke's proposal for constructing a hyper-K\"ahler structure on $\M'$ that extends smoothly over $\B''$ to a complete hyper-K\"ahler structure. This will serve as our guide in future papers, though we will make some modifications.
\end{itemize}
\item The formalism of Gaiotto, Moore, and Neitzke is most interesting near the singular locus $\B''$. In  \S\ref{sec:ass}, we introduce additional qualitative and quantitive assumptions near $\B''$.   These assumptions will imply that the local model geometries are multi-Ooguri-Vafa like.  As mentioned before,the non-degeneracy of $\varpi(\zeta)$ is the biggest problem. One of these assumptions, \ref{it:vPosAss}, and accompanying Figure \ref{fig:setup} is particularly novel. With this assumption, we are able to not just say "there is \emph{some} neighborhood $U \subset \B$ of $u \in \B''$ over which $\M_U$ admits a model hyper-K\"ahler metric", but to actually quantify how big it is. These preliminary estimates are essential for follow-up papers in this series. 
\item In \S\ref{sec:smooth}, we extend $\M'$ over $\B'$ to a smooth manifold $\M$. This is non-trivial. We do this in two step, first as $\widetilde{\M'}$ and then as $\M$. For Ooguri-Vafa, these two steps are shown in Figure \ref{fig:manifolds}. 
\begin{figure}[h!]
\begin{centering} 
\includegraphics[height=1.2in]{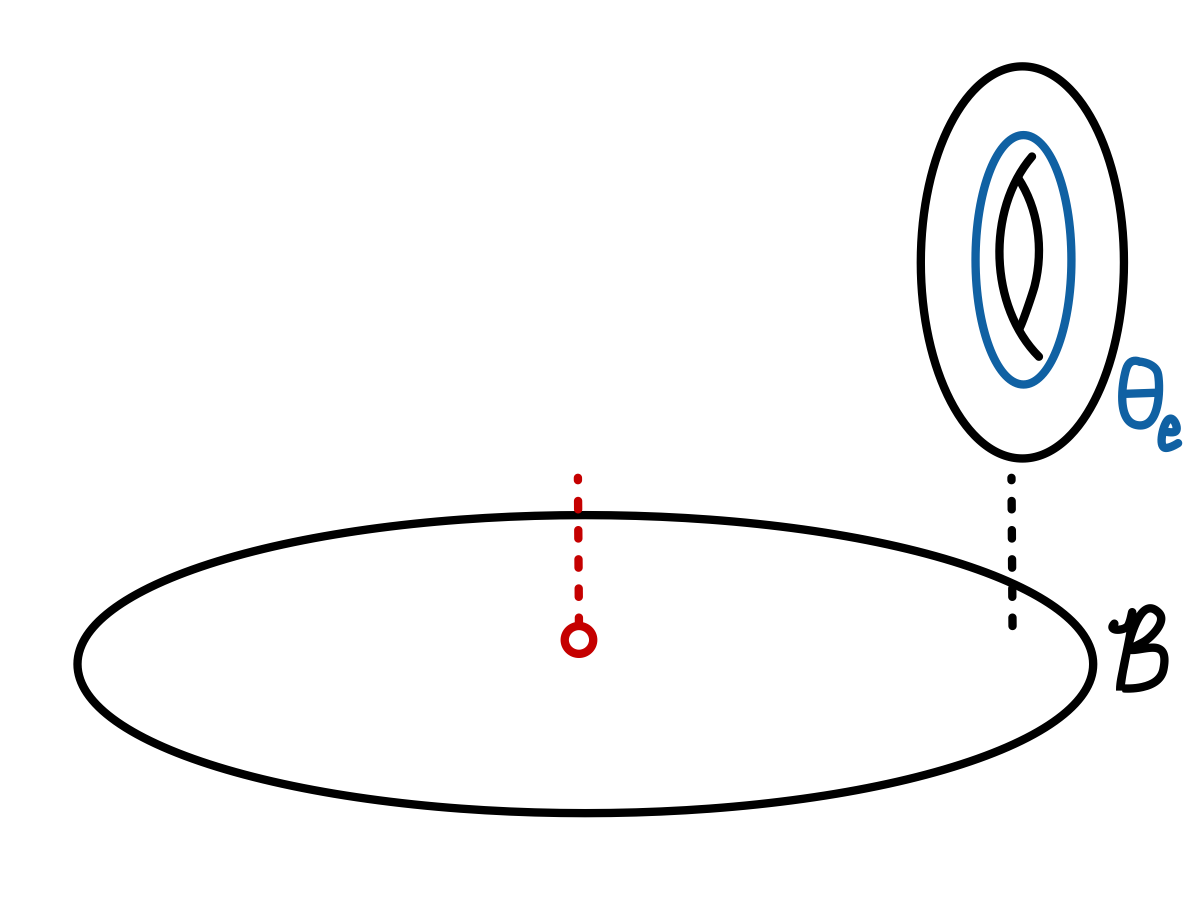} 
\includegraphics[height=1.2in]{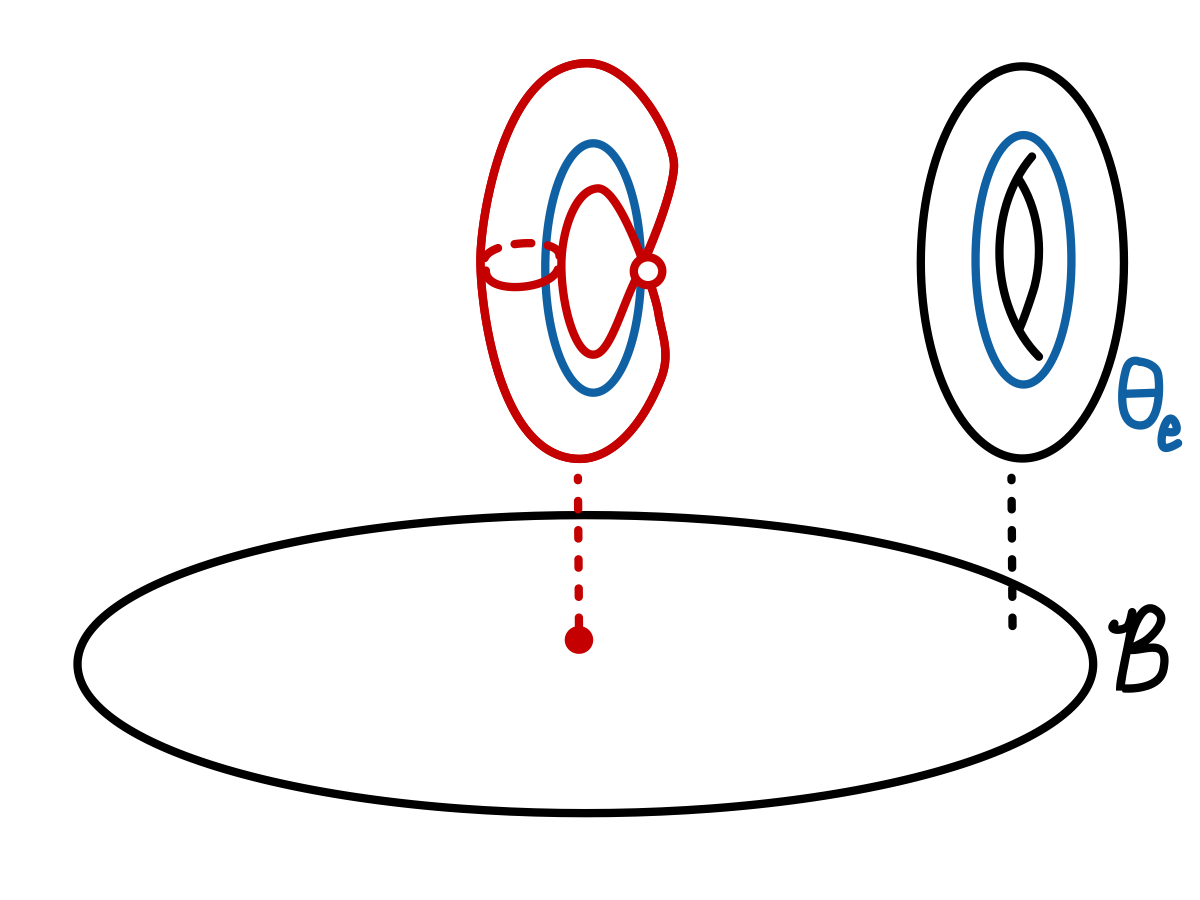} 
\includegraphics[height=1.2in]{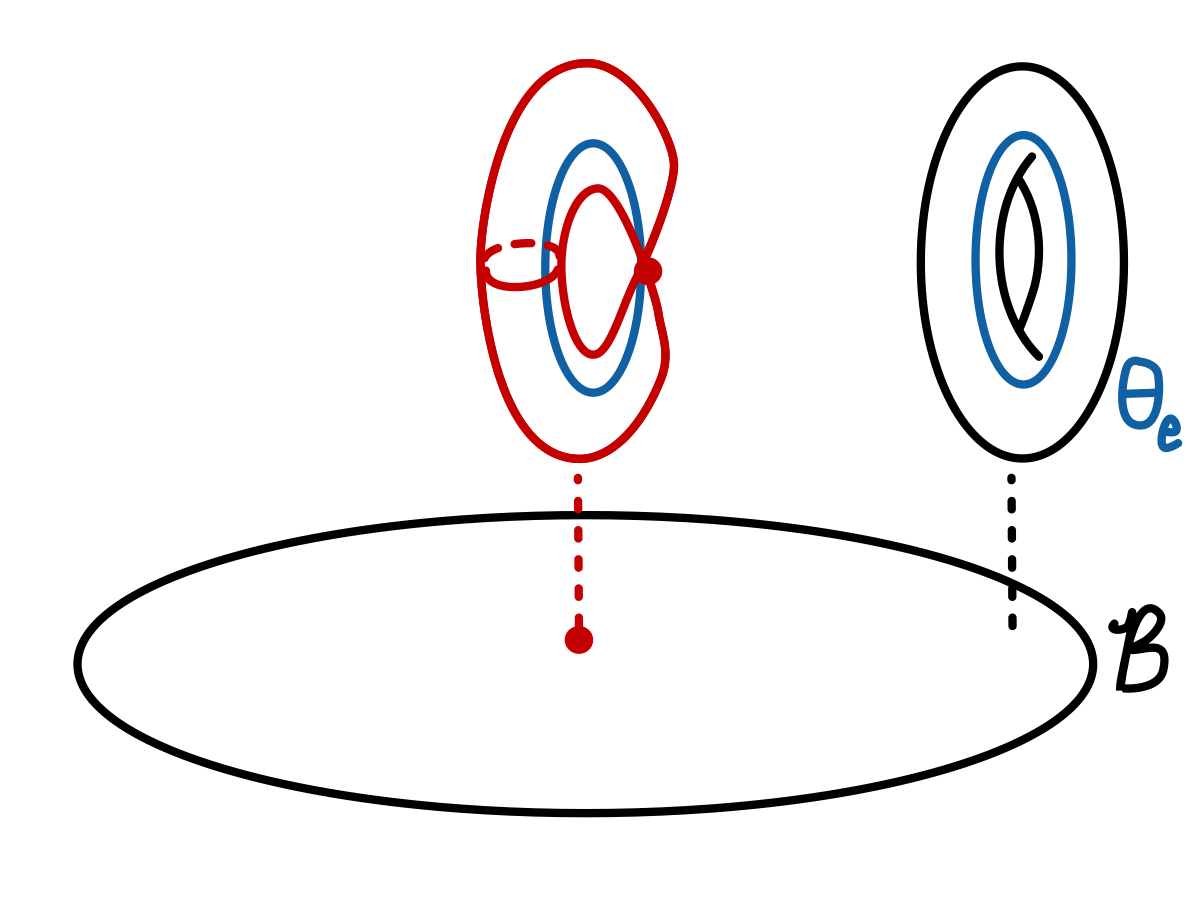}
\caption{\label{fig:manifolds} For Ooguri-Vafa, here is an image showing the difference of:\\
\textsc{(Left)} $\mathcal{M}'$, in which the singular fiber is missing;\\
\textsc{(Center)} $\widetilde{\mathcal{M}'}$, in which one point of the singular fiber over $0$ is missing;\\ \textsc{(Right)} $\mathcal{M}$, which is smooth and complete in a neighborhood of the singular fiber.}
\end{centering}
\end{figure}
This depends on some technical results developed in \S\ref{sec:GH}, in which we describe these model geometries in the higher-dimensional analogue of the Gibbons-Hawking Ansatz. 
\item In \S\ref{sec:modelgeom}, we describe the model hyper-K\"ahler Ooguri-Vafa-like geometries, culminating in the main theorem of the paper, Theorem \ref{thm:modelgeom}.
\end{itemize}

\subsubsection*{Acknowledgements.} We thank A. Tripathy for his collaboration at early stages in this series of papers, N. Addington, D. Allegretti and A. Neitzke for helpful and enjoyable conversations on related matters, and S. Kachru, R. Mazzeo, and A. Vasy for many enjoyable conversations and their enduring confidence that this project would reach completion.  
Laura is partially supported by NSF grant DMS-2005258.
This material is based upon work supported by the National Science
Foundation under Grant No. DMS-1928930, while the authors were in
residence at the Simons Laufer Mathematical Sciences Institute
(formerly MSRI) in Berkeley, California, during the Fall 2022
semester.

\section{Set up} \label{sec:first}

In this section, we introduce the inputs for Gaiotto--Moore--Neitzke's conjectural construction, as well as a number of useful notions which may be defined using these inputs. 
\begin{itemize}
\item 
In \S\ref{sec:semiflat}, we introduce Data \ref{item:D1}-\ref{item:thetaf} $(\mathcal{B}, \widehat{\Gamma}, \avg{,},  Z, \theta_f)$.  We then construct a smooth manifold $\mathcal{M}' \to \mathcal{B}'$ from this data in Construction \ref{constr:M'}. The main result is the construction of a hyper-K\"ahler structure $\omega^{\semif}(\zeta)$ on $\mathcal{M}'$ in Theorem \ref{prop:semiflathyperkahler}.
\item 
In \S\ref{sec:semiflat}, we also introduce the final Datum \ref{item:Omega} of a function $\Omega: \widehat{\Gamma} \to \mathbb{Q}$.
\item Lastly, in \S\ref{sec:strategy}, we describe Gaiotto--Moore--Neitzke's approach in \cite{GMN:walls} to construct a hyper-K\"ahler structure from this data. This will serve as our guide in follow up papers in this series, though will make some modifications in future papers in order to carefully state and prove our results. In the present paper, we follow their approach on the nose (up to rotating the ray we integrate on) in order to construct simple model hyper-K\"ahler geometries.
\end{itemize}
To the reader familiar with Gaiotto--Moore--Neitzke's construction, \footnote{For a comparison with \cite{neitzke:hkReview}, our (D1) is his ``Data 1'' and ``Data 2'';(D2), (D3) is ``Data 3abc''; (D4) is  ``Data 4''  together with ``Condition 1'', ``Condition 2'',  ``Condition 4.'' (We omit ``Condition 3'' since it is implied by ``Condition 4''); (D5) is ``Data 5''; (D6) is ``Data 6'' together with ``Condition 5''. We don't discuss the additional conditions on $\Omega$ in this paper. 
}.
we highlight the new:
\begin{itemize}
\item Throughout, we do not assume that the antisymmetric $\ZZ$-valued pairing  $\avg{,}$ on $\widehat{\Gamma}$ induces a unimodular pairing on $\Gamma$; consequently, rather than a symplectic basis, we must work with a Frobenius basis (see Definition  \ref{def:Frobeniusbasis} of Appendix \ref{sec:lattices}); this causes moderate headaches throughout.  
\item We give a direct proof that $\PP^1$-family of $\varpi^{\mathrm{sf}}(\zeta)$ of closed $2$-forms give rise to a pseudo-hyper-K\"ahler structure; as discussed in Corollary \ref{cor:twistor}, this ammounts to showing the non-degeneracy of the holomorphic Darboux coordinates for $\varpi^{\rm sf}(\zeta)$ for all $\zeta \in \CC^\times$ and for $\omega_+^{\rm sf}$. (Of course, the semi-flat geometry is already described in \cite{Freed}.)

\end{itemize}

\subsection{Basic data for semi-flat geometry}\label{sec:semiflat}

Our construction of hyper-K\"ahler manifolds begins with so-called `semi-flat' manifolds. We begin by explaining our construction thereof.

\subsubsection{From basic data to the semi-flat manifold $\mathcal{M}'$} \label{sec:constructSF}

In this subsection, we describe the fixed data (D1-D5) and the resulting real manifold $\mathcal{M}'$.

\bigskip

Fix the following data:
\begin{enumerate}[(D1)]
    \item \label{item:D1}$\mathcal{B}$ a connected complex manifold of dimension $\mathrm{dim}_{\mathbb{C}}\mathcal{B}=r$, a divisor $\B''$ called the \emph{singular} or \emph{discriminant} locus, and its complement $\B' = \B\setminus \B''$ called the \emph{regular} or \emph{smooth} locus. We will sometimes call $\mathcal{B}$ the \emph{Coulomb branch}.
    \item An exact sequence of local systems of lattices over $\mathcal{B}'$
    \[0 \rightarrow \Gamma_f \rightarrow \widehat{\Gamma} \rightarrow \Gamma \to 0 \ .\]
The lattice $\Gamma_f$ is a trivial fibration, and we refer to a fiber thereof (which we also denote by $\Gamma_f$) as the \emph{flavor charge lattice}. We refer to the fiber of $\Gamma$ over $u\in \B'$ as $\Gamma_u$; this lattice, which we term the \emph{gauge charge lattice}, has dimension $2r$. Similarly, we refer to the fiber of $\widehat\Gamma$ over $u$ as $\widehat\Gamma_u$, which we call the \emph{charge lattice}. Points in any of these lattices may be referred to as \emph{charges}. We may write expressions such as $\gamma\in \Gamma$ to indicate that $\gamma$ is a flat local section of $\Gamma$.

    \item A flat antisymmetric pairing 
    \[\avg{\; , \;}: \widehat{\Gamma}_u \times \widehat{\Gamma}_u \to \mathbb{Z} \]
such that for all $u$ we have $\Gamma_f = \{\gamma \in \widehat{\Gamma}_u: \avg{\gamma, \gamma'}=0 \mbox{ for all $\gamma' \in \widehat{\Gamma}_u$}\}$.
\end{enumerate}

By choosing arbitrary lifts from $\Gamma_u$ to $\hat\Gamma_u$, this induces a well-defined symplectic pairing on $\Gamma_u$. A good basis for $\Gamma_u$ is a Frobenius\footnote{In physically-realizable situations, the symplectic pairing is always unimodular (i.e. $p_i=1$), so that the inverse pairing is integral \cite{s:lines}. However, this is not important for the formalism of this series of papers, and it is convenient to omit this assumption in order to be able to construct moduli spaces of $PU(2)$ Higgs bundle with regular singularities, since the fibers of the Hitchin fibration in this case are polarized, but need not be principally polarized. In contrast, we are generally unable to construct moduli spaces of $SU(2)$ Higgs bundle with regular singularities because the fibers in this case need not be polarized. See the discussions in Footnote \ref{ft:polarization}, as well as \S\S6.3, 7.3.1, 7.5, and 8.4 of \cite{GMN:framed}, for more details. We stress that there is no real loss in requiring the pairing to be integral and unimodular if one wishes --- the issues we are alluding to in this footnote only lead to quotients by finite groups or finite covers. Also, imposing this requirement yields the existence of a symplectic basis for $\Gamma_u$, which is often convenient.}   basis (see Definition \ref{def:Frobeniusbasis}) of $\Gamma_u$, i.e. a basis $\gamma_{e_i}, \gamma_{m_i}$ with $\avg{\gamma_{m_i},\gamma_{e_j}}=p_i\delta_{ij}$, $\avg{\gamma_{m_i},\gamma_{m_j}}=\avg{\gamma_{e_i},\gamma_{e_j}}=0$ such that the elementary divisors $p_i$ satisfy $p_1|p_2| \cdots |p_r$. Note that since $p_1, \cdots, p_r: \mathcal{B}' \to \ZZ$ are continuous, they are locally constant.  
The dual lattice $\Gamma_u^*$ inherits a symplectic pairing. If $\gamma^*_{m_i},\gamma^*_{e_i}$ denote the corresponding dual vectors, then (see Lemma \ref{lem:dualsymp} and the paragraph following), we have $\avg{\gamma^*_{m_i},\gamma^*_{e_j}}=p_i^{-1} \delta_{ij}$, $\avg{\gamma^*_{m_i},\gamma^*_{m_j}}=\avg{\gamma^*_{e_i},\gamma^*_{e_j}}=0$. Note that if one rescaled the symplectic pairing on $\Gamma_u^*$ by $p_r$, one would get an integral symplectic pairing with dual elementary divisors $1| \cdots |p_{i}^{-1} p_r| \cdots |p_1^{-1}p_r$.

\begin{enumerate}[resume*]
    \item \label{item:Z} A family of homomorphisms $Z(u): \widehat{\Gamma}_u \to \mathbb{C}$ whose restriction to $\Gamma_f$ is constant and such that $Z_\gamma$ is holomorphic for each $\gamma\in \hat\Gamma$. Regarding $dZ$ and $d\bar Z$ as $\Gamma^*$-valued 1-forms, using the constancy of $Z|_{\Gamma_f}$, we also require that $\avg{dZ\wedge dZ}=0$ and that $\avg{dZ\wedge d\bar Z}$ be a K\"ahler form on $\B'$.

    \underline{Notation}: $Z_\gamma(u)$ is the \emph{central charge} of $\gamma$ at $u$.
    \item \label{item:thetaf} A homomorphism $\theta_f: \Gamma_f \to \mathbb{R}/2 \pi \mathbb{Z}$.

    \item \label{item:Omega} A map $\Omega: \widehat{\Gamma} \to \mathbb{Q}$ satisfying $\Omega(\gamma,u)=\Omega(-\gamma,u)$ for all $\gamma,u$,  as well as the \emph{support condition}, \emph{Wall Crossing Formula}, and \emph{growth conditions}. We defer the discussion of these three conditions to the follow-up paper. 

\underline{Notation}: We denote the values of this map as $\Omega(\gamma, u)$, where $u \in \mathcal{B}'$ and $\gamma \in \widehat{\Gamma}_u$. We refer to these values as \emph{BPS state counts}, \emph{BPS indices}, or \emph{BPS invariants}. The former terminology is motivated by the fact that for a given $\gamma$, $\Omega(\gamma,u)$ is expected to generically (i.e., away from real codimension 1 walls at which it is discontinuous) be a locally constant integer. Even then, it can be a negative integer, and so the `index' terminology is somewhat more precise. We stressed `for a given $\gamma$' because walls can be dense in $\B'$, and so the overall map $\Omega(\cdot,u)$ --- which we term the \emph{BPS spectrum} --- can be discontinuous or have rational values at dense subsets of $\B'$. Anyways, while these are the general expectations, this integrality will play no role in our formalism, and so we do not assume it. (In fact, even rationality, as opposed to reality, will not be important, but we assume it because $\Omega(\gamma,u)$ is always expected to be rational.)

We term a charge $\gamma\in \hat\Gamma_u$ to be \emph{active} if $\Omega(\gamma,u)\not=0$ and \emph{inactive} otherwise.
\end{enumerate}

At the expense of making some arbitrary choices, we can make (D4) a bit more concrete.

    \begin{lemma}[\ref{item:Z} revisited] \label{lem:D4}
    At some $u\in \B'$, let $\tilde\gamma_{m_i},\tilde\gamma_{e_i}$ be a Frobenius basis of $\Gamma_u$, and let $\gamma_{m_i},\gamma_{e_i}$ be a lift thereof to $\hat\Gamma_u$. Extend these to local sections of $\hat\Gamma$. Define $a_i = Z_{\gamma_{e_i}}$ and $a^D_i = p_i^{-1} Z_{\gamma_{m_i}}$. Then,
\begin{enumerate}[(i)]\item 
both $\{da_i\}$ and $\{da_i^D\}$ span $T^*_{(1,0)}\B'$;
\item locally, we may take the $a_i$ to be our complex coordinates on $\B'$;
\item 
$\avg{dZ\wedge dZ}=0$ is equivalent to the statement that,
locally, there is a holomorphic function $\F(\{a_i\})$, called the prepotential, with $a_i^D = \partial_{a_i} \F$;
\item \label{it:Imtau}
in terms of the prepotential, the K\"ahler form is
\be \label{eq:Kahlersemif}\avg{dZ\wedge d\bar Z}=
2i \sum_{i,j} \Imag(\underbrace{\partial_{a_i} \partial_{a_j}\F}_{\tau_{ij}}) da_i \wedge d\bar a_j ;  \ee 
the symmetric matrix \begin{equation}\label{eq:tau}\tau_{ij} = \partial_{a_i} \partial_{a_j} \mathcal{F}\end{equation} has a positive-definite imaginary part;
an associated K\"ahler potential is
\be K = 4 \Imag \sum_i \bar a_i \partial_{a_i} \F \ . \ee
\end{enumerate}
    \end{lemma}

\begin{rem}[Period matrix $\tau$]\label{rem:periodmatrix}
    In the semi-flat geometry we are about to describe, the symmetric matrix $\tau_{ij}$
    will play the role of a period matrix of a complex torus (see Lemma \ref{lem:semiflat}).
\end{rem}

    \begin{proof}\hfill
\begin{enumerate}[(i):]
   \item The positive-definiteness of the K\"ahler metric $g_{\B'}$ on $\B'$ implies that both $\{da_i\}$ and $\{da_i^D\}$ span $T^*_{(1,0)}\B'$.  Suppose there is a $v\in T\B'$ such that $da_i(v)=0$ for all $i$ then $g_{\B'}(v,v)=0$ shows that $v=0$ (and similarly for $da_i^D$).  Hence,  $\{da_i\}$  (and similarly for $da_i^D$) spans $T^*_{(1,0)}\B'$.
    \item (ii)  follows from (i)
\item We compute
\begin{align} 0=\avg{dZ\wedge dZ} &= 2 \sum_{i=1}^r dZ_{\gamma_{e_i}} \wedge dZ_{\gamma_{m_i}} \avg{\gamma_{e_i}^*, \gamma_{m_i}^*} \nonumber \\
&= 2 \sum_{i=1}^r da_i \wedge p_i da_i^D \cdot -p_i^{-1}\nonumber\\
&=-2\sum_{i=1}^r da_i \wedge da_i^D  \nonumber \\
&= -2\sum_{i,j} \frac{\partial a_i^D}{\partial a_j} da_i\wedge da_j \nonumber\\&= 2 d \sum_i a_i^D da_i.\end{align}
Locally, the closed $1$-form $\sum_i a_i^D da_i$ is exact, hence it is equal to $d\mathcal{F}$. 
\item We compute that $\avg{d Z \wedge d \overline{Z}}$ is
\be \avg{dZ\wedge d\bar Z} = -\Real \sum_i da_i\wedge d\bar a^D_i = -\Real \sum_{i,j} \overline{\partial_{a_i}\partial_{a_j}\F} da_i\wedge d\bar a_j = 2i \sum_{i,j} \Imag(\partial_{a_i} \partial_{a_j}\F) da_i \wedge d\bar a_j \ . \ee
It is automatically closed.
The associated Riemannian metric is positive-definite if, and only,
if the imaginary part of $\tau$ is positive-definite. One can check that function $\mathcal{F}$ satisfies $\avg{d Z \wedge d \overline{Z}}= \frac{i}{2} \partial \overline{\partial} \mathcal{F}$.
\end{enumerate}
\end{proof}
Finally, we note that the above lemma implies the following global results:
\begin{corollary}[global consequence of \ref{item:Z}] \label{cor:global} \hfill
\begin{enumerate}[(i)]
\item $dZ(u): 
(T^{(1,0)}\B')_u\to \Gamma_u^*\otimes\CC$ is injective
\item  $dZ(u): \Gamma_u\otimes\CC\to (T^*_{(1,0)}\B')_u$ is surjective. 
\end{enumerate}
 Introduce Lagrangian decompositions $\Gamma_u\otimes \CC \cong \Lambda_m\oplus \Lambda_e$ and $\Gamma_u^*\otimes \CC \cong \Lambda_m^* \oplus \Lambda_e^*$ and let $\pi_e: \Gamma_u^*\otimes\CC\to \Lambda_e^*$ be the projection with kernel $\Lambda_m^*$. Then, 

\begin{enumerate}
\item[(iii)] $\pi_e\circ dZ(u): (T^{(1,0)}\B')_u\to \Lambda_e^*$ is bijective 
\item[(iv)] $dZ(u): \Lambda_e \to (T^*_{(1,0)}\B')_u$ is bijective
\end{enumerate}
\end{corollary}

\bigskip

Both to illustrate some consequences of these definitions and because we will use this in a follow up paper, we now prove the following lemma:
\begin{lemma} \label{lem:coorRat}
If $u\in \B'$ and $\gamma,\gamma'\in \hat\Gamma_u$ are linearly independent then the following four conditions cannot simultaneously hold:
\be Z_\gamma(u)\not=0 \ , \quad Z_{\gamma'}(u)\not=0 \ , \quad \Imag(Z_{\gamma'}/Z_\gamma)|_u = 0 \ , \quad d \arg(Z_{\gamma'}/Z_\gamma)|_u = 0 \ . \ee
\end{lemma}
\begin{proof}
Suppose, toward a contradiction, that all of these conditions hold. 

Let $f=Z_{\gamma'}/Z_\gamma$. Then, $\arg(f)=\half\log(f/\bar f)$ gives $d\arg(f)=\frac{\partial f}{2f}-\frac{\bar\partial f}{2f}$, which shows that the last condition is equivalent to $d(Z_{\gamma'}/Z_\gamma)|_u=0$, i.e. \be dZ_{\gamma'}|_u = \frac{Z_{\gamma'}}{Z_\gamma} dZ_\gamma |_u.\ee

Let $\tilde\gamma,\tilde\gamma'$ be the respective images of $\gamma,\gamma'$ under $\hat\Gamma_u\to \Gamma_u$. Let $\tilde\gamma_{e_1},\ldots,\tilde\gamma_{e_r},\tilde\gamma_{m_1},\ldots,\tilde\gamma_{m_r}$ be a Frobenius basis for $\Gamma_u\otimes\RR$ such that $\tilde\gamma_{e_1}=\tilde\gamma$, $\avg{\tilde\gamma_{e_i},\tilde\gamma_{e_j}}=\avg{\tilde\gamma_{m_i},\tilde\gamma_{e_j}}=0$, and $\avg{\tilde\gamma_{m_i},\tilde\gamma_{e_j}}=p_i\delta_{ij}$. If $\avg{\tilde\gamma,\tilde\gamma'}=0$, choose $\tilde\gamma_{e_2} = \tilde\gamma'$. Finally, choose lifts $\gamma_{e_i},\gamma_{m_i}$ of $\tilde\gamma_{e_i},\tilde\gamma_{m_i}$ to $\hat\Gamma_u$ such that if any of the $\tilde\gamma_{e_i}$ is either of $\tilde\gamma,\tilde\gamma'$ then its lift is, respectively, $\gamma$ or $\gamma'$. We choose $a_i=Z_{\gamma_{e_i}}$ to be our local holomorphic coordinates near $u$. 

Then, the last two conditions imply that $\partial_{a_1} Z_{\gamma'}|_u = \frac{Z_{\gamma'}(u)}{Z_\gamma(u)} \in \RR$ and $\partial_{a_j} Z_{\gamma'}|_u = 0$ for $j>1$. This is clearly a contradiction if $\gamma_{e_2}=\gamma'$, so we must have $\avg{\gamma,\gamma'}\not=0$. In this case, write $\tilde\gamma'=\sum_i(c_i \tilde\gamma_{e_i}+d_i \tilde\gamma_{m_i})$, and note $d_1 \neq 0$.  Then, define a non-zero vector $\gamma'' = \sum_i d_i \gamma_{m_i}$ and note that $\Imag(\partial_{a_i} Z_{\gamma''}|_u)=0$. That is, the matrix $\Imag(\partial_{a_i} Z_{\gamma_{m_j}})= \Imag(\tau_{ij})$ has $(d_1,\ldots,d_r)^T$ in its kernel, which contradicts the fact that $\Imag(\tau_{ij})$ is positive-definite.
\end{proof}

\bigskip

We will now describe $\mathcal{M}'$.

\begin{defn}
A \emph{twisted character} on $\hat\Gamma_u$ is a map $\phi: \hat\Gamma_u\to \CC^\times$ with $$\phi(\gamma+\gamma') = (-1)^{\avg{\gamma,\gamma'}} \phi(\gamma) \phi(\gamma');$$ a \emph{twisted unitary character} is a twisted character valued in the unit circle. 
    We will often represent such a twisted unitary character in the form $\phi(\gamma) = e^{i \theta_\gamma}$, where $\theta: \hat\Gamma_u\to \RR/2\pi\ZZ$ satisfies $\theta_{\gamma+\gamma'} = \theta_\gamma + \theta_{\gamma'} + \pi \avg{\gamma,\gamma'}$.
\end{defn}

\begin{constr}[$\mathcal{M}'$] Given the data (D1-D5), we construct a local system of tori $\pi: \M'\to \B'$ whose fiber $\M'_u$ over $u$ is the space of twisted unitary characters $\theta: \hat\Gamma_u\to \RR/2\pi\ZZ$ whose restriction to $\Gamma_f$ is $\theta_f$ in \ref{item:thetaf}.
\label{constr:M'}
\end{constr}

\begin{rem}[Parameterization]\label{rem:fiberparam}
We equip $\M'$ with a smooth structure by declaring these coordinates to be smooth.  
Each fiber may be conveniently parametrized by choosing a basis $\tilde\gamma_i$ of $\Gamma_u$ and a lift $\gamma_i$ thereof to $\hat\Gamma_u$; then, the values $\theta_{\gamma_i}$ provide coordinates on the fiber over $u$ and manifest the fact that this fiber is a $2r$-torus.
\end{rem}

We note that $d\theta$ may be regarded as a $\Gamma^*$-valued 1-form on $\M'$, by the constancy of $\theta|_{\Gamma_f}$ (and the fact that the constants $\pi \avg{\gamma,\gamma'}$ are annihilated by the exterior derivative).

\begin{rem}[Untwisted Unitary Characters] \label{rem:untwisted}
It is sometimes preferable to work with untwisted unitary characters instead of twisted ones. At least locally, this is always possible. To do so, we choose\footnote{In particular, pick a Frobenius basis  $\tilde\gamma_{m_i}, \tilde\gamma_{e_i}$ of $\Gamma_u$ where $p_i=\avg{\tilde\gamma_{m_i}, \tilde\gamma_{e_i}}$. Then, a flat quadratic refinement is determined by the values of $\sigma(\tilde\gamma_{m_i}), \sigma(\tilde\gamma_{e_i})$ which are unconstrained. Then the value of 
\[\sigma(\sum_i (a_i \tilde\gamma_{e_i} + b_i \tilde\gamma_{m_i}))= \prod_{i=1}^r (-1)^{a_i b_i p_i} \sigma(\tilde\gamma_{e_i})^{a_i} \sigma(\tilde\gamma_{m_i})^{b_i}.\]
This can be extended to a neighborhood $U \ni u$ over which $\Gamma$ is trivial.

} a flat quadratic refinement $\sigma: \Gamma_u\to \{\pm 1\}$ of the pairing $(-1)^{\avg{\gamma,\gamma'}}$, i.e. a map obeying
\be \sigma(\gamma)\sigma(\gamma') = (-1)^{\avg{\gamma,\gamma'}} \sigma(\gamma+\gamma') \ . \ee

We can then map a twisted character $e^{i\theta_\gamma}$ to an untwisted one $e^{i\tilde\theta_\gamma}$ via
\be e^{i\theta_\gamma} = e^{i\tilde\theta_\gamma} \sigma(\gamma) \ . \label{eq:untwist} \ee
(Since $\gamma\in\hat\Gamma_u$ here, $\sigma(\gamma)$ is defined by first using the map $\hat\Gamma_u\to \Gamma_u$.) 
However, globally, this may not be invariant under the symplectic monodromies of $\Gamma_u$.

We note that twisted characters and untwisted characters both provide coordinates parameterizing the fibers $\M'_u$. Since $\theta_\gamma = \widetilde{\theta}_\gamma - i \log \sigma(\gamma)$ formally, the partial derivative operators are equal:
\begin{equation}\frac{\partial}{\partial \theta_\gamma} = \frac{\partial}{\partial \widetilde{\theta}_\gamma}.\end{equation}

There are some circumstances in which a global quadratic refinement exists (and so, in particular, there is a global $\tilde\theta=0$ section when $\theta_f=0$). Gaiotto, Moore, and Neitzke explain in \cite{GMN:classS} that this is the case for $SO(3)$ Hitchin systems. It was noted in \cite{mz:K3HK3} that this is also the case for a class of examples with $r=1$. Here, we note that the same choice made therein actually works for all examples with $r=1$ so long as the monodromies of the $\Gamma$ local system are valued in $SL(2,\ZZ)$, so that the $\Gamma$ local system is contained in a local system of unimodular superlattices $\Lambda_u$. For some $u\in \B'$, fix such a $\Lambda_u$. Representing an element therein via a pair $(p,q)\in \ZZ^2$, where $\avg{(p,q),(p',q')} = pq'-qp'$, we define a quadratic refinement $\sigma((p,q)) = (-1)^{p-pq+q}$ (which we can then restrict to $\Gamma_u$ if we wish). We verify that this is invariant under all $SL(2,\ZZ)$ transformations $\twoMatrix{a}{b}{c}{d}$ of $\Lambda_u$:
\be (-1)^{(ap+bq)-(ap+bq)(cp+dq)+(cp+dq)} = (-1)^{p(a+ac+c)+q(b+bd+d)-pq(ad-bc)} = (-1)^{p+q-pq} \ , \ee
where the last equality used $ad-bc=1$ and the fact that both $a+ac+c$ and $b+bd+d$ are odd, since $a$ and $c$ are not both even (and similarly for $b,d$). So, there is no obstruction to extending this to all of $\B'$. Furthermore, this choice is canonical, in the sense that it is the unique quadratic refinement which is invariant under all of $SL(2,\ZZ)$.

For the sake of clarity, we stress that quadratic refinements and untwisted characters play almost no role in the formalism of this series of papers; they are only necessary in some applications. The exception is that we employ them for some results in the present section. 
\end{rem}

\begin{rem}[Passing to a unimodular superlattice]\label{rem:unimodular}
Given a Frobenius basis $\gamma_{m_1}, \cdots, \gamma_{m_r}$, $\gamma_{e_1}, \cdots, \gamma_{e_r}$ of $(\Gamma_u, \avg{,})$, one can view $\Gamma_u$ as a sublattice of a unimodular lattice $(\hat\Gamma_u, \avg{,})$ 
generated by $\tilde\gamma_{m_i}=\frac{1}{p_i} \gamma_{m_i}, \gamma_{e_i}$. 
Then, $\tilde \M'_u$ is a $|\tilde \Gamma_u : \Gamma_u|= \prod_{i=1}^r p_i$-fold cover of $\M'_u$. 
While $\M'_u$ and $\tilde \M'_u$ are different, there is a natural identification of the tangent spaces $T_m \M'_u$ and $T_{\hat m} \M'_u$ where the twisted unitary character $\hat m$ on $\tilde \Gamma_u$ restricts to $m$ on $\Gamma_u$.  In many places we can effectively work with the tangent space to the (simpler) unimodular superlattice, so going forward, the formulas that will appear will often look as if we are working with $\tilde\gamma_{m_i}=\frac{1}{p_i}\gamma_{m_i}$ and $\gamma_{e_i}$.
\end{rem}
\subsubsection{Properties of semi-flat holomorphic Darboux coordinates}

\begin{definition}[$\mathcal{X}^{\mathrm{sf}}$]
    For each $\zeta\in\CC^\times$ and $m\in \M'$, we now define the twisted character
\be \X^{\rm sf}_\gamma(\zeta)(m) = \exp\parens{\frac{1}{\zeta} Z_\gamma(\pi(m)) + i \theta_\gamma(m) + \zeta \overline{Z_\gamma(\pi(m))}} \ . \label{eq:xSF} \ee
(Henceforth, we will generally just write $Z_\gamma(m)$.)
\end{definition}

 We then have:
\begin{lemma}[Nondegeneracy] \label{lem:nondegen}
 Let $\{\gamma_i\}$ be a lift to $\hat\Gamma$ of a basis of sections of $\Gamma$ over a contractible open set $U\subset \B'$. Then, for all $\zeta\in\CC^\times$, $\{\X^{\rm sf}_{\gamma_i}(\zeta)\}$ provide valid coordinates on $\pi^{-1}(U)\subset \mathcal{M'}$.
\end{lemma}
\begin{proof}

Since $U$ is contractible, we trivialize all of the local systems and omit most $u$ subscripts in this proof. We also pick a quadratic refinement $\sigma: \Gamma\to \{\pm 1\}$ and use it to define untwisted characters $e^{i\tilde\theta_\gamma} = e^{i\theta_\gamma} \sigma(\gamma)$ and
\be \tilde \X^{\rm sf}_\gamma(\zeta) = \X^{\rm sf}_\gamma(\zeta) \sigma(\gamma) = \exp\parens{\frac{1}{\zeta} Z_\gamma(\pi(m)) + i \tilde \theta_\gamma(m) + \zeta \overline{Z_\gamma(\pi(m))}} \ . \ee
Locally in $\pi^{-1}(U)$, we can lift $\tilde\theta$ to a homomorphism $\tilde\theta: \hat\Gamma\otimes\RR\to \RR$. Having done so, it suffices to check that $\tilde\Y^{\rm sf}_{\gamma_i}(\zeta):=\log \tilde\X^{\rm sf}_{\gamma_i}(\zeta)$ provide valid coordinates. Since $\tilde\Y^{\rm sf}(\zeta)$ is a homomorphism $\hat\Gamma\otimes\RR\to \CC$, this will be the case if and only if it holds for the lift to $\hat\Gamma\otimes\RR$ of any basis of $\Gamma\otimes\RR$. Similarly, $\tilde\theta_{\gamma_i}$ are good fiber coordinates, and hence the same is true for any lift to $\hat\Gamma\otimes\RR$ of a basis of $\Gamma\otimes\RR$. We hence replace $\gamma_i$ with a lift to $\hat\Gamma\otimes\RR$ of a  Frobenius basis $\gamma_{m_i},\gamma_{e_i}$ of $\Gamma\otimes\RR$. 

As above, we write $a_i = Z_{\gamma_{e_i}}$ and $a^D_i = p_i^{-1} Z_{\gamma_{m_i}}$ and take the $a_i$ to be our local coordinates on $U$. We then want to compute the Jacobian 
\be \abs{\frac{
\partial(\{\Imag \tilde\Y^{\rm sf}_{\gamma_{e_i}}\}_{i=1}^r,
\{\frac{1}{p_i}\Imag \tilde\Y^{\rm sf}_{\gamma_{m_i}}\}_{i=1}^r, 
\{ \Real \tilde\Y^{\rm sf}_{\gamma_{e_i}}\}_{i=1}^r,  
\{ \frac{1}{p_i}\Real \tilde\Y^{\rm sf}_{\gamma_{m_i}}\}_{i=1}^r
)}{\partial(\tilde\theta_{\gamma_{e_1}},\ldots,\tilde\theta_{\gamma_{e_r}},\frac{1}{p_1}\tilde\theta_{\gamma_{m_1}},\ldots,\frac{1}{p_r}\tilde\theta_{\gamma_{m_r}},a_1,\ldots,a_r,\bar a_1,\ldots,\bar a_r)}} \ . \ee
(As discussed in Remark \ref{rem:unimodular}, the presence of the $p_i$'s in the above formula is because we are effectively working with a chart on the unimodular superlattice.)
We note
\begin{align}
\Imag \tilde\Y^{\rm sf}_{\gamma_{e_i}} &= \frac{1}{2i}(\zeta^{-1} - \bar \zeta) a_i  + \tilde \theta_{\gamma_{e_i}} + \frac{1}{2i}(\zeta - \bar \zeta^{-1}) \bar a_i  \nonumber\\
\Real \tilde\Y^{\rm sf}_{\gamma_{e_i}} &= \frac{1}{2}(\zeta^{-1} +\bar \zeta) a_i  + \frac{1}{2}(\zeta + \bar \zeta^{-1}) \bar a_i  \nonumber\\
\frac{1}{p_i} \Imag \tilde\Y^{\rm sf}_{\gamma_{m_i}} &= \frac{1}{2i}(\zeta^{-1} - \bar \zeta) a_i^D + \frac{1}{p_i} \tilde\theta_{\gamma_{m_i}} + \frac{1}{2i}(\zeta - \bar \zeta^{-1})  \bar a_i^D \nonumber\\
\frac{1}{p_i} \Real \tilde \Y^{\rm sf}_{\gamma_{m_i}} &= \frac{1}{2}(\zeta^{-1} +\bar \zeta) a_i^D  + \frac{1}{2}(\zeta + \bar \zeta^{-1})  \bar a_i^D.
\end{align}
In block form, using the symmetric $r\times r$ matrix $\tau$ with positive-definite imaginary part introduced in \eqref{eq:tau} as well as the $r\times r$ identity and zero matrices $I_r$ and $0_r$, this Jacobian matrix is
\begin{align}
\begin{pmatrix}
I_{r} & 0_r & \frac{1}{2i}(\zeta^{-1}-\bar\zeta) I_{r} & \frac{1}{2i}(\zeta-\bar\zeta^{-1}) I_{r} \\
0_r & I_r & \frac{1}{2i} (\zeta^{-1}-\bar\zeta) \tau & \frac{1}{2i}(\zeta-\bar\zeta^{-1}) \bar{\tau} \\
0_r & 0_r & \half (\zeta^{-1}+\bar\zeta) I_r & \half(\zeta+\bar\zeta^{-1}) I_r \\
0_r & 0_r & \half (\zeta^{-1}+\bar\zeta) \tau & \half (\zeta+\bar\zeta^{-1}) \bar\tau
\end{pmatrix} \ . \label{eq:sfJacobianMat}
\end{align}
By twice employing the identity
\be \det\twoMatrix{A}{B}{C}{D} = \det A \det(D-CA^{-1}B) \label{eq:blockDet} \ee
for the determinant of a block matrix with $A$ invertible, and noting that $\zeta^{-1}+\bar\zeta$ is always non-zero, the Jacobian determinant is found to be
\be
\begin{vmatrix} \half (\zeta^{-1}+\bar\zeta) I_r & \half(\zeta+\bar\zeta^{-1}) I_r \\
\half (\zeta^{-1}+\bar\zeta) \tau & \half (\zeta+\bar\zeta^{-1}) \bar\tau \end{vmatrix} = \abs{\frac{\zeta^{-1}+\bar\zeta}{2}}^{2r} (-2i)^r \det \Imag \tau \ .
\ee
Since this is nonzero, the change of variables is nonsingular.
\end{proof}

Besides the nondegeneracy condition above, the $\X$'s satisfy the following reality condition:
\begin{lemma}[Reality condition]\label{lem:realitysf}
$\X^{\rm sf}_{\gamma(\zeta)}$ satisfy
\be \overline{\X^{\rm sf}_\gamma(\zeta)} = 1/\X^{\rm sf}_\gamma(-1/\bar \zeta) \ . \label{eq:xSFReality} \ee
\end{lemma}

\subsubsection{Semi-flat hyper-K\"ahler structure}

\begin{definition}[semi-flat holomorphic symplectic forms]
For every $\zeta\in\CC^\times$, we use the symplectic pairing on $\Gamma^*$ to define a closed 2-form
\begin{align}
\varpi^{\rm sf}(\zeta) &= \frac{1}{8\pi} \avg{d\log \X^{\rm sf}(\zeta) \wedge d\log \X^{\rm sf}(\zeta)} \nonumber \\
&= \frac{i}{4\pi \zeta} \avg{dZ \wedge d\theta} + \frac{1}{8\pi} (2 \avg{dZ\wedge d\bar Z} - \avg{d\theta\wedge d\theta}) + \frac{i \zeta}{4\pi} \avg{d\theta\wedge d\bar Z} \ .
\end{align}
\end{definition}
Note that we have used here the fact that $\avg{dZ\wedge dZ} = 0$ in order to eliminate the terms multiplying $\zeta^{\pm 2}$. 

\begin{thm}[semi-flat hyper-K\"ahler structure]\label{prop:semiflathyperkahler}
The $\CC^\times$-family of closed $2$-forms $\varpi^{\mathrm{sf}}(\zeta)$ 
produces a hyper-K\"ahler structure on $\M'$, called the semi-flat structure. 
\end{thm}

By construction, the coordinates $\log \mathcal{X}^{\rm sf}_{\gamma_{e_i}}$ and $p_i^{-1} \log \mathcal{X}^{\rm sf}_{\gamma_{m_i}}$ are holomorphic Darboux coordinates for $\varpi^{\rm sf}(\zeta)$. Before the proof of the above proposition, we similarly give
 holomorphic Darboux coordinates for $\omega^{\semif}_+$:

\begin{lem}\label{lem:omega+sf} The  closed $2$-form 
    \begin{align}\omega^{\rm sf}_+ &= -\frac{1}{2\pi} \avg{dZ\wedge d\theta}
        =\sum_{i=1}^r da_i \wedge dz_i
    \end{align}
    for 
    \be z_i = \frac{p_i^{-1} \tilde\theta_{\gamma_{m_i}}-\sum_{j=1}^r \tau_{ij} \tilde\theta_{\gamma_{e_j}}}{2\pi}.\ee
    Moreover, $a_i, z_i$ provide valid coordinates on $\pi^{-1}(U)$. 
\end{lem}

\begin{proof}[Proof of Lemma \ref{lem:omega+sf}]
In the notation of Lemma \ref{lem:nondegen}, we write 
\begin{align} \omega_+^{\rm sf} &= - \frac{1}{2 \pi} \avg{dZ \wedge d \theta} \nonumber \nonumber  \\
&=-\frac{1}{2\pi} \left(\sum_{i=1}^r \avg{\gamma_{e_i}^*, \gamma_{m_i}^*} dZ_{\gamma_{e_i}} \wedge d \widetilde{\theta}_{\gamma_{m_i}}  + 
\sum_{i=1}^r dZ_{\gamma_{m_i}} \wedge d \widetilde{\theta}_{\gamma_{e_i}} \avg{\gamma_{m_i}^*, \gamma_{e_i}^*} 
\right) \nonumber  \\
&=-\frac{1}{2\pi} \left(\sum_{i=1}^r -p_i^{-1} da_i \wedge d \widetilde{\theta}_{\gamma_{m_i}}  + 
\sum_{i=1}^r p_i d a_i^D \wedge d \widetilde{\theta}_{\gamma_{e_i}} p_i^{-1}
\right)\nonumber  \\
&=
\frac{1}{2\pi} \sum_{i=1}^r (p_i^{-1} da_i\wedge d\tilde\theta_{\gamma_{m_i}} - \sum_{j=1}^r \tau_{ij} da_i \wedge d\tilde\theta_{\gamma_{e_j}})\nonumber  \\
&
= \sum_{i=1}^r da_i \wedge d \parens{\frac{p_i^{-1}\tilde\theta_{\gamma_{m_i}}-\sum_{j=1}^r \tau_{ij} \tilde\theta_{\gamma_{e_j}}}{2\pi}} \ . \end{align}
We can bring $\tau_{ij}$ inside of the exterior derivative because
\be -\sum_{i=1}^r da_i \wedge d \tau_{ij} = \sum_{i,k=1}^r \partial_{a_k} \tau_{ij} da_i\wedge da_k = \sum_{i, k=1}^r \partial_{a_k} \partial_{a_i} \partial_{a_j} \F da_i \wedge da_k = 0 \ . \ee
Now, we want to check that $a_i$ and $z_i = \frac{p_i^{-1} \tilde\theta_{\gamma_{m_i}}-\sum_{j=1}^r \tau_{ij} \tilde\theta_{\gamma_{e_j}}}{2\pi}$ comprise a set of good local complex coordinates on $\M'$. The relevant Jacobian matrix this time is
\be
\frac{\partial(a_1,\ldots,a_r,\bar a_1,\ldots,\bar a_r, z_1, \ldots, z_r, \bar z_1, \ldots,\bar z_r)}{\partial(a_1,\ldots,a_r,\bar a_1,\ldots,\bar a_r,\frac{1}{p_1}\tilde\theta_{\gamma_{m_1}},\ldots,\frac{1}{p_r} \tilde\theta_{\gamma_{m_r}},\tilde\theta_{\gamma_{e_1}},\ldots,\tilde\theta_{\gamma_{e_r}})} \ . \ee
Using the same notation as before, as well as the notation $\partial_a^2 H$ to denote the Hessian matrix with entries $\partial_{a_i} \partial_{a_j} H$, this is
\be \begin{pmatrix}
I_r & 0 & 0 & 0 \\
0 & I_r & 0 & 0 \\
-\frac{1}{2\pi} \partial_a^2\sum_j \partial_{a_j} \F \tilde\theta_{\gamma_{e_j}} & 0 & \frac{1}{2\pi} I_r & -\frac{1}{2\pi} \tau \\
0 & -\frac{1}{2\pi} \partial_{\bar a}^2 \sum_j \partial_{\bar a_j} \bar \F \tilde \theta_{\gamma_{e_j}} & \frac{1}{2\pi} I_r & -\frac{1}{2\pi} \bar\tau
\end{pmatrix} \ ,
\ee
where $\F$ is the prepotential of Lemma \ref{lem:D4}(ii). 
As before, the determinant is found to be
\be \begin{vmatrix}
\frac{1}{2\pi} I_r & -\frac{1}{2\pi} \tau \\
\frac{1}{2\pi} I_r & -\frac{1}{2\pi} \bar\tau
\end{vmatrix} = (2\pi)^{-2r} (-2i)^r \det \Imag\tau \ ,
\ee
which is nonzero. 
\end{proof}

\begin{proof}[Proof of Theorem \ref{prop:semiflathyperkahler}]
Using the same Frobenius basis as Lemma \ref{lem:nondegenModel}, we
write
\begin{equation}
\varpi^{\rm sf}(\zeta)= \frac{1}{8 \pi} \sum_{i=1}^r d \mathcal{Y}_{e_i}^{\rm sf} \wedge d \left(p_i^{-1}\mathcal{Y}_{m_i}^{\rm sf}\right),
\end{equation}
and observe that the nondegeneracy
result in Lemma \ref{lem:nondegenModel} implies that $\varpi^{\semif}(\zeta)$ is holomorphic symplectic. Similarly, the nondegeneracy result in Lemma \ref{lem:omega+sf} implies that $\omega_+$ (and hence $\omega_-$) is holomorphic symplectic. Finally, Theorem \ref{thm:twistor} gives us a pseudo-hyper-K\"ahler structure on $\M'$. Lastly,  the signature is as claimed since $\Imag \tau$ is positive-definite.
\end{proof}
The ``semi-flat'' terminology is introduced in \cite{Freed}.  The pseudo-hyper-K\"ahler structure on $\M'$ is called ``semi-flat'' because  the restriction of the metric to each torus fiber is flat, as we show in the following lemma:
\begin{lemma}\label{lem:semiflat}
Let $(\Imag \tau)^{-1/2}$ be the unique symmetric square root of $(\Imag \tau)^{-1}$, and define $(w_1, \cdots, w_r) =((\Imag \tau)^{-1/2}(z_1, \cdots, z_r)$. Then for $u \in \mathcal{B}'$, the restriction of $\omega_3$ to $\M'_u$ is
\be \omega_3|_u = \frac{i}{2}  \cdot \frac{1}{4 \pi} \sum_{i=1}^r dw_i \wedge d \bar w_i, \ee
the standard (flat) K\"ahler form on $\CC^r$.
The normalized period matrix is \be
\left(\mathrm{diag}(p_rp_1^{-1}, \cdots, p_rp_r^{-1})|p_r\tau \right).\ee
\end{lemma}
\begin{proof}
In the $\zeta=0$ complex structure, the projection map $\pi$ is a holomorphic submersion. 
For, the pseudo-K\"ahler form
\be \omega_3 = \frac{1}{8\pi}(2\avg{dZ\wedge d\bar Z} - \avg{d\theta\wedge d\theta}) \ee
shows that $\pi$ is, in fact, a Riemannian submersion, and the restriction of $\omega_3$ to a fiber is
\be \omega_3|_u = \frac{1}{4\pi} \sum_i p_i^{-1} d\tilde\theta_{\gamma_{e_i}}\wedge d\tilde\theta_{\gamma_{m_i}} \ . \ee 

In terms of $z_i, \bar z_j$, the expression for $\omega_3$ is\footnote{c.f. \cite[(3.19)]{GMN:walls}}\footnote{
An efficient way to compute this is to relate the column vector $dz=(dz_1, \cdots, dz_r)$ and the column vector 
 $d\tilde\theta = (d \tilde\theta_{\gamma_{e_1}}, \cdots, d \tilde \theta_{\gamma_{e_r}}, d \tilde \theta_{\gamma_{m_1}}, \cdots, d\tilde\theta_{\gamma_{m_r}})$ via the matrix $dz = M d\tilde\theta$ for $M= \frac{1}{2\pi}\left( -\tau \; | \; P^{-1} \right)$ where $P=\mathrm{diag}(p_1, \cdots, p_r)$ and, then compute \begin{align*}
 \sum_{i,j}(\Imag \tau)^{-1}_{ij} dz_i \wedge d \bar z_j 
 &= dz^T \otimes (\Imag \tau)^{-1} d \bar{z} - d\bar z^T \otimes (\Imag \tau)^{-1} dz \\
 &=d\tilde\theta^T \otimes \begin{pmatrix} 0 & - (\tau - \bar \tau) (\Imag \tau)^{-1} P^{-1} \\ -P^{-1}(\Imag \tau)^{-1} (\bar \tau - \tau) &^ 0 \end{pmatrix} d \tilde\theta\\
  &=d\tilde\theta^T \otimes \begin{pmatrix} 0 & - 2i P^{-1} \\ 2i P^{-1} &^ 0 \end{pmatrix} d \tilde\theta\\
  &=-2i \sum_{i=1}^r p_i^{-1} d\tilde \theta_{\gamma_{e_i}} \wedge d\tilde \theta_{\gamma_{m_i}}.
\end{align*}
}
\be \omega_3|_u = \frac{i}{2}  \cdot \frac{1}{4 \pi} \sum_{i,j} \left((\mathrm{Im} \tau)^{-1}\right)_{ij} dz_i \wedge d \bar z_j \ee
so, the metric is positive-definite. 
Let $(\Imag \tau)^{-1/2}$ be the unique symmetric square root of $(\Imag \tau)^{-1}$, and define $(w_1, \cdots, w_r) =((\Imag \tau)^{-1/2}(z_1, \cdots, z_r)$, so that we get the standard (flat) K\"ahler form on $\CC^r$:
\be \omega_3|_u = \frac{i}{2}  \cdot \frac{1}{4 \pi} \sum_{i=1}^r dw_i \wedge d \bar w_i. \ee
Finally, we note that shifting $\theta \mapsto \theta + 2 \pi$,  
$z_i=\frac{p_i^{-1}\widetilde\theta_{\gamma_{m_i}}-\sum_{j=1}^r \tau_{ij} \widetilde\theta_{\gamma_{e_j}}}{2\pi}$ are complex coordinates defined modulo the addition of a column of $\mathrm{diag}(p_1^{-1}, \cdots, p_r^{-1})$ 
or $\tau$; $w_i$ are complex coordinates defined modulo the addition of a column of $(\Imag \tau)^{-1/2} \mathrm{diag}(p_1^{-1}, \cdots, p_r^{-1})$ or $(\Imag \tau)^{-1/2} \tau$, so that we find that the normalized period matrix (see \cite[p. 306]{GriffithsHarris}) is \be
\left(\mathrm{diag}(p_rp_1^{-1}, \cdots, p_rp_r^{-1})|p_r\tau \right),\ee
as promised in Remark \ref{rem:tau}. Note that $(p_r p_1^{-1}, \cdots, p_r p_r^{-1})$ are the dual elementary divisors for $\Gamma_u^*$ (see \cite[p. 315]{GriffithsHarris}).
This complex torus is an abelian variety precisely because $\tau$ is symmetric and $\mathrm{Im}\; \tau$ is positive definite \cite[p. 306]{GriffithsHarris}. It is principally polarized if, and only if, $p_i=1$. 
\end{proof}

\begin{rem}[Complex torus fibration]
We note that 
$\M'$ is a complex torus fibration in complex structure $\zeta=0$, but it need not be an abelian fibration (or a complex integrable system) --- i.e., $\pi$ need not admit a section. (When $r=1$, one sometimes says that $\M'$ is a genus 1 fibration which may or may not be an elliptic fibration.) In orthogonal complex structures, this translates to the fact that $\M'$ is a special Lagrangian torus fibration, but it need not be a real integrable system. The existence of a section is often determined by the choice of $\theta_f$. We say `often' because even when $\theta_f=0$ it is not evident that a section exists, since $\theta$ is a twisted unitary character, not a unitary character, and so the existence of a global $\theta=0$ section can be obstructed. As described in Remark \ref{ref:untwisted}, the better description for looking for such a global section often involves untwisted characters.
\end{rem}

\bigskip

Lastly, for future use we note the following expression, which is written in terms of terms appearing in the higher-dimensional analog of the Gibbons-Hawking Ansatz (see \S\ref{sec:GH}): 
\begin{lemma}\label{lem:GHsf} The family of holomorphic symplectic forms  
    \begin{align}
        \varpi^{\rm sf}(\zeta) 
        &= \frac{1}{4\pi i} \sum_{i=1}^r d\log \X^{\rm sf}_{\gamma_{e_i}}(\zeta)\wedge \brackets{ p_i^{-1} d\theta_{\gamma_{m_i}}+A_i^{\mathrm{sf}} + \sum_{j=1}^r V^{\mathrm{sf}}_{ij}(\zeta^{-1} da_j-\zeta d\bar a_j) } \ ,
        \end{align}
        where
        \begin{align*}
        A_i^{\mathrm{sf}} &= -\sum_{j=1}^r \Real\tau_{ij} d\theta_{\gamma_{e_j}}\\
        V_{ij}^{\semif} &=\Imag\tau_{ij}
        \end{align*}
\end{lemma}
\begin{proof}
    We compute: 
\begin{align}
\varpi^{\rm sf}(\zeta) &= \frac{1}{8 \pi} \avg{ d \log \X^{\rm sf}(\zeta) \wedge d \log \X^{\rm sf}} \nonumber \\
&= \frac{1}{4 \pi} \sum_{i=1}^r d \log \X^{\rm sf}_{\gamma_{e_i}}(\zeta) \wedge p_i^{-1} d \log \X^{\rm sf}_{\gamma_{m_i}}(\zeta) \nonumber \\
&= \frac{1}{4 \pi} \sum_{i=1}^r d \log \X^{\rm sf}_{\gamma_{e_i}}(\zeta) \wedge  \left[p_i^{-1}i d \theta_{\gamma_{m_i}}+ \sum_{j=1}^r (\zeta^{-1} \partial_{a_j}a_i^D da_j + \zeta  \partial_{\bar a_j} \bar a_i^D d \bar a_j) \right] \nonumber \\
&=
-\frac{1}{4\pi} \sum_{i=1}^r d\log\X^{\rm sf}_{\gamma_{e_i}}(\zeta)\wedge \brackets{ p_i^{-1}id\theta_{\gamma_{m_i}}+\sum_{j=1}^r (\zeta^{-1} \tau_{ij} da_j + \zeta \bar\tau_{ij} d\bar a_j)} \nonumber \\
&= \frac{1}{4\pi i} \sum_{i=1}^r d\log \X^{\rm sf}_{\gamma_{e_i}}(\zeta)\wedge \brackets{  p_i^{-1}d\theta_{\gamma_{m_i}}+ \sum_{j=1}^r \parens{- i \Real\tau_{ij} ( d\log\X^{\rm sf}_{\gamma_{e_j}} - i d\theta_{\gamma_{e_j}}) + \Imag\tau_{ij}(\zeta^{-1} da_j-\zeta d\bar a_j) }} \nonumber \\
&= \frac{1}{4\pi i} \sum_{i=1}^r d\log \X^{\rm sf}_{\gamma_{e_i}}(\zeta)\wedge \brackets{  p_i^{-1}d\theta_{\gamma_{m_i}}+A_i^{\mathrm{sf}} + \sum_{j=1}^r V^{\mathrm{sf}}_{ij}(\zeta^{-1} da_j-\zeta d\bar a_j) } \ ,
\end{align}
where
\begin{align}
A_i^{\mathrm{sf}} &= -\sum_{j=1}^r \Real\tau_{ij} d\theta_{\gamma_{e_j}}\nonumber\\
V_{ij}^{\semif} &=\Imag\tau_{ij}.
\end{align}

\end{proof}
\subsubsection{The space of twisted characters $\mathcal{T}' \to \mathcal{B}'$}\label{sec:twistedcharacters}

For later convenience, we now say a bit more about the space in which $\X^{\rm sf}(\zeta)$ is valued. See Figure \ref{fig:diagram} throughout this section. There is a local system $\T'\to \B'$ whose fiber $\T'_u$ over $u$ is the space of twisted characters $X: \hat\Gamma_u \to \CC^\times$. This is a local system of complex algebraic tori, and it also is naturally a local system of holomorphic Poisson manifolds. The Poisson bracket on $\T'_u$ 
\be \{\; , \;\}: C^\infty (\T'_u) \times C^\infty (\T'_u) \to C^\infty(\T'_i)\ee is determined by the bracket on the coordinates functions labeled by $\gamma \in \widehat{\Gamma}_u$: \begin{align} \mathrm{eval}_\gamma:\T'_u \to \CC^\times \nonumber \\ X \mapsto X_{\gamma};\end{align} 
namely 
\footnote{\label{ft:polarization}The integrality of the symplectic pairing is crucial for this result. If we attempted to allow it to assume rational values, or even real values, and defined twisted characters to be functions satisfying $X_{\gamma+\gamma'}=e^{i\pi\avg{\gamma,\gamma'}} X_\gamma X_{\gamma'}$, then for either choice of sign in $\{X_{\gamma},X_{\gamma'}\} = 4\pi e^{\pm i\pi \avg{\gamma,\gamma'}} \avg{\gamma,\gamma'} X_{\gamma+\gamma'}$ we would not have anticommutativity, the Leibniz rule, or the Jacobi identity.}
\be \{\mathrm{eval}_\gamma,\mathrm{eval}_{\gamma'}\} = 4\pi (-1)^{\avg{\gamma,\gamma'}} \avg{\gamma,\gamma'} \mathrm{eval}_{\gamma+\gamma'} \ . \ee
We will write this as
\be \{X_\gamma,X_{\gamma'}\} = 4\pi (-1)^{\avg{\gamma,\gamma'}} \avg{\gamma,\gamma'} X_{\gamma+\gamma'} \ . \ee
The symplectic leaves of these Poisson manifolds are labelled by the restriction of $X$ to $\Gamma_f$, and each such restriction $X_f$ defines a local system of holomorphic symplectic complex algebraic tori $\T'_{X_f}$. (That is, the fiber $\T'_{X_f,u}$ is the space of twisted characters of $\hat\Gamma_u$ whose restriction to $\Gamma_f$ is the homomorphism $X_f$.) The holomorphic symplectic form on a fiber $\T'_{X_f,u}$ is $\varpi_{X_f} = \frac{1}{8\pi} \avg{d\log X\wedge d\log X}$, where we regard $d\log X$ as a $\Gamma^*$-valued 1-form. 

In particular, for each $\zeta\in\CC^\times$ we obtain a local system of tori with holomorphic symplectic structure $\T'_\zeta := \T'_{X^{\rm sf}_f(\zeta)} \to \mathcal{B}'$ (whose fibers we denote by $\T'_{\zeta,u}$\footnote{Due to the interest of corank one Poisson manifolds, we note that $\cup_{\zeta\in\CC^\times} \T'_{\zeta,u}$ is a natural complete Poisson submanifold --- i.e., an immersed submanifold which is a union of symplectic leaves --- of $\T'_u$.}), where the homomorphism $X^{\rm sf}_f(\zeta)$ is given by \eqref{eq:xSF}, i.e. for constants 
$\{Z_{\gamma}, \theta_\gamma\}_{\gamma \in \Gamma_f}$ from \ref{item:Z} and \ref{item:thetaf} 
\be X_{f,\gamma}^{\semif}(\zeta)= \mathcal{X}_\gamma^{\semif}(\zeta)=
\exp\left( \frac{1}{\zeta} Z_{\gamma} + i \theta_\gamma + \zeta \overline{Z_\gamma}\right), \qquad \gamma \in \Gamma_f. \ee We denote the natural  holomorphic symplectic form on $\mathcal{T}'_\zeta$ by $\varpi_\zeta$.

\begin{figure}[h!]
\begin{centering} 
\includegraphics[height=2.5in]{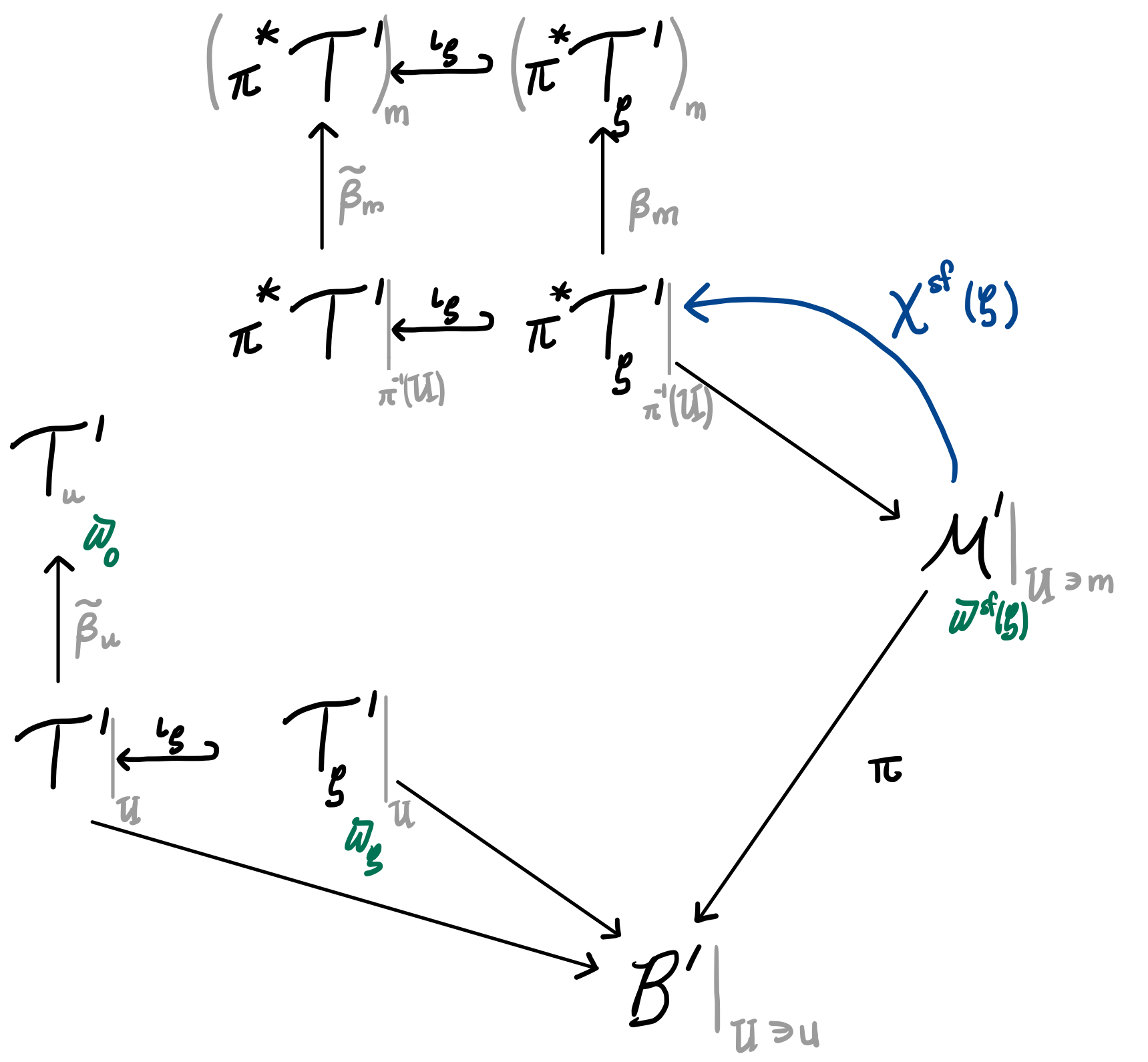} 
\caption{\label{fig:diagram} 
The map $\mathcal{X}^{\semif}(\zeta)$ is naturally a section of $\pi^* \mathcal{T}'_\zeta$.
Lemmata \ref{lem:interp1}, \ref{lem:interp2}, \ref{lem:interp3} relate the holomorphic symplectic form $\varpi^\semif(\zeta)$ on $\mathcal{M}'$ to the holomorphic symplectic forms $\varpi_0$ on $\mathcal{T}'_u$ and $\varpi_\zeta$ on $\mathcal{T}'_\zeta$ (all shown in green) via various restrictions and trivializations (shown in gray) and pullbacks.
}
\end{centering}
\end{figure}

We then see that $\X^{\rm sf}(\zeta)$ is a local section of $\pi^* \T'_\zeta$. If we consider such a section over a set of the form $\pi^{-1}(U)$, where $U$ is a contractible open subset of $\B'$, and we let $\beta_m$ denote the map which parallel transports the points of $(\pi^* \T'_\zeta)|_{\pi^{-1}(U)}$ to a fiber $(\pi^* \T'_\zeta)_{m}$ for some fixed $m\in \pi^{-1}(U)$, then we see that\footnote{c.f.\cite[Footnote 2 on p. 6]{neitzke:hkReview}.}
\begin{lemma}\label{lem:interp1}The holomorphic symplectic form $\varpi^{\rm sf}(\zeta)|_{\pi^{-1}(U)}$ is the pullback via $\beta_m\circ\X^{\rm sf}(\zeta)$ of the holomorphic symplectic form $\pi^* \varpi_\zeta$ on $(\pi^*\T'_{\zeta})_m$.
\end{lemma}
\begin{proof}
    This is immediate since 
    \be \varpi^\semif(\zeta) = \frac{1}{8\pi} \avg{d \log \X^{\semif} \wedge d \log \X^{\semif}}, \qquad \varpi_{X_f} = \frac{1}{8 \pi} \avg{d \log X \wedge d \log X}.\ee
\end{proof}

A final observation will be useful for comparing the various 2-forms $\varpi_{X_f}$: pick any basis $\tilde\gamma_i$ for $\Gamma_u$ and any lift thereof $\gamma_i$ to $\hat\Gamma_u$. Then, for every $X_f$ we have $\varpi_{X_f} = \frac{1}{8\pi} \sigma^{ij} d\log X_{\gamma_i} \wedge d\log X_{\gamma_j}$ for some antisymmetric matrix $\sigma^{ij}$ which does not depend on $X_f$. That is, we may define a 2-form $\varpi_0 = \frac{1}{8\pi} \sigma^{ij} d\log X_{\gamma_i} \wedge d\log X_{\gamma_j}$ on $\T'_u$, and then $\varpi_{X_f}=i_{X_f}^* \varpi_0$, where $i_{X_f}$ is the inclusion of $\T'_{X_f,u}$ into $\T'_u$. Then, in particular, 
\begin{lemma}\label{lem:interp2}The holomorphic symplectic form $\varpi^{\rm sf}(\zeta)|_{\pi^{-1}(U)} = (i_{\zeta}\circ \beta_m \circ \X^{\rm sf}(\zeta))^* \varpi_0$, where $i_\zeta = i_{X^{\rm sf}_f(\zeta)}$.\end{lemma} 
Alternatively, since $\pi^* \mathcal{T}'_\zeta \subset \pi^* \mathcal{T}'$, if we regard $\iota_\zeta \circ \X^{\rm sf}(\zeta)$ as a local section of $\pi^* \T'$ 
and let $\tilde\beta_m$ denote the parallel transport map taking $(\pi^* \T')|_{\pi^{-1}(U)}$ to $(\pi^* \T')_m$, 
then \begin{lemma}\label{lem:interp3}The holomorphic symplectic form $\varpi^{\rm sf}(\zeta)|_{\pi^{-1}(U)} = (\tilde\beta_m\circ \iota_\zeta  \circ \X^{\rm sf}(\zeta))^* \varpi_0$.\end{lemma} We stress that while $\varpi_{X_f}$ is independent of the choice of the charges $\gamma_i$, $\varpi_0$ is not.

\subsection{Gaiotto--Moore--Neitzke's Prescription}
\label{sec:strategy}
Fix $u \in B'$. For each $\gamma \in \widehat{\Gamma}_u$, let $\ell_\gamma(u) = -Z_{\gamma}(u)  \RR^{+}$.
The goal in \cite{GMN:walls} is to produce a map $\mathcal{X}: \M'_u \times \CC^\times \to \mathcal{T}'_u$ with the following properties:
\begin{enumerate}[(1)]
\item $\X$ depends piecewise holomorphically on $\zeta \in \CC^\times$, with discontinuities only at the rays $\ell_\gamma(u)$ for active $\gamma \in \widehat{\Gamma}_u$.
\item Suppose furthermore that at $u \in \B'$  there is no pair of active charges $\gamma_1, \gamma_2 \in \widehat{\Gamma}_u$ (i.e. $\Omega(\gamma_i, u) \neq 0$) such that $\ell_{\gamma_1} =\ell_{\gamma_2}$ and $\avg{\gamma_1, \gamma_2} \neq 0$. 
Then, the limits $\X^\pm$ of $\X$ as $\zeta$ approaches $\ell$ from both sides exist and are 
related by\footnote{It is because $\X$ has such prescribed jumps that such problems are called Riemann--Hilbert problems. } 
$\X^{+} = S_\ell^{-1} \circ \X^{-}$ where
 $$S_\ell(u)= \prod_{\gamma : \ell_{\gamma}(u) = \ell} \mathcal{K}_\gamma^{\Omega(\gamma, u)}$$
where $\mathcal{K}_\gamma$ is a birational Poisson automorphism of $\T'_u$ defined by 
\be \mathcal{K}_\gamma X_{\gamma'} = (1- \X_\gamma)^{\avg{\gamma', \gamma}}X_{\gamma'} \ee 
when the indexing set is finite. 
\item $\X$ obeys the reality condition $\X(-1/\bar \zeta) = \rho^* \X(\zeta)$ where $\rho: \T'_u \to \T'_u$ is an antiholomorphic involution of $\T'_u$ defined by $\rho^* X_\gamma = \bar X_{-\gamma}$.
\item For any $\gamma$, $\lim_{\zeta \to 0} \X_{\gamma}(\zeta)/ \X_{\gamma}^{\semif}(\zeta)$ exists.
\end{enumerate}
The functions $\X_\gamma(\zeta)$ will be holomorphic Darboux coordinates for a $\CC^\times$ family of closed $2$-forms $\varpi(\zeta)$.  They conjecture that $\varpi(\zeta)$ give a hyper-K\"ahler structure on $\M'$ and its natural completition. 

At $u \in \B'$ satisfying the same condition as (2), such a function $\X$ with properties (1)-(4) will obey the integral relation\footnote{As discussed in \cite[p. 42]{GMN:walls} in the bullet point beginning ``Recall that the solution\ldots'', there are other integral relations that produce different maps $\mathcal{X}$ satisfying the properties above. This particular integral relation is simple and there is some physical motivation for it. However, as Gaiotto--Moore--Neitzke mention, there might be a different integral relation that produces $\mathcal{X}(\zeta)$ such that it is clear that the associated holomorphic symplectic form $\omega_+$ (from the $\zeta^{-1}$ term of $\varpi(\zeta)$) is equal to $\omega_+^{\rm sf}$.}: 
\be
\X_{\gamma}(m, \zeta) = \X_{\gamma}^{\semif}(m, \zeta) \exp \left[ - \frac{1}{4 \pi i} \sum_{\gamma'} \Omega(\gamma; u) \avg{\gamma, \gamma'} \int_{\ell_{\gamma'}(u)} \frac{d \zeta'}{\zeta'} \frac{\zeta' + \zeta}{\zeta'-\zeta} \log\left(1- \X_{\gamma'}(m, \zeta') \right) \right].
\ee
(The appropriate modification required when $\avg{\gamma_1, \gamma_2}\neq 0$ but the indexing set is finite is discussed in \cite[Appendix C, (C.16)]{GMN:walls} using the Wall Crossing Formula. When the indexing set it infinite, one must be careful.)
Consequently, Gaiotto--Moore--Neitzke propose seeking a function $\X$ satisfying the above integral relation\footnote{As Neitzke notes on \cite[p. 8]{neitzke:hkReview}, Gaiotto--Moore--Neitzke did not give a complete proof that being a solution of the integral relation implied (4).}. A natural way to produce a solution is by iteration.   When one is sufficiently far away from $\B''$, they propose iterating from the semiflat holomorphic Darboux coordinates $\X^{\semif}$.   When one is near $\B''$ and other components of $\B''$ are sufficiently far away, one should iterate from a different model geometry.

\bigskip

\emph{One cannot make this approach analytically rigorous as it is written, as Gaiotto--Moore--Neitzke themselves make clear.}  The set of rays in $\CC_\zeta$ at which $\X$ has discontinuities could be dense. What then does ``piecewise holomorphic'' mean? The collection of $u \in \B'$ that do not satisfy the restriction in (2) could, in general, also be dense in $\B'$. 
 Consequently, following their suggestion, in the follow-up papers in this series of papers, we will make a few modifications to their prescription.
 
 However, in the remainder of this paper, we will rigorously construct simple model geometries on pieces of $\M$ using their prescription.

\section{Assumptions near $\B''$}  \label{sec:ass}

In this section, we describe a final set of assumptions that we will make that constrain the behavior of all of our data near the excised divisor $\B''$. Looking forward,
these assumptions ensure that the appropriate modification of the Gaiotto--Moore--Neitzke integral relation converges after one iteration.

\bigskip

The assumptions are as follows.
\begin{enumerate}[({A}1)]
\item \label{it:A1}Let $u$ be a point in $\B''$. Then, there is an open neighborhood $U\subset\B$ of $u$ such that the restriction of the local system $\hat\Gamma$ to $U\cap \B'$ participates in a morphism of short exact sequences of local systems of lattices:
\be \begin{tikzcd}
0 \ar[r] & \hat\Gamma_{\rm light} \ar[r] \ar[d,hookrightarrow] & \hat\Gamma \ar[r] \ar[d,"{\rm id}"] & \Gamma_{\rm heavy} \ar[r] \ar[d] & 0 \\
0 \ar[r] & \hat \Gamma_{\rm local} \ar[r] & \hat\Gamma \ar[r] & \Gamma_{\rm nonlocal} \ar[r] & 0 \makebox[0pt][l]{\ ,}
\end{tikzcd} \ee
where $\hat\Gamma_{\rm light},\Gamma_{\rm heavy},\hat\Gamma_{\rm local},$ and $\Gamma_{\rm nonlocal}$ are all trivial, and $\hat\Gamma_{\rm local} = \{\gamma\in \hat\Gamma \bigm| \avg{\gamma,\gamma'}=0 \ \forall \gamma'\in \hat\Gamma_{\rm light}\}$.

\noindent \underline{Notation:} We denote the image of the composition $\hat\Gamma_{\rm light}\hookrightarrow \hat\Gamma\to \Gamma$ by $\Gamma_{\rm light}$ and the image of $\hat\Gamma_{\rm local}\hookrightarrow \hat\Gamma\to \Gamma$ by $\Gamma_{\rm local}$, and we define $\ell := \rank \Gamma_{\rm light}$. (Lemma \ref{lem:latBasis} then implies that $\rank \Gamma_{\rm local} = 2r-\ell$.)

\item \label{it:omega} The restriction of $\Omega$ to $\hat\Gamma_{\rm light}$ is constant, integral, and non-negative and $\tilde\Gamma_{\rm light} := \{\gamma\in \hat\Gamma_{\rm light} \bigm| \Omega(\gamma)\not=0\}$ is finite.

\item \label{it:coords} If $\gamma\in \hat\Gamma_{\rm local}$ then $Z_\gamma$ extends to a holomorphic function on $U$, $dZ|_{\hat\Gamma_{\rm local}}:\hat\Gamma_{\rm local}\otimes\CC\to T^*_{(1,0)}\B$ is surjective, and $\B''\cap U = \{u'\in U \bigm| Z_\gamma(u')=0\mbox{ for some } \gamma\in \tilde\Gamma_{\rm light} \}$.
\end{enumerate}

The following is a nice set of coordinates on $U$ in a neighborhood of $u$: 

\begin{lemma}\label{lem:basisnearsingular} Near $u$, there are 
$\gamma_{e_1}, \cdots, \gamma_{e_\ell} \in \widehat{\Gamma}_{\rm light}$
and $\gamma_{e_{\ell+1}}, \cdots, \gamma_{e_r} \in \widehat{\Gamma}_{\mathrm{local}}$ such that 
   such that the images of $\gamma_{e_1}, \cdots, \gamma_{e_r}$ in $\Gamma_u$ generate a primitive Lagrangian  
   sublattice $\Lambda_e$ of $\Gamma_{\mathrm{local}}\subset \Gamma_u$,
   and $Z_{\gamma_{e_1}}, \cdots, Z_{\gamma_{e_r}}$ comprise a set of holomorphic coordinates near $u$.
\end{lemma}

\begin{proof}
Suppose if $\gamma_{e_1},\ldots,\gamma_{e_r}\in \hat\Gamma_{\rm local}$ are such that $Z_{\gamma_{e_1}},\ldots,Z_{\gamma_{e_r}}$ comprise a set of holomorphic coordinates near $u$, then in particular they are holomorphic coordinates on an open set in $U\cap\B'$ with $u$ in its boundary. So, Corollary \ref{cor:global} applies. 
Lemma \ref{lem:latBasis} shows that we may always assume that $\gamma_{e_1},\ldots,\gamma_{e_\ell}\in \hat\Gamma_{\rm light}$.
\end{proof}

\begin{enumerate}[resume*]
\item \label{it:charge} For any local section $\gamma\in \hat\Gamma$, the central charge is of the form
\be Z_\gamma = \tilde Z_\gamma + \frac{1}{4\pi i} \sum_{\gamma'\in \tilde\Gamma_{\rm light}} \avg{\gamma,\gamma'} \Omega(\gamma') Z_{\gamma'} (\log(Z_{\gamma'}/\pi) - 1) \ , \ee
where $\tilde Z_\gamma$ is a holomorphic function on $U$.
\end{enumerate}
\begin{rem}[Monodromy]\label{rem:monodromy}
As we wind around a cycle in $\B'\cap U$, this corresponds to a monodromy
\be \gamma \mapsto \gamma - \sum_{\gamma'\in \tilde\Gamma_{\rm light}} \frac{n_{\gamma'} \avg{\gamma,\gamma'}\Omega(\gamma')}{2}\gamma' \ , \label{eq:gammaMono} \ee
where the cycle winds $n_{\gamma'}$ times around the divisor $Z_{\gamma'}=0$ in the direction of increasing $\arg Z_{\gamma'}$. We note that the $\half$ at first seems to pose a problem for integrality, but since $\Omega(\gamma')=\Omega(-\gamma')$ and $n_{\gamma'}=n_{-\gamma'}$, the contributions from $\gamma'$ and $-\gamma'$ pair up to cancel out the $\half$.
\end{rem}

\begin{definition}[$\tau, \widetilde{\tau}$]\label{def:tau}
Choose a primitive Lagrangian sublattice $\Lambda_e$ as in Lemma \ref{lem:basisnearsingular}. 
Consider flat local sections $\gamma_{m_1},\ldots,\gamma_{m_r}$,$\gamma_{e_1},\ldots,\gamma_{e_r}\in \hat\Gamma$ whose images in $\Gamma$ comprise a Frobenius basis and such that $\gamma_{e_1},\ldots,\gamma_{e_r}\in \Lambda_e$. As in the semi-flat case,  $a_i = Z_{\gamma_{e_i}}$, $a_i^D = p_i^{-1} Z_{\gamma_{m_i}}$, and \begin{equation}\tau_{ij} = 
\partial_{a_j} a_i^D.\end{equation} 
For any $\gamma\in \hat\Gamma_{\rm light}$ with image $\tilde\gamma\in \Gamma_{\rm light}$, we write $\tilde\gamma=\sum_i p_i^{-1} c_{\gamma,i}\tilde\gamma_{e_i}$, where $c_{\gamma,i}=\avg{\gamma_{m_i},\gamma} \in p_i\ZZ \subset \ZZ$ for all $i$. We also define 
\be \tilde\tau_{ij} = \tau_{ij} - \frac{1}{4\pi i} \sum_{\gamma\in \tilde\Gamma_{\rm light}}\Omega(\gamma) p_i^{-1} p_j^{-1} c_{\gamma,i} c_{\gamma,j} \log(Z_\gamma/\pi) \ . \label{eq:tauLog} \ee
\end{definition}
\begin{rem}[Properties of $\tau, \widetilde{\tau}$]\label{rem:tau}
 Like $\tau$, $\tilde \tau$ is symmetric. 
Note that just as the function $\tau_{ij}$ satisfies $\tau_{ij} = p_i^{-1} \partial_{a_j} Z_{\gamma_{m_i}}$, the function $\tilde\tau_{ij}$ satisfies $\tilde\tau_{ij} = p_i^{-1}\partial_{a_j} \tilde Z_{\gamma_{m_i}}$; these last equalities make it clear that $\tilde\tau_{ij}$ is holomorphic on $U$. We note that
\be \Real\tau_{ij} = \Real\tilde\tau_{ij} + \frac{1}{4\pi} \sum_{\gamma\in\tilde\Gamma_{\rm light}} \Omega(\gamma) p_i^{-1} p_j^{-1} c_{\gamma,i} c_{\gamma,j} \arg(Z_\gamma) \ , \label{eq:realTildeTau} \ee
and
\be \Imag\tau_{ij} = \Imag\tilde\tau_{ij} - \frac{1}{4\pi} \sum_{\gamma\in\tilde\Gamma_{\rm light}} \Omega(\gamma) p_i^{-1} p_j^{-1} c_{\gamma,i} c_{\gamma,j} \log(|Z_\gamma|/\pi) \ . \label{eq:imTildeTau} \ee
The choice of branch cut used to define $\arg(Z_\gamma)$ is not important (changing this just redefines $\tilde Z_\gamma$), but it is important that we be consistent in said choice throughout this section. 
\end{rem}

Before the next assumption, 
we reminder the reader that 
since $\widehat{\Gamma}_{\mathrm{light}}$ and $\widehat{\Gamma}_{\mathrm{local}}$  and  are trivial lattices on $U$, they can be extended to the singular locus $\mathcal{B}'' \cap U$. 
\begin{defn}
    Let $\Theta_{\mathrm{local}}$ be the set of unitary characters on $\widehat{\Gamma}_{\mathrm{local}}$ which restrict to $\theta_f$ on $\widehat{\Gamma}_{\mathrm{local}} \cap \Gamma_f$. 
\end{defn}
For $u' \in U \cap \mathcal{B}$, $\Theta_{\mathrm{local}}$ is a $T^{2r-\ell}$ subtorus of $\mathcal{M}'_{u'}$. However, again, since $\widehat{\Gamma}_{\mathrm{local}}$ is trivial, this subtorus extends over $U \cap \mathcal{B}''$. The subtorus $\Theta_{\mathrm{local}}$ contains a $T^{\ell}$ subtorus $\Theta_{\mathrm{light}}$:
\begin{defn}Let $\Theta_{\mathrm{light}}$ be the set of unitary characters on $\widehat{\Gamma}_{\mathrm{light}}$ which restrict to $\theta_f$ on $\widehat{\Gamma}_{\mathrm{light}} \cap \Gamma_f$.
\end{defn}

\begin{enumerate}[({A}5)] 
\item 
\label{it:smooth} 
\begin{enumerate}[(i)]\item \label{it:light}
Let $\Theta_{\rm light}$ be the set of unitary characters on $\hat\Gamma_{\rm light}$ which restrict to $\theta_f$ on $\hat\Gamma_{\rm light}\cap \Gamma_f$. For any $u'\in U\cap\B''$ and $\theta\in \Theta_{\rm light}$, let $\tilde S_{u',\theta}$ be the subset of those $\gamma\in \tilde\Gamma_{\rm light}$ such that $Z_\gamma(u')=\theta_\gamma=0$, and define a subset $S_{u',\theta}$ of $\tilde S_{u',\theta}$ by arbitrarily choosing a representative of each element of $\tilde S_{u',\theta}/\avg{-1}$. Then, the image of $S_{u',\theta}$ in $\Gamma_{\rm light}$ is a basis for a primitive sublattice thereof. 

\item \label{it:Om1} Finally, $\Omega(\gamma)=1$ for all $\gamma\in \tilde\Gamma_{\rm light}$.
\end{enumerate}
\end{enumerate}

We emphasize that for most $\theta \in \Theta_{\mathrm{light}}$, the set $\widetilde{S}_{u', \theta}$ will be empty; hence, the condition in \ref{it:smooth}
\ref{it:light} is trivial.

\begin{rem}
We strangely have the redundant assumption that $\Omega$ is non-negative on $\tilde\Gamma_{\rm light}$ in assumption \ref{it:omega}. The reason for this is that the present assumption in \ref{it:smooth}\ref{it:Om1} can be dispensed with at the expense of constructing orbifolds instead of manifolds.
To clarify this fact, we will only employ \ref{it:smooth} when it is needed, and will otherwise work in the more general setting where it need not hold.
\end{rem}

\begin{enumerate}[({A}6)]
\item \label{it:vPosAss} There exists an open subset $U'\subset U\cap\B'$ on which 
\be \omega_{U'}:=\avg{dZ\wedge d\bar Z}-\frac{i}{2\sqrt{\pi}} \sum_{\gamma\in\tilde\Gamma_{\rm light}} \frac{\Omega(\gamma)}{\sqrt{|Z_\gamma|}(e^{2|Z_\gamma|}-1)} dZ_\gamma\wedge d\bar Z_\gamma \ee
is a K\"ahler form, and any $u'\in U$ is contained in a closed holomorphic submanifold $D\subset U$ with boundary contained in $U'$ and on which the restriction of the homomorphism $Z$ to a corank 1 sublattice of $\hat\Gamma_{\rm light}$ which contains $\hat\Gamma_{\rm light}\cap \Gamma_f$ is constant. (See Figure \ref{fig:setup}.)

\begin{figure}[h!]
\begin{centering} 
\includegraphics[height=3.0in]{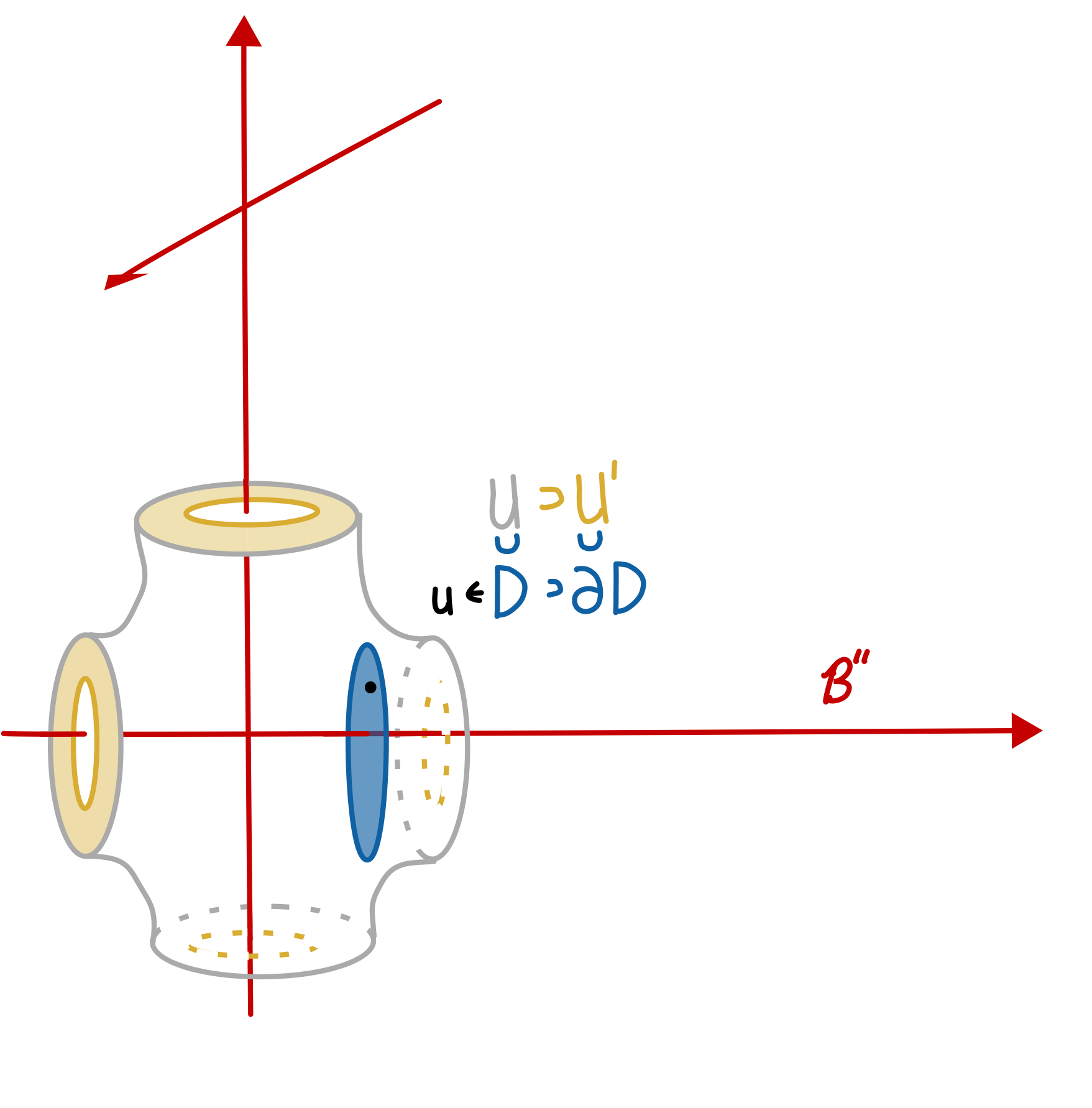} 
\caption{\label{fig:setup} The open sets in Assumption \ref{it:vPosAss} are as follows. The set $U$ is the interior of the indicated gray region. The subset $U'$ is the indicated gold shell within the regular locus $\mathcal{B}'$.
}
\end{centering}
\end{figure}
\end{enumerate}
This assumption is particular novel. 
We can unpack this assumption more explicitly as we did in 
Lemma \ref{lem:D4}\ref{it:Imtau}:
\begin{lemma}\label{lem:A6}
    $\omega_{U'}$ is a K\"ahler form on $U'$ if, and only if, 
\be \Imag\tau_{ij} - \frac{1}{4\sqrt{\pi}} \sum_{\gamma\in\tilde\Gamma_{\rm light}} \frac{\Omega(\gamma) p_i^{-1} p_j^{-1} c_{\gamma,i} c_{\gamma,j}}{\sqrt{|Z_\gamma|}(e^{2|Z_\gamma|}-1)} \ee
is positive-definite.
\end{lemma}
\begin{proof}With the notation introduced above, this K\"ahler form can be written as
\be 2i \sum_{i,j}\parens{\Imag\tau_{ij} - \frac{1}{4\sqrt{\pi}} \sum_{\gamma\in\tilde\Gamma_{\rm light}} \frac{\Omega(\gamma) p_i^{-1} p_j^{-1}c_{\gamma,i} c_{\gamma,j}}{\sqrt{|Z_\gamma|}(e^{2|Z_\gamma|}-1)}} da_i\wedge d\bar a_j \ . \ee
So, the matrix above
must be positive-definite.
\end{proof}
\begin{rem}[\ref{it:vPosAss} intuition]
Intuitively, this requirement means that $U$ must extend far enough away from $\B''$ so that the negative-semi-definite second term is negligible compared to the positive-definite first term on $U'$. (Physically, this is the statement that the UV cutoff of an effective abelian gauge theory with light hypermultiplets is large enough compared to the parameter $1/R=\pi$ mentioned in the introduction.)(See Proposition \ref{prop:OVex}(b) for a concrete computation illustrating the requirements of \ref{it:vPosAss}.)

This assumption did not appear in earlier related papers since those papers stated their theorems ``for $R$ sufficiently large'' rather than ``for $R$ fixed''. By making $R$ larger, one make the positive-definite first term arbitrarily large, hence, the negative-semi-definite second term comparatively neglible.
\end{rem}

\subsection*{Extended example: multi-Ooguri-Vafa}\label{sec:multiOV}
One motivating family of examples are the multi-Ooguri-Vafa spaces. After presenting the multi-Ooguri-Vafa space, we give sufficient conditions such that \ref{it:A1}-\ref{it:vPosAss} are satisfied at $u \in \mathcal{B}''$.
\begin{constr}[multi-Ooguri-Vafa space]\label{constr:multiOV} 
Let $\mathbf{m}=(m_1, \cdots, m_N) \in \mathbb{C}^N$ and $\mathbf{y}=(y_1, \cdots, y_N) \in \mathbb{R}/ 2\pi \mathbb{Z}$ be vectors such that $\sum m_i = 0$ and $\sum y_i \in 2 \pi \mathbb{Z}$.
The multi-Ooguri-Vafa space $\mathcal{M}=\mathcal{M}(\mathbf{m}, \mathbf{y})$ is constructed as follows:
\begin{itemize}
\item The Coulomb base is an open disk $\mathcal{B}=\{|u|<\pi\} \subset \mathbb{C}_u$ and the singular locus is $\mathcal{B}''=\{m_1, \cdots, m_N\} \subset \mathcal{B}$.  Let $\mathcal{B}'_{\mathrm{triv}}$ be a simply-connected subset of $\mathcal{B}'$ obtained by removing rays $\mathrm{cut}_i$ from each distinct point $m_i$ of $\mathcal{B}''$ (see Figure \ref{fig:multiOVcuts})
  \begin{figure}[h!]
\begin{centering} 
\includegraphics[height=1.5in]{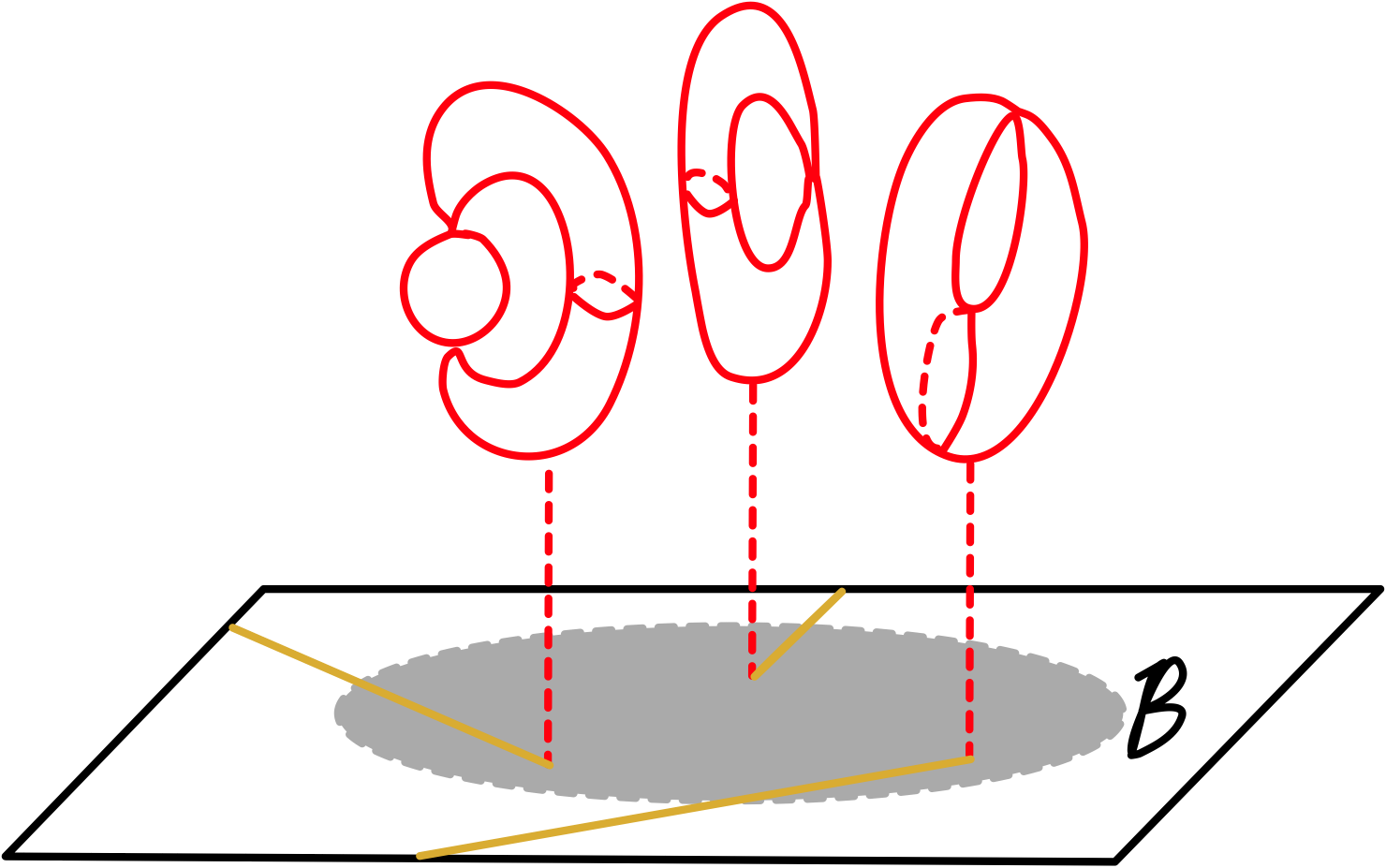}
\caption{\label{fig:multiOVcuts} A multi-Ooguri-Vafa space that is a perturbation of an $I_4$ singular fiber into three singular fibers: $I_2, I_1, I_1$. The cuts in $\mathcal{B}$ are shown in gold.}
\end{centering}
\end{figure}
\item We describe the local system of free abelian groups over $\mathcal{B}'=\mathcal{B}-\mathcal{B}''$ 
\[0 \to \Gamma_f \to \widehat{\Gamma} \to \Gamma \to 0\]
by giving a basis on $\mathcal{B}'_{\mathrm{triv}}$ and the holonomy as one goes around the point $u=m_i$ counterclockwise, crossing the ray $r_i$.
The fibers of the flavor lattice are $\Gamma_{\rm f} = \{\mathbf{v}\in \ZZ^N \bigm| \sum_i v_i = 0\}$ with basis $\{\gamma_i = v_i - v_{i+1}: i=1, \cdots, N-1\}$. We will let $\gamma_N=v_N-v_1= - \sum_{i=1}^{N-1} \gamma_i$. 
As one goes around the point $u=m_i$ counterclockwise 
\begin{align*}
    \rho_{i}: \qquad \gamma_e &\mapsto \gamma_e  \\
    \gamma_m &\mapsto \gamma_m + \sum_{j: m_j =m_i} (\gamma_e -\gamma_j)\\
   \gamma_j&\mapsto \gamma_j
\end{align*}
 
\item The charges\footnote{
One can check that indeed $Z$ is valued in $\mathrm{Hom}(\widehat{\Gamma}, \mathbb{C})$ by checking that
\begin{equation}\label{eq:glue} \left.Z_{\gamma}\right|_{r_i^-} =\left.Z_{\rho_i(\gamma)}\right|_{r_i^+} \end{equation}
for all $\gamma \in \widehat{\Gamma}$, where $r_i^\pm$ denotes the value at $r_i$ if one analytically continues from the clockwise (-) or counterclockwise (+) sides of $r_i$. We demonstrate that \eqref{eq:glue} holds for $\gamma_m$. Note that the value of $\log \frac{u -m_i}{\pi}$ satisfies 
\[\left.\log  \left(\frac{u-m_i}{\pi} \right) \right|_{r_i^-}=\left.\log  \left(\frac{u-m_i}{\pi} \right) \right|_{r_i^+} + 2 \pi i,\]
so 
\begin{align*}
    \left.Z_{\gamma_m}\right|_{r_i^-} &=\left.Z_{\gamma_m}\right|_{r_i^+} + \#\{j:m_j=m_i\}(u-m_i)\\
&=\left.Z_{\gamma_m} + \sum_{j: m_j=m_i} Z_{\gamma_e} - Z_{\gamma_j}  \right|_{r_i^+}\\
&= \left.Z_{\rho_i(\gamma_m)}\right|_{r_i^+}
    \end{align*}
Consequently, 
\begin{align*}
    \left.Z_{\gamma}\right|_{r_i^-} =\left.Z_{\rho_i(\gamma)}\right|_{r_i^+}
\end{align*}} $Z: \widehat{\Gamma} \to \mathbb{C}$ are
\begin{align*}
Z_{\gamma_{i}} &= -m_i\\
    Z_{\gamma_e}(u)&=u\\
    Z_{\gamma_m}(u) &= \frac{1}{2\pi i} \sum_{i=1}^N (u-m_i) \left(\log \frac{u-m_i}{\pi} - 1 \right)
\end{align*}
where $\log \frac{u-m_i}{\pi}$ has a branch cut at $\mathrm{cut}_i$. 
\item The symplectic pairing on $\widehat{\Gamma}$ is determined by $\avg{\gamma_m, \gamma_e}=1$. (Hence, $p_1=1$.)
\item 
The homomorphism $\theta_f: \Gamma_f \to \mathbb{R}/2 \pi \mathbb{Z}$ is given by 
\be
\theta_{\gamma_{i+1}-\gamma_i}=2 \pi y_i.
\ee
Lastly, $\Omega(\gamma,u)$ are zero except for
$$\Omega(\pm (\gamma_e + \gamma_i), u) = 1 \qquad u \in \mathcal{B}'.$$ 
In particular, note that there are no walls.
\end{itemize}
\end{constr}
\begin{proposition}\label{prop:OVex}
Consider $u=m_{i_0}$ in $\mathcal{B}''$ with open neighborhood $\mathcal{B}$.  
Let $J_i=\{j:m_j=m_i\}$.  Suppose  that for each $m_i \in \mathcal{B}''$, the sums in \begin{equation} \left\{\sum_{k=1}^{j-1} 2 \pi y_k \right \}_{j \in J_i} \label{eq:OVsums},\end{equation}
are all distinct in $\RR/2 \pi \ZZ$ and moreover $|m_i|< \frac{\pi}{2}$, then \ref{it:A1}-\ref{it:vPosAss} are satisfied at $u=m_{i_0}$.
\end{proposition}

Intuitively,  if the complex parameters in $Z|_{\Gamma_f}$ are tuned so that there is a non-generic singular fiber over $u'$, then a generic choice---i.e. satisfying \eqref{eq:OVsums}---of real parameters in $\theta|_{\Gamma_f}$ eliminates the singularities of the hyper-K\"ahler metric.

\begin{proof}\hfill
\begin{enumerate}[({A}1)]
\item 
The lattices are 
\begin{equation*} \begin{tikzcd}
0 \ar[r] & \hat\Gamma_{\rm light}= \left\langle\gamma_e, \Gamma_f \right\rangle\ar[r] \ar[d,"{\rm id}"] & \hat\Gamma \ar[r] \ar[d,"{\rm id}"] & \Gamma_{\rm heavy}=\avg{[\gamma_m]} \ar[r] \ar[d,"{\rm id}"] & 0 \\
0 \ar[r] & \hat \Gamma_{\rm local}=\left\langle \gamma_e, \Gamma_f \right\rangle \ar[r] & \hat\Gamma \ar[r] & \Gamma_{\rm nonlocal}=\left\langle[\gamma_m]\right\rangle 
\ar[r] & 0 \makebox[0pt][l]{\ .}
\end{tikzcd}
\end{equation*}
Note that $\hat{\Gamma}_{\mathrm{light}}$ is characterized by the property of being the smallest \emph{primitive} lattice containing $\gamma_e + \gamma_j$ for $j=1, \cdots, N$.
Since $\gamma_e, \gamma_j$ do not monodromize, $\widehat{\Gamma}_{\mathrm{light}}$ and $\widehat{\Gamma}_{\mathrm{local}}$ are both trivial.
Note that $\Gamma_{\mathrm{light}}$, the image of $\widehat{\Gamma}_{\mathrm{light}}$ in the gauge lattice $\Gamma$, is generated by $[\gamma_e]$.
To see that the lattices $\Gamma_{\mathrm{heavy}}$ and $\Gamma_{\mathrm{nonlocal}}$ are trivial we simply observe that as one goes around the point $u=m_i$ counterclockwise, 
\[\gamma_m \mapsto \gamma_m + \sum_{j \in J_i}(\gamma_e+\gamma_j),\]
and that $\sum_{j \in J_i}(\gamma_e+\gamma_j)$ is in $\widehat{\Gamma}_{\mathrm{light}}$.
We can similarly check that $\widehat{\Gamma}_{\mathrm{local}}$ is the sublattice that pairs trivially with $\widehat{\Gamma}_{\mathrm{light}}$.
\item Here,
$\widetilde{\Gamma}_{\mathrm{light}}=\{\pm(\gamma_e + \gamma_j)\}_{j=1}^N$,
is of cardinality $2N$.
We can see that $\Omega$ has constant, integral, non-negative value $1$ on $\widetilde{\Gamma}_{\mathrm{light}}$.
\item When the base $\mathcal{B}$ is parameterized by coordinate $u=Z_{\gamma_e}$, then $\mathcal{B}''=\{m_i: i=1, \cdots, N\}$. Indeed, $Z_{\gamma_e + \gamma_i}(u)=u-m_i$ vanishes at $u=m_i$.
\item Charges $\gamma_e, \gamma_m,\{\gamma_j\}_{j =1}^N$ generate $\widehat{\Gamma}$, and we note that, of these, only the central charge $Z_{\gamma_m}$ monodromizes. We write
\[Z_{\gamma_m} = \underbrace{\frac{1}{2 \pi i} \sum_{j \in \{1, \cdots, N\}-J_i} (u-m_j) \left( \log \frac{u-m_j}{\Lambda} - 1 \right)}_{\widetilde{Z}'_{\gamma_m}} + \frac{1}{2 \pi i} \sum_{j \in J_i} (u-m_j) \left( \log \frac{u-m_j}{\pi} - 1 \right). 
\]
The function $\widetilde{Z}'_{\gamma_m}$ is holomorphic on a neighborhood of $m_i$, but it not quite the function appearing in Assumption \ref{it:charge}. The remainder $Z_{\gamma_m} -\widetilde{Z}'_{\gamma_m}$ is double the sum
\[\sum_{j\in J_i} \frac{1}{4 \pi i} \avg{\gamma_m, \gamma_e +\gamma_j} \Omega(\gamma_e + \gamma_j) Z_{\gamma_e + \gamma_j} (\log(Z_{\gamma_e + \gamma_j}/\pi) -1)\]
Since $\widetilde{\Gamma}_{\mathrm{light}}$ also contains the elements $\{-(\gamma_e + \gamma_j)\}_{j \in J_i}$, we note that the difference between the contribution from $+(\gamma_e + \gamma_j)$ and $-(\gamma_e +\gamma_j)$ is a holomorphic function on $U$.
\item We already stated that $\Omega(\gamma)=1$ for all $\gamma \in \widetilde{\Gamma}_{\mathrm{light}}$. 

For $u'=m_i \in \mathcal{B}''$, the subset of $\widetilde{\Gamma}_{\mathrm{light}}$ for which $Z_{\gamma}(u')=0$ is $\{\pm(\gamma_e + \gamma_j)\}_{j \in J_i}$. Parameterize $\Theta_{\mathrm{light}}$ by the value $\theta=\theta_{\gamma_e + \gamma_i} \in \RR/2 \pi \ZZ$. For each element $\gamma \in \{\pm(\gamma_e + \gamma_j)\}_{j \in J_i}$, $\gamma \in \widetilde{S}_{u', \theta}$ for one value of $\theta$. Namely, 
\be0=\theta_{\pm(\gamma_e + \gamma_j)}=\pm \left(\theta_{\gamma_e +\gamma_i} + \sum_{k=i}^{j-1}(\theta_{k+1}-\theta_k) \right) = \pm\left(\theta +\sum_{k=i}^{j-1} 2\pi y_k, \right)\ee
implies that 
\begin{equation}
    \theta = -\sum_{k=i}^{j-1} 2 \pi y_k \in \RR/2 \pi \ZZ,
\end{equation}
where the sum is taken cyclically. 
Since $[\gamma_e + \gamma_j]$ for $j \in J_i$ all represent the same class in $\Gamma$, the condition that the image of $S_{u', \theta}$ is a basis for a primitive sublattice of the rank one lattice $\Gamma_{\mathrm{light}}$ is equivalent here to the condition that $|S_{u', \theta}|=0 \mbox{ or }1,$
i.e. for any distinct $j_1, j_2 \in J_i$ 
\begin{equation}
     \sum_{k=j_1}^{j_2-1} 2 \pi y_k \notin \RR/2 \pi \ZZ.
\end{equation}
This is precisely the condition we imposed.

\begin{figure}[h!]
\begin{centering}
\includegraphics[height=1.0in]{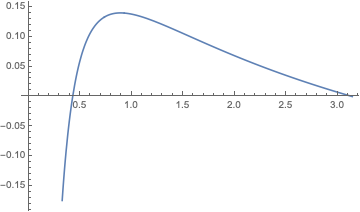}
\caption{\label{fig:function}The proof of \ref{it:vPosAss} in part (b) refers to the function $f(r) = -\frac{1}{2\pi} \log \frac{r}{\pi} + \frac{1}{2\sqrt{\pi}} \frac{1}{\sqrt{|z|} (\exp(2|z|)-1)}$ which we here graph on $r \in (0, \pi)$. There is a root $r_0$ of $f$ near $.42$.}
\end{centering}
\end{figure}

\item Following Lemma \ref{lem:A6} and setting $p_1=1$, we compute: 
\begin{align*}
&\Imag\tau_{11} - \frac{1}{4\sqrt{\pi}} \sum_{\gamma\in\tilde\Gamma_{\rm light}} \frac{\Omega(\gamma) c_{\gamma,1} c_{\gamma,1}}{\sqrt{|Z_\gamma|}(e^{2|Z_\gamma|}-1)}\\
&\qquad =
\sum_{i=1}^{N} \left(-\frac{1}{2 \pi}  \log \frac{|u-m_i|}{\pi}
- \frac{1}{2\sqrt{\pi}} \frac{1}{\sqrt{|u-m_i|}(e^{2|u-m_i|}-1)}\right)
\end{align*}
We observe that the plot of $f(r) = -\frac{1}{2\pi} \log \frac{r}{\pi} + \frac{1}{2\sqrt{\pi}} \frac{1}{\sqrt{|z|} (\exp(2|z|)-1)}$ is given in Figure \ref{fig:function}.

Let $r_0$ be the location of the root of $f(r)$ near $r=.42$ Then $\omega_1$ is clearly positive on the annulus $U'=\{|u| \in (R_1,\pi)\}$ if $|m_i|+ r_0< R_1$ for all $i$. In particular, if $|m_i|< \frac{\pi}{2}$, then $\omega_{U'}$ is positive on $U'=\{|u| \in (\pi -1, \pi)\}.$
\end{enumerate}

\end{proof}

\begin{rem}[$p \neq 1$]\label{rem:pnot1}
Note that one can make a multi-Ooguri-Vafa model with $\avg{\gamma'_m, \gamma_e}=p$ an arbitrary non-zero integer from the above multi-Ooguri-Vafa models with $\avg{\gamma_m, \gamma_e}=1$ as follows by taking $\gamma'_m = p \gamma_m$. Using the same trivialization of $\widehat{\Gamma}$ on $\B'_{\rm triv}$, we are effectively taking $Z_{\gamma'_m}=p Z_{\gamma_m}$ and lattice that now monodromizes as $\gamma'_m \mapsto \gamma'_m + \sum_{j:m_j=m_i}(\gamma_e-\gamma_j)$ around $u=m_i$. Now since our new lattice $\widehat{\Gamma}_{(p)}$ is just a sublattice of the unimodular lattice (see Remark \ref{rem:untwisted}), it is unsurprising that the semi-flat ``prepotential'' $\mathcal{F}$ and $\tau$ are unchanged. One should view this new semi-flat manifold with semi-flat metric as a $\ZZ_p$ quotient of the original semi-flat metric coming from the unimodular superlattice, since $\theta_{\gamma_m} \in \frac{1}{p}\vartheta + \frac{2 \pi \ZZ}{p}$ all correspond to $\theta_{\gamma'_m} = \vartheta$. As we discuss in Example \ref{ex:OVagain}, the smooth model geometry on $\M$ will just be the ordinary multi-Ooguri-Vafa model geometry with relabelling of magnatic coordinate. 
\end{rem}

\section{Smooth manifold structure}\label{sec:smooth}
In this section, we extend $\M' \to \B'$ over the singular locus to a manifold $\M \to \B$. We do this in two steps, first as $\widetilde{\M'}$---which we define in \S\ref{sec:smooth1}---and then as $\M$ in \S\ref{sec:smooth2}. For Ooguri-Vafa, these two steps are shown in \S\ref{sec:introduction} in Figure \ref{fig:manifolds}.

\subsection{Smooth manifold structure, I}\label{sec:smooth1}

Since the lattice $\widehat{\Gamma}$ monodromizes (see Remark \ref{rem:monodromy}), the twisted character $\theta_\gamma$ typically monodromizes. (It doesn't monodromize for $\gamma \in \widehat{\Gamma}_{\mathrm{local}}$.) In this next lemma we define a new coordinate $\theta'_{\gamma}$ which doesn't monodromize, hence is a candidate for a smooth local coordinate on $\tilde \M'$.

\begin{lemma}[Definition of $\theta'$] \label{lem:thetaprime}Let $u$ be a point of $\mathcal{B}''$ and let $U \subset \mathcal{B}$ be an open neighborhood as in \ref{it:A1}. Choose some $u'\in U\cap\B'$; for each $\gamma' \in \widetilde{\Gamma}_{\mathrm{light}}$, choose representatives for the equivalence classes $\theta_{\gamma'}(u'), \mathrm{arg}(Z_{\gamma'}(u')) \in \mathbb{R}/2 \pi \mathbb{Z}$ such that $\theta_{\gamma'} = 2 \pi - \theta_{- \gamma'}$. Then, for each $\gamma\in \hat\Gamma_{u'}$,
\be \theta'_{\gamma} := \theta_\gamma - \sum_{\gamma'\in\tilde\Gamma_{\rm light}} \frac{\avg{\gamma,\gamma'} \Omega(\gamma') \arg(Z_{\gamma'})}{2\pi} (\theta_{\gamma'} - \pi) \label{eq:thetaPrime} \ee
defines a function from $\pi^{-1}(U\cap\B')\setminus\cup_{\gamma'\in \tilde\Gamma_{\rm light}} \theta_{\gamma'}^{-1}(0)$ to $\RR/2\pi\ZZ$. Observe that $\theta'_\gamma = \theta_\gamma$ if $\gamma\in \hat\Gamma_{\rm local}$. 
\end{lemma}

\begin{rem}
    (This function depends on the choices of the representatives for the equivalence classes $\theta_{\gamma'},\arg(Z_{\gamma'})\in\RR/2\pi\ZZ$. As above, we must have $\theta_{\gamma'}=2\pi-\theta_{-\gamma'}$ mod $4\pi$; we make the precise choice $\theta_{\gamma'}\in (0,2\pi)$ for all $\gamma'\in\tilde\Gamma_{\rm light}$.  In contrast, the choice of representative of $\arg(Z_{\gamma'})$ does not matter for the present purposes --- i.e., it does not affect the manifold $\tilde\M'$ which we are about to define --- but for later purposes should, as mentioned above, be chosen consistently with \eqref{eq:realTildeTau} in Remark \ref{rem:tau}.
     We excise $\cup_{\gamma'\in \tilde\Gamma_{\rm light}} \theta_{\gamma'}^{-1}(0)$ so that we can smoothly lift $\theta_{\gamma'}$ to $\RR$ --- i.e., these functions $\theta'_\gamma$ are invariant under monodromies in $U$, but not in the torus fibers.
     
\end{rem}
\begin{proof} The monodromy of $\widehat{\Gamma}$ in 
\eqref{eq:gammaMono} implies that the twisted character $\theta$ monodromizes as
\be \theta_\gamma \mapsto \theta_\gamma + \sum_{\gamma'\in \tilde\Gamma_{\rm light}} \frac{n_{\gamma'} \avg{\gamma,\gamma'}\Omega(\gamma')}{2} (\theta_{\gamma'}+\pi \avg{\gamma,\gamma'}) = \theta_\gamma + \sum_{\gamma'\in \tilde\Gamma_{\rm light}} \frac{n_{\gamma'} \avg{\gamma,\gamma'}\Omega(\gamma')}{2} (\theta_{\gamma'}\pm \pi) \ , \label{eq:thetaMono} \ee
where the cycle winds $n_{\gamma'}$ times around the divisor $Z_{\gamma'}=0$ in the direction of increasing $\mathrm{arg}\, Z_{\gamma'}$.
Here, the half-integrality issue requires that our lifts from $\RR/2\pi\ZZ$ to $\RR$ satisfy $\theta_{\gamma'} = -\theta_{-\gamma'}$ mod $4\pi$, not just mod $2\pi$. The equality in \eqref{eq:thetaMono} is mod $2\pi$, and uses the fact that $n^2\equiv n$ (mod 2) for all $n\in \ZZ$. The meaning of the $\pm$ sign is that for each pair $\pm\gamma'$, one choice should be made for $\gamma'$ and the other for $-\gamma'$; the choice here is irrelevant mod $2\pi$. We introduce this new notation in order to stress that the second term in \eqref{eq:thetaMono} is linear in $\gamma$. However, we can now eliminate the `$\pm$' notation by instead assuming that \begin{equation}\theta_{\gamma'}=2\pi-\theta_{-\gamma'} \;\mathrm{mod} \; 4\pi, \label{eq:removepm}\end{equation} in which case we have
\be \theta_\gamma \mapsto \theta_\gamma + \sum_{\gamma'\in\tilde\Gamma_{\rm light}} \frac{n_{\gamma'} \avg{\gamma,\gamma'} \Omega(\gamma')}{2}(\theta_{\gamma'}-\pi) \ . \ee
So, if we choose some $u'\in U\cap\B'$, then for each $\gamma\in \hat\Gamma_{u'}$,
\be \theta'_{\gamma} := \theta_\gamma - \sum_{\gamma'\in\tilde\Gamma_{\rm light}} \frac{\avg{\gamma,\gamma'} \Omega(\gamma') \arg(Z_{\gamma'})}{2\pi} (\theta_{\gamma'} - \pi)  \ee
defines a function from $\pi^{-1}(U\cap\B')\setminus\cup_{\gamma'\in \tilde\Gamma_{\rm light}} \theta_{\gamma'}^{-1}(0)$ to $\RR/2\pi\ZZ$, and $\theta'_\gamma = \theta_\gamma$ if $\gamma\in \hat\Gamma_{\rm local}$.

\end{proof}

\begin{constr}[$\widetilde{\M'}$]\label{constr:Mtilde}
    We glue the manifold $U\times \M'_{u'}\setminus \cup_{\gamma'\in\tilde\Gamma_{\rm light}} \theta_{\gamma'}^{-1}(0)$ onto $\pi^{-1}(U\cap\B')$ via the diffeomorphism which identifies the fiber coordinate $\theta$ in the former with the function $\theta'$ in the latter, on $\pi^{-1}(U\cap\B')\setminus \cup_{\gamma'\in\tilde\Gamma_{\rm light}} \theta_{\gamma'}^{-1}(0)$. In this way, we construct a new fibered manifold $\pi: \tilde\M'\to \B$ which agrees with $\pi: \M'\to \B'$ over $\B'$. 
\end{constr}

\bigskip

\noindent \underline{Notation:} For later use, we introduce the notation $\tilde\theta'$ for an untwisted  unitary character on $\hat\Gamma_{u'}$ which is obtained from $\theta'$ by choosing a quadratic refinement at $u'$, as in Remark \ref{rem:untwisted}.

\bigskip

Lastly, we include a final lemma about these new primed coordinates:
\begin{lem}\label{lem:primedJacobian}
    The Jacobian determinant
    \be \abs{\frac{\partial(\tilde\theta_{\gamma_{e_1}},\ldots,\tilde\theta_{\gamma_{e_r}},\tilde\theta_{\gamma_{m_1}},\ldots,\tilde\theta_{\gamma_{m_r}},a_1,\ldots,a_r,\bar a_1,\ldots,\bar a_r)}{\partial(\tilde\theta_{\gamma_{e_1}},\ldots,\tilde\theta_{\gamma_{e_r}},\tilde\theta'_{\gamma_{m_1}},\ldots,\tilde\theta'_{\gamma_{m_r}},a_1,\ldots,a_r,\bar a_1,\ldots,\bar a_r)}}= 1. \label{eq:primedJacobian}\ee
\end{lem}
\begin{proof}The coordinate vector fields are related by 
    \begin{align}
    \frac{\partial}{\partial \theta_{m_i}} &=\frac{\partial}{\partial \theta'_{m_i}} \nonumber \\
      \frac{\partial}{\partial \theta_{e_i}} &=\frac{\partial}{\partial \theta_{e_i}} + \sum_j \frac{\partial \theta'_{m_j}}{\partial \theta_{e_i}} \frac{\partial}{\partial \theta'_{m_j}}  \nonumber \\
      &= \frac{\partial}{\partial \theta_{e_i}} + \sum_j \sum_{\gamma' \in \widetilde{\Gamma}_{\mathrm{light}}} \frac{\avg{\gamma_{m_j}, \gamma'} \Omega(\gamma')\mathrm{arg}(Z_\gamma')}{2\pi}  \frac{\partial \theta_{\gamma'}}{\partial \theta_{e_i}} \frac{\partial}{\partial \theta'_{m_j}}\\\
           \frac{\partial}{\partial a_i} &=\frac{\partial}{\partial a_i} + \sum_j \frac{\partial \theta'_{m_j}}{\partial a_i} \frac{\partial}{\partial \theta'_{m_j}}  \nonumber \\
      &= \frac{\partial}{\partial a_i} + \sum_j \sum_{\gamma' \in \widetilde{\Gamma}_{\mathrm{light}}} \frac{\avg{\gamma_{m_j}, \gamma'} \Omega(\gamma')(\theta_{\gamma'} - \pi)}{2 \pi} \frac{\partial \mathrm{arg}(Z_\gamma')}{\partial a_i} \frac{\partial}{\partial \theta'_{m_j}}.
       \end{align}
        and likewise for $\frac{\partial}{\partial \bar a_i}$.
  Consequently, the inverse of the Jacobian determinant in \eqref{eq:primedJacobian} is 
\begin{align}
  \abs{\frac{\partial(\tilde\theta_{\gamma_{e_1}},\ldots,\tilde\theta_{\gamma_{e_r}},\tilde\theta'_{\gamma_{m_1}},\ldots,\tilde\theta'_{\gamma_{m_r}},a_1,\ldots,a_r,\bar a_1,\ldots,\bar a_r)}{\partial(\tilde\theta_{\gamma_{e_1}},\ldots,\tilde\theta_{\gamma_{e_r}},\tilde\theta_{\gamma_{m_1}},\ldots,\tilde\theta_{\gamma_{m_r}},a_1,\ldots,a_r,\bar a_1,\ldots,\bar a_r)}}
  &= \abs{\begin{pmatrix} I_r & 0_r& 0_r& 0_r \\ *& I_r & *  &*  \\0 &0 & I_r &0 \\ 0& 0& 0& I_r \end{pmatrix}}\nonumber\\
  &=1.
\end{align}
\end{proof}

\subsection{Generalized Gibbons--Hawking Ansatz}\label{sec:GH}

In order to make use of Assumption \ref{it:vPosAss}, i.e. (see Lemma \ref{lem:D4}) the assumption that 
\be \Imag\tau_{ij} - \frac{1}{4\sqrt{\pi}} \sum_{\gamma\in\tilde\Gamma_{\rm light}} \frac{\Omega(\gamma) p_i^{-1} p_j^{-1}c_{\gamma,i} c_{\gamma,j}}{\sqrt{|Z_\gamma|}(e^{2|Z_\gamma|}-1)} \ee
is positive definite on some open set $U'$ (roughly obtained from removing neighborhood of $\B''$ from $U$, as shown in Figure \ref{fig:setup}), in Lemma \ref{lem:gwOV}, we restate and generalize  
\cite[Lemma 3.1]{gross:OV}.
Gross--Wilson's motivation in \cite[\S3]{gross:OV} is to describe the Ooguri-Vafa metric via the Gibbons--Hawking ansatz, as a $\ZZ$-quotient of a $U(1)$-fiber bundle over an open subset of $\RR^3$
\footnote{We follow and adapt the conventions in \cite[\S4.1]{GMN:walls} from Ooguri-Vafa to multi-Ooguri-Vafa, setting $R=\frac{1}{\pi}$, though we write $\Theta= i d \theta + i A$ rather than $\Theta^{GMN}=i d \theta + 2 \pi i A$.}.

We will state  Lemma \ref{lem:gwOV} in a more quantitative form than \cite[Lemma 3.1]{gross:OV}, since the original lemma is of the `for sufficiently large $R$' variety described in the last subsection of \S\ref{sec:introduction}. This creates new possibilities, with new complications, since in this setting with fixed $R$ we can consider singular fibers in $\M$ which are close (relative to the length scale set by the size of the torus fibers), but not coincident (as in Lemma 4.1 of \cite{chen:k3sing}). 
\bigskip

Before launching into the technical details, we recall the 4d Gibbons-Hawking ansatz which produces 4d hyper-K\"ahler manifolds with a triholomorphic $U(1)$-action . The data  is (1) an open set $\widetilde{Y} \subset \CC \times \RR$ with coordinates $a, \theta_{\gamma_e}$ and flat metric $g=4\de a \de \overline{a} +  \de \theta_e^2$, 
(2) a positive harmonic function $V: \widetilde{Y} \to \RR$ such that if we define $F=2 \pi i \star \de V \in \Omega(\widetilde{Y}, i\RR)$, the associated cohomology class $[\frac{i}{2\pi} F]$ lies in the image of $H^2(\widetilde{Y},\ZZ)$ inside $H^2(\widetilde{Y}, \RR)$ and (3) a principle $U(1)$-bundle $\pi:\widehat{\M}' \to \widetilde{Y}$ with connection $\Theta \in \Omega^1(\widetilde{\M}', i \RR)$ such that $\pi^*F = \de \Theta$, i.e.
\be \Theta =i \de \theta + i A, \ee 
where $d A = \star d V$.
From this data, we get a hyper-K\"ahler structure on $\widehat{\M}'$. The hyper-K\"ahler metric is \begin{equation}\label{eq:gGH}
g= 4 \left( V \pi^*\left(\de u \de \overline{u}+\frac{1}{4} \de \theta_e^2\right)+ V^{-1}\left(\frac{1}{2 \pi i} \Theta\right)^2\right). \end{equation}
The family of holomorphic symplectic forms is 
\be \varpi(\zeta) = -\frac{1}{4 \pi} \xi_e \wedge \xi_m,\ee
where 
\begin{align}
\xi_e&=i \de \theta_e + \zeta^{-1} \de u + \zeta \de \overline{u} \nonumber \\
\xi_m &= \Theta + \pi i V( \zeta^{-1} \de u - \zeta \de \overline{u}).
\end{align}
We can quotient by the $\ZZ$ action $\theta_e \mapsto \theta_e + 2 \pi R \ZZ$ for $R=\frac{1}{\pi}$ to get a hyper-K\"ahler manifold $\widetilde{\M}'$.

\bigskip

This extends to higher-dimensions.
If a hyper-K\"ahler manifold $\mathcal{M}$ has a triholomorphic $U(1)^r$ action, then there is an associated hyper-K\"ahler moment map $\mu: \mathcal{M} \to \Imag \HH \otimes \mathfrak{u}(1)^r \simeq \RR^{3r}.$
A fundamental idea in Lindstr\"om--Rocek \ref{rocek:tensorDuality} is such a hyper-K\"ahler manifold can be locally constructed from a real-valued function $\Psi$ on an open subset $\widetilde{Y} \subset \RR^{3n} \simeq \CC^n_{a_i} \times \RR^n_{\theta_{e_i}}$. The Legendre transform of this function gives a K\"ahler potential for the hyperk\"ahler metric.
Pedersen and Poon \cite{poon:hk} then observe\footnote{Here, we closely follow \cite{Bielawski}.
} that the hyper-K\"ahler metric has the form 
\begin{equation}
g=4\sum_{i,j=1}^r V_{ij}(\de a_i \de \bar a_j +\frac{1}{4}\de \theta_{e_i} \de \theta_{e_j})
+  \frac{1}{4}(V^{-1})_{ij}(\de \theta_{\gamma_{m_i}} + A_i )(\de \theta_{\gamma_{m_j}} + A_j),
\end{equation}
where
where $V_{ij}$ are polyharmonic functions on $\widetilde{Y}$, meaning they are harmonic on each affine subspace $a + \Imag \HH \otimes \RR \alpha$ for $\alpha \in (\mathfrak{u}(1)^r)^* -\{0\}$, 
and $d \theta_{\gamma_{m_i}} + A_i$'s, $i=1, \cdots, r$ describe a connection on the $U(1)^r$ bundle.
Moreover, in terms of the original function $\Psi$, $V_{ij} = \frac{1}{4} \partial_{\theta_{e_i}} \partial_{\theta_{e_j}} \Psi$ and $A_i = \frac{i}{2} \sum_k \left(\partial_{\theta_{e_i}} \partial_{\bar a_k} \Psi \de \bar a_k - \partial_{\theta_{e_i}} \partial_{ a_k} \Psi \de a_k\right)$.

\bigskip
We now outline the results in Lemma \ref{lem:gwOV}.
In (a-c), we closely mirror \cite[Lemma 3.1(a-c)]{gross:OV}, though our bounds are stronger.
In (d-h) we describe $Y, V_{ij}$ and a connection $A_i$ (appearing in \eqref{eq:ai}) as in the higher-dimensional Gibbons--Hawking ansatz in terms of our data. 
Note that we will not actually use the machinery of the higher-dimensional Gibbons--Hawking ansatz to prove that our model geometries are smooth hyper-K\"ahler manifolds; instead, we will again use holomorphic Darboux coordinates via Corollary \ref{cor:twist}.

\begin{lemma} \hfill
\label{lem:gwOV}
\begin{enumerate}[(a)]
\item The series\footnote{Note that Gross-Wilson's $T$ is $\frac{1}{4\pi}$ times our $T$.}
\be T(w,\theta) = \sum_{m=-\infty}^\infty \parens{\frac{\pi}{\sqrt{4|w|^2+(2\pi m+\theta)^2}} - \kappa_m} \ , \quad \kappa_m = \piecewise{\half \log \frac{2|m|+1}{2|m|-1}}{m\not=0}{0}{m=0} \label{eq:tj} \ee
converges absolutely and uniformly on compact subsets of $(\CC\times S^1)\setminus \{(0,0)\}$ (where $S^1=\RR/2\pi\ZZ$, so $0\in S^1$ is $[0]\in \RR/2\pi\ZZ$) to a harmonic function $T$, where $\CC\times S^1$ is endowed with the flat metric $g=4 dw\, d\bar w+d\theta^2$. The termwise partial derivatives of this series of all orders also converge absolutely and uniformly on compact sets to the corresponding partial derivative of $T$.
\item When $|w|\not=0$, we have
\be T(w,\theta) = -\log(|w|/\pi) + \sum_{n\not=0} e^{in\theta} K_0(2|n w|) \ , \label{eq:poisson} \ee
where $K_0$ is the modified Bessel function which satisfies \cite[\S10.32]{NIST}
\be K_0(|y|)=\half\int_{-\infty}^\infty e^{iy\sinh u}\, du = \int_{0}^\infty \cos(y\sinh u)\, du = \int_0^\infty \frac{\cos(y t)}{\sqrt{1+t^2}}\, dt \ee
for all $y\in \RR\setminus\{0\}$. The series in \eqref{eq:poisson} converges absolutely and uniformly on compact subsets of $(\CC\times S^1)\setminus(\{0\}\times S^1)$.
\item When $|w|\not=0$,
\be |T(w,\theta)+\log(|w|/\pi)| \le \frac{\sqrt{\frac{\pi}{|w|}}}{e^{2|w|}-1} \ . \ee
\end{enumerate}
Now make assumptions \ref{it:A1}-\ref{it:charge}, \ref{it:vPosAss}. 
\begin{enumerate}[resume*]
\item \label{it:d}Given $u' \in U \cap \mathcal{B'}$. Fix $\Lambda_e$ a maximal isotropic subspace of $\Gamma_u$ such that $\Gamma_{\mathrm{light}} \subset \Lambda_e \subset \Gamma_{\mathrm{local}}$.  Choose a Frobenius basis $\gamma_{e_1},\ldots,\gamma_{e_r},\gamma_{m_1},\ldots,\gamma_{m_r}$ of $\Gamma_u$.
Define
\be Y = U\times\Theta\setminus \cup_{\gamma\in\tilde\Gamma_{\rm light}} (Z_\gamma^{-1}(0) \times \theta_\gamma^{-1}(0)) \ , \ee
where $\Theta$ is the $r$-torus consisting of unitary characters on $\ZZ[\gamma_{e_1},\ldots,\gamma_{e_r}]$. Endow this with the flat metric $g=\sum_{i=1}^r 4 da_i d\bar a_i + \sum_{i=1}^r d\theta_{\gamma_{e_i}}^2$. Then, each entry of the real symmetric matrix
\be \label{eq:Vog} V_{ij}(u,\theta) = \Imag\tilde\tau_{ij}(u) + \frac{1}{4\pi} \sum_{\gamma\in\tilde\Gamma_{\rm light}} \Omega(\gamma)p_i^{-1} p_j^{-1}c_{\gamma,i} c_{\gamma,j} T(Z_\gamma(u),\theta_\gamma) \ee
is a harmonic function on $Y$ which, moreover, is harmonic on any submanifold defined by fixing the restriction of both $Z$ and $\theta$ to some sublattice of $\hat\Gamma_{\rm light}$ which contains $\hat\Gamma_{\rm light}\cap \Gamma_f$.
\item $V$ is positive-definite on $Y$.
\item \label{it:Vij} If $u\in U\cap\B'$ and $\theta\in \Theta$, then 
\begin{align}
V_{ij}(u,\theta) &= \Imag \tau_{ij}(u) + \frac{1}{4\pi} \sum_{\gamma\in\tilde\Gamma_{\rm light}} \Omega(\gamma) p_i^{-1} p_j^{-1} c_{\gamma,i} c_{\gamma,j} \sum_{n\not=0} e^{in\theta_\gamma} K_0(2|n Z_\gamma(u)|) \nonumber \\
&= \Imag\tau_{ij}(u) + \frac{1}{4\pi} \sum_{\gamma\in\tilde\Gamma_{\rm light}} \Omega(\gamma)p_i^{-1} p_j^{-1} c_{\gamma,i} c_{\gamma,j} \int_{\ell_\gamma(u)} \frac{d\zeta}{\zeta} \frac{\X^{\rm sf}_\gamma(\zeta)}{1-\X^{\rm sf}_\gamma(\zeta)} \ , \label{eq:newV}
\end{align}
where $\ell_\gamma(u)$ is the contour from 0 to $\infty$ along which $\arg(-Z_\gamma(u)/\zeta)=0$.
\item \label{it:bundle} $\pi^{-1}(U)\subset \tilde\M'$ admits the structure of a principal $U(1)^r$-bundle $p: \pi^{-1}(U)\to Y$ with a connection whose curvature is given by the $r$ real 2-forms
\item 
\be F_i = \sum_{j,k=1}^r \brackets{i d\theta_{\gamma_{e_j}}\wedge (\partial_{a_k} V_{ij} \, da_k - \partial_{\bar a_k} V_{ij} \, d\bar a_k) + 2i \partial_{\theta_{\gamma_{e_i}}} V_{jk} \, da_j\wedge d\bar a_k} \ . \label{eq:curvature} \ee
\end{enumerate}
\end{lemma}

\begin{figure}[h!]
\begin{centering} 
\includegraphics[height=1.2in]{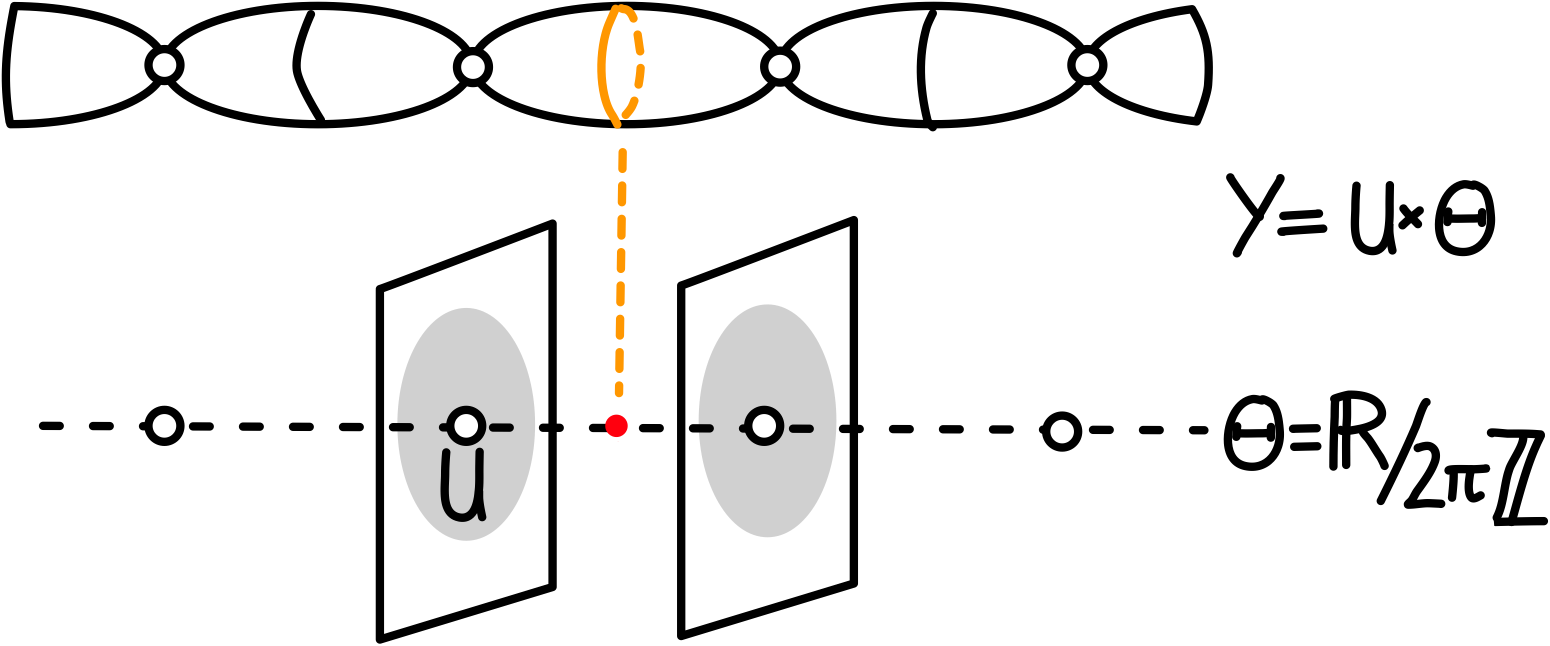} 
\caption{\label{fig:Y} The map $p: \widetilde{M'} \to Y$ described in Lemma \ref{lem:gwOV}\ref{it:d}-\ref{it:bundle}is shown for Ooguri-Vafa. More precisely, note that picture is of this is the $\ZZ^r$ universal cover of $\Theta$, with implicit $\ZZ^r$ action.}
\end{centering}
\end{figure}

\begin{proof}\hfill
\begin{enumerate}[(a)]
\item See \cite{gross:OV} for a proof of (a), except the last part of (a) which follows from the proof of (c) by the Weierstrass M-test. (Absolute convergence in (a) is not proven in this reference, but this is easily verified, since $\kappa_m = \frac{1}{2|m|}+\Oo(1/m^2)$. The constants $\kappa_m$ in \eqref{eq:tj} differ from those in \cite{gross:OV}, but one easily verifies that the limit $T$ is the same with either choice of constants. See also  \cite[pp 36-37]{mz:K3HK} for a concise discussion of the Poisson resummation that yields \eqref{eq:poisson}.)
\item  See \cite{gross:OV} for a proof of (b), except for the uniform and absolute convergence in (b) which similarly follows from the proof of (c) by the Weierstrass M-test.
\item  When $y>0$ we have \cite[\S\S10.37,10.40]{NIST}
\be 0 \le K_0(y) \le \sqrt{\frac{\pi}{2y}} e^{-y} \ . \ee
So, when $|w|\not=0$ we have $|K_0(2|nw|)|\le \sqrt{\frac{\pi}{4|w|}} e^{-2|nw|}$, and so
\be
\abs{T(w,\theta)+\log(|w|/\pi)} \le \sqrt{\frac{\pi}{|w|}} \sum_{n>0} e^{-2|nw|}
= \frac{\sqrt{\frac{\pi}{|w|}}}{e^{2|w|}-1} \ .
\ee
\item  Let $\tilde Y$ be a submanifold defined by fixing the restriction of both $Z$ and $\theta$ to a corank $k$ sublattice $\Gamma_{\rm intermediate}$ of $\Gamma_{\rm local}$ which contains $\Gamma_f$ (where $\tilde Y=Y$ if $k=r$), and let $\tilde g$ be the induced metric. We take $a_{\iota(1)},\ldots,a_{\iota(k)},\theta_{\gamma_{e_{\iota(1)}}},\ldots,\theta_{\gamma_{e_{\iota(k)}}}$ to be our coordinates on $\tilde Y$, where $\iota$ is a permutation of $\{1,\ldots,r\}$. The metric then takes the form
\be \tilde g = 4\sum_{I,J=1}^{k} h_{IJ} ( da_{\iota(I)} d\bar a_{\iota(J)} + \frac{1}{4} d\theta_{\gamma_{e_{\iota(I)}}} d\theta_{\gamma_{e_{\iota(J)}}}) \ , \ee
where $h_{IJ}$ is a real positive-definite matrix with inverse $h^{IJ}$. (Explicitly, if for each $i>k$ we have $\gamma_{e_{\iota(i)}}=\sum_{I=1}^{k} d_{i,I} \gamma_{e_{\iota(I)}}$, modulo $\Gamma_{\rm intermediate}$, then $h_{IJ} = \delta_{IJ}+\sum_{i=k+1}^r d_{i,I} d_{i,J}$. However, the precise choice of $h$ does not affect the following argument.) If $\gamma=\sum_{I=1}^{k} c'_{\gamma,I} \gamma_{e_{\iota(I)}}$, modulo $\Gamma_{\rm intermediate}$, then
\begin{align} &h^{IJ} \left(\half\partial_{a_{\iota(I)}} \partial_{\bar a_{\iota(J)}}+\half \partial_{a_{\iota(J)}} \partial_{\bar a_{\iota(I)}}+\partial_{\theta_{\gamma_{e_{\iota(I)}}}} \partial_{\theta_{\gamma_{e_{\iota(J)}}}}\right) T(Z_\gamma(u),\theta_\gamma) \nonumber \\
&= h^{IJ} c'_{\gamma,I} c'_{\gamma,J} (\partial_{Z_\gamma} \partial_{\bar Z_\gamma} + \partial_{\theta_\gamma}^2) T(Z_\gamma,\theta_\gamma) \nonumber \\
&= 0 \ . 
\end{align}
Similarly,
\be h^{IJ} \left(\half\partial_{a_{\iota(I)}} \partial_{\bar a_{\iota(J)}}+\half \partial_{a_{\iota(J)}} \partial_{\bar a_{\iota(I)}}+\partial_{\theta_{\gamma_{e_{\iota(I)}}}} \partial_{\theta_{\gamma_{e_{\iota(J)}}}}\right) \Imag\tilde\tau_{ij} = 0 \ . \ee
It follows that $V_{ij}$ is harmonic on $\tilde Y$.

\item  For any constant non-zero vector $v$, $u\in U'$, and $\theta\in T^r$, then using \eqref{eq:tauLog}, 
\begin{align}
v^T V(u,\theta) v &= v^T \Imag\tau \, v + \frac{1}{4\pi} \sum_{\gamma\in \tilde\Gamma_{\rm light}} \Omega(\gamma) (v\cdot (p^{-1} c)_\gamma)^2 (T(Z_\gamma(u),\theta_\gamma)+\log(|Z_\gamma(u)|/\pi)) \nonumber \\
&\ge v^T\parens{\Imag\tau - \frac{1}{4\sqrt{\pi}} \sum_{\gamma\in\tilde\Gamma_{\rm light}} \frac{\Omega(\gamma) (p^{-1} c)_\gamma \cdot (p^{-1}c)_\gamma}{\sqrt{|Z_\gamma(u)|}(e^{2|Z_\gamma(u)|}-1)}} v \nonumber \\
&> 0 \ ,
\end{align}
where $(p^{-1} c)_\gamma$ is the vector with component $(p^{-1} c_\gamma)_i= p^{-1}_i c_{\gamma, i}$.
If $(u,\theta)\in Y$ with $u\in U\setminus U'$, then we find a disc $\tilde D$ in $U$ which contains $u$, has $\partial \tilde D\subset U'$, and on which $Z$ is fixed on a corank 1 sublattice $\Gamma_{\rm intermediate}$ of $\Gamma_{\rm local}$ which contains $\Gamma_f$. We now let $\tilde Y$ be the 3-submanifold of $Y$ consisting of those $(u',\theta')\in Y$ with $u'\in \tilde D$ and $\theta'|_{\Gamma_{\rm intermediate}}=\theta|_{\Gamma_{\rm intermediate}}$. The maximum principle for the harmonic function $v^T V v$ on $\tilde Y$ then implies that $v^T V(u,\theta) v > 0$. (Note that any non-compactness of $\tilde Y$ due to loci of the form $\cup_{\gamma\in\tilde\Gamma_{\rm light}}(Z_\gamma^{-1}(0)\times \theta_\gamma^{-1}(0))$ does not cause problems for this argument, as $v^T V v$ diverges to $+\infty$ there. E.g., one can excise neighbhorhoods of these loci in $\tilde Y$ such that $v^T V v>0$ on the new interior boundaries of $\tilde Y$.)

\item  The first equality in \eqref{eq:newV} is trivial. To evaluate the integral in the second line, we first note that $|\X^{\rm sf}_\gamma(\zeta)|$ is bounded above on $\ell_\gamma(u)$ by $e^{-2|Z_\gamma(u)|}<1$. So, the series $\sum_{n>0} (\X_\gamma^{\rm sf}(\zeta))^n$ converges absolutely and uniformly on $\ell_\gamma(u)$ to $\frac{\X^{\rm sf}_\gamma(\zeta)}{1-\X^{\rm sf}_\gamma(\zeta)}$. Since the partial sums of this series, divided by $|\zeta|$, are dominated by the integrable function $\frac{|X^{\rm sf}_\gamma(\zeta)|}{|\zeta|(1-|\X^{\rm sf}_\gamma(\zeta)|)}$, the dominated convergence theorem gives
\be \int_{\ell_\gamma(u)} \frac{d\zeta}{\zeta} \frac{\X^{\rm sf}_\gamma(\zeta)}{1-\X^{\rm sf}_\gamma(\zeta)} = \sum_{n>0} \int_{\ell_\gamma(u)} \frac{d\zeta}{\zeta} (\X^{\rm sf}_\gamma(\zeta))^n = 2 \sum_{n>0} e^{in\theta_\gamma} K_0(2|n Z_\gamma(u)|) \ . \ee
To arrive at the second equality in \eqref{eq:newV} from here, we combine the $\gamma$ and $-\gamma$ sums, using $\Omega(-\gamma) c_{-\gamma,i}c_{-\gamma,j}=\Omega(\gamma) c_{\gamma,i} c_{\gamma,j}$.

\item Valid local coordinates on the fibers of this $U(1)^r$-bundle are given by $\tilde\theta_{\gamma_m}$ on $\pi^{-1}(U\cap\B')$ and by $\tilde\theta'_{\gamma_m}$ on $\pi^{-1}(U)\setminus \cup_{\gamma\in\tilde\Gamma_{\rm light}} \theta_\gamma^{-1}(0)$. (
The `local' part is important: 
an arbitrary choice of lift from $\RR/2\pi\ZZ$ to $\RR$ is necessary.) On $\pi^{-1}(U\cap\B')$, we take our connection to be 
$ p_i(p_i^{-1}d \theta_{\gamma_{m_i}}+A_i)$, where
\begin{align}
A_i =& - \sum_{j=1}^r \Real\tau_{ij} d\theta_{\gamma_{e_j}}\nonumber\\ 
& + \frac{1}{4} \sum_{\gamma'\in\tilde\Gamma_{\rm light}} \Omega(\gamma') p_i^{-1} c_{\gamma',i} d\arg Z_{\gamma'} \sum_{m\in\ZZ} \parens{\frac{2\pi m+\theta_{\gamma'}}{\sqrt{4|Z_{\gamma'}|^2+(2\pi m+\theta_{\gamma'})^2}} - \lambda_{\gamma',m}} \ , \label{eq:ai}
\end{align}
and
\be \lambda_{\gamma,m} = \sgn(\theta_\gamma+2\pi m) + \piecewise{2\parens{m+\frac{\theta_\gamma-\pi}{2\pi}}}{0<m+\frac{\theta_\gamma}{2\pi}<1}{0}{\rm else} \ , \ee
where $\sgn(0)=0$. (One easily verifies that near $2\pi m+\theta_{\gamma'}=0$, $\lambda_{\gamma',m}=2m+\theta_{\gamma'}/\pi$, so $A_i$ is smooth there.) 

On the patch where $\theta'_{\gamma_{m_i}}$ are good coordinates, the connection is $p_i(p_i^{-1}d\theta'_{\gamma_{m_i}}+A'_i)$, where
\begin{align}\label{eq:Aiprime} A'_i =& - \sum_{j=1}^r \Real\tilde\tau_{ij} d\theta_{\gamma_{e_j}}\\ \nonumber
&  + \frac{1}{4}\sum_{\gamma'\in\tilde\Gamma_{\rm light}} \Omega(\gamma') p_i^{-1}c_{\gamma',i} d\arg Z_{\gamma'} \sum_{m\in\ZZ} \parens{\frac{2\pi m+\theta_{\gamma'}}{\sqrt{4|Z_{\gamma'}|^2+(2\pi m +\theta_{\gamma'})^2}} - \sgn(2\pi m+\theta_{\gamma'})} \ .  \end{align}
This is smooth even at $\B'\cap U$ --- the singularities in $A_i$ precisely cancel those in $d\theta_{\gamma_{m_i}}$! (In particular, our earlier assumption that $\theta_{\gamma'}\in (0,2\pi)$ in \eqref{eq:thetaPrime} is crucial here.) These series are uniformly and absolutely convergent on compact subsets of their respective coordinate patches, and may be differentiated arbitrarily many times term-by-term, by the M-test. 

A straightforward computation shows that $dA_i=dA'_i=F_i$.

It is also worthwhile to sketch an alternative, less explicit, proof which illuminates the salient features possessed by $V_{ij}$. We prove that some principal $U(1)^r$-bundle over $Y$ with a connection with curvature $F_i$ exists by showing that $dF_i=0$ on $Y$ and that $[F_i/2\pi]\in H^2(Y,\RR)$ is in the image of $H^2(Y,\ZZ)\to H^2(Y,\RR)$ for all $i$. To do so, we observe that $V_{ij}$ extends as a distribution to $U\times\Theta$ which satisfies the two equations:
\be 
\partial_{\theta_{\gamma_{e_{\ell}}}} V_{ij} = \partial_{\theta_{\gamma_{e_{i}}}} V_{\ell j} \qquad \forall  i, j, \ell
\ee 
and 
\begin{align}& (\partial_{a_k} \partial_{\bar a_\ell}+\partial_{\theta_{\gamma_{e_k}}}\partial_{\theta_{\gamma_{e_\ell}}}) V_{ij} \nonumber \\
&= -\sum_{\gamma\in\tilde\Gamma_{\rm light}} \frac{\Omega(\gamma) p_i^{-1} c_{\gamma,i}p_j^{-1}c_{\gamma,j}p^{-1}_kc_{\gamma,k}p^{-1}_{\ell} c_{\gamma,\ell}}{8|Z_\gamma|} \delta(|Z_\gamma|) \sum_{m\in \ZZ} \delta(\theta_\gamma+2\pi m) \qquad \forall i, j, k, \ell \ . \end{align}
It follows that the current $F_i$ on $U\times \Theta$ satisfies 
\begin{align}
dF_i &= \sum_{j,k,\ell=1}^r 2i (\partial_{a_k} \partial_{\bar a_\ell} + \partial_{\theta_{\gamma_{e_k}}} \partial_{\theta_{\gamma_{e_\ell}}}) V_{ij} \, d\theta_{\gamma_{e_j}} \wedge da_k\wedge d\bar a_\ell \nonumber \\
&= \sum_{\gamma\in\tilde\Gamma_{\rm light}} \frac{\Omega(\gamma) p_i^{-1} c_{\gamma,i} d\theta_\gamma\wedge dZ_\gamma\wedge d\bar Z_\gamma}{4i |Z_\gamma|} \delta(|Z_\gamma|) \sum_{m\in \ZZ} \delta(\theta_\gamma+2\pi m) \nonumber \\
&= \sum_{\gamma\in\tilde\Gamma_{\rm light}} \frac{-\Omega(\gamma) p_i^{-1}c_{\gamma,i} d\theta_\gamma \wedge d|Z_\gamma| \wedge d\arg Z_\gamma}{2} \delta(|Z_\gamma|) \sum_{m\in\ZZ} \delta(\theta_\gamma+2\pi m) \ ,
\end{align}
and the result follows once we observe that the contributions from $\gamma$ and $-\gamma$ are equal. If we wanted to complete this argument, we would show that the first Chern class of this bundle agrees with that of $p: \pi^{-1}(U)\to Y$.
\end{enumerate}
\end{proof}

\subsection{Smooth manifold structure, II}\label{sec:smooth2}
We finish our construction of a smooth manifold $\M$ in Proposition \ref{lem:SmoothII}, making use of  Assumption \ref{it:smooth}. 

\bigskip

As motivation for Proposition \ref{lem:SmoothII}, recall that the Taub-NUT model is a hyper-K\"ahler metric on $\RR^4$ of type ALF; in the Gibbons--Hawking ansatz, it corresponds to the positive harmonic function $V= \frac{1}{4 \pi |y|} + e$ on $(\Imag \HH)_y$ for some choice of $e>0$. In \cite[Example 2.5]{gross:OV}, Gross--Wilson describe the smooth coordinates on Taub-NUT near the preimage of  $y=0 \in \Imag \HH$. An alternate proof of the smoothness is to realize Taub-NUT as a finite-dimensional hyper-K\"ahler quotient at a regular value. This is explained in complete detail in \cite[\S3.1]{gibbons:hkMonopoles}\footnote{We modify their construction slightly. In \cite{gibbons:hkMonopoles}, Taub-NUT is obtained as a hyperk\"ahler quotient of $\HH_q \times \HH_w$ with standard flat Euclidean metric; we identify $w = x + y$ for $x \in \RR$ and $y \in \Imag \HH \simeq \RR^3$, and then instead take the hyperk\"ahler quotient of the $\ZZ$-quotient $\HH_q \times (\RR^3_y \times S^1_x)$.}.  Taub-NUT is obtained as a hyper-K\"ahler quotient of $\HH_q\times (\RR^3_y\times S^1_x)\hqnoxi U(1)$ with $U(1)$ action given by $q\mapsto q e^{i\alpha}$, $y\mapsto y$, $x\mapsto x+\alpha$. The moment map of this $U(1)$ action is $\mu = \half q i \bar q + y$, valued in $\Imag \HH$. 
The hyperk\"ahler quotient is $\HH_q \times (\RR^3_y \times S^1_x) \hqnoxi U(1)=\mu^{-1}(0)/U(1)$, and one can compute (see \cite[(25)]{gibbons:hkMonopoles}) that the induced hyperk\"ahler metric on the quotient is Taub-NUT. As in the usual Gibbons--Hawking presentation, it is apparent that Taub-NUT is fibered over $(\Imag \HH)_y$ with $S^1$ fibers except at $y=0$. However, this is merely a coordinate singularity, and in fact 
$q$ is a good global coordinate; we can solve for $y$ in terms of $q$ as $y=-\half q i \bar q$. 
Writing $q=w_1+w_2 j$, where $w_i\in \CC$, then $q$ is in the fiber over $y$ where
\begin{align}y=
-\half q i \bar q &= -\half (w_1+w_2 j)i(\bar w_1-j\bar w_2) \nonumber \\
&= \underbrace{-\half(|w_1|^2-|w_2|^2)}_\theta i  + \underbrace{i w_1 w_2}_Z j.
\end{align}
Analogues of $\theta, Z$ will appear in the following proposition.

\bigskip

\begin{proposition}\label{lem:SmoothII}
Assume \ref{it:smooth}.
The fibered manifold $\pi: \tilde\M'\to \B$ in Construction \ref{constr:Mtilde} may be extended to another fibered manifold $\pi: \M\to \B$. Over $U$, $\M$ has the structure of a fibered manifold $p: \pi^{-1}(U)\to U\times\Theta$ which extends the  $U(1)^r$ fiber bundle defined in Lemma \ref{lem:gwOV}\ref{it:bundle}.
In particular, the fiber over $(u'', \theta) \in (U, \Theta)$ is $U(1)^{r-|S_{u'', \theta}|}$.
\end{proposition}

\begin{rem}[Orbifolds]
If one dispenses with assumption \ref{it:smooth}, 
then the preceding lemma may be upgraded to define an orbifold $\M$ instead of a manifold. 
The relevant hyper-K\"ahler quotient $H_S$ is the same as in the proof, but now $S$ and $\tilde S$ should be regarded as multisets where the multiplicity of $\gamma$ is $\Omega(\gamma)$, and the action of the copy of $U(1)$ associated to $\gamma\in S$ maps $x_i$ to $x_i+\Omega(\gamma) p_i^{-1}c_{\gamma,i}\alpha_\gamma$.
\end{rem}

\begin{proof}
Fix some $u''\in U'\cap\B''$ and $\theta\in \Theta_{\rm light}$ such that $S_{u'', \theta} \neq \emptyset$. We will describe a smooth structure on an open neighborhood
$Y'$ of $(u'',\theta)$ in $U\times\Theta$.

First, let $U'$ be a neighborhood of $u''$ such that 
\be U' \cap \mathcal{B''} = \{Z_{\gamma}^{-1}(0): \gamma \in \{ \gamma' \in \widetilde{\Gamma}_{\mathrm{light}}: Z_{\gamma'}(u'')=0 \}\}.\ee
Choose $\Lambda_e$ a primitive maximal isotropic sublattice contained in $\widehat{\Gamma}_{\mathrm{local}}$
containing $\widehat{\Gamma}_{\mathrm{light}}$.
Choose some $u' \in U' \cap \mathcal{B}'$.
Let $\gamma_{e_1}, \cdots, \gamma_{e_r}, \gamma_{m_1}, \cdots, \gamma_{m_r}$ be a Frobenius basis of $\Gamma_u$ with $\gamma_{e_i} \in \Lambda_e$ and $d_i = \avg{\gamma_{m_i}, \gamma_{e_i}}$. Then define $Y' = U' \times \Theta$ as in Lemma \ref{lem:gwOV}. 
By Lemma \ref{lem:thetaprime}, for each $\gamma \in \widehat{\Gamma}_{u'}$, $\theta'_{\gamma}$ defines a function from $\pi^{-1}(U \cap \mathcal{B}')\backslash \cup_{\gamma' \in \widetilde{\Gamma}_{\mathrm{light}}} \theta_{\gamma'}^{-1}(0)$ to $\RR/ 2 \pi \ZZ$.

Let $\tilde S_{u'',\theta}$ and $S_{u'',\theta} \subset \widehat{\Gamma}_{\mathrm{light}}$ be as in Assumption \ref{it:smooth}; 
we abbreviate these as $\tilde S=\tilde S_{u'',\theta}$ and $S=S_{u'',\theta}$ and let $s=|S|$. 
By Assumption \ref{it:smooth}, these generate a primitive sublattice of $\Gamma_{\mathrm{light}}$.
By construction, for each $\gamma \in S$, 
\be \gamma = \sum_{i=1}^r p_i^{-1} c_{\gamma,i} \gamma_{e_i},\ee
where $c_{\gamma, i}= \avg{\gamma_{m_i}, \gamma}$ following Definition \ref{def:tau}.

\bigskip

Then, a neighborhood of $p^{-1}(Y')$ will be shown to be diffeomorphic to a neighborhood in the hyper-K\"ahler quotient 
\be H_S = (\prod_{\gamma\in S} \HH) \times (\RR^3\times \RR/2\pi\ZZ)^r\hqnoxi G_S, \qquad \mbox{where $G_S = \prod_{\gamma\in S} U(1)$}.\ee
We parameterize the copy of $\HH$ associated to $\gamma\in S$ by a quaternion $q_\gamma$ and the $i$-th copy of $\RR^3\times \RR/2\pi\ZZ$ by a pair $(y_i,x_i)$, where $y_i$ is an imaginary quaternion and $x_i\in \RR/2\pi\ZZ$.
We take the action of the copy of $U(1)$ associated to $\gamma\in S$ to be  \begin{align} e^{i \alpha_{\gamma}}: q_\gamma &\mapsto q_\gamma e^{i\alpha_\gamma}\nonumber \\ q_{\gamma'} &\mapsto q_{\gamma'} \mbox{ if $\gamma' \neq \gamma$} \nonumber\\
y_i &\mapsto y_i \nonumber \\
x_i&\mapsto x_i+
 \Omega(\gamma) p_i^{-1}c_{\gamma,i} \alpha_\gamma.\end{align}

Note that this action is well-defined, i.e. $\Omega(\gamma) p_i^{-1} c_{\gamma, i}$ is an integer, precisely because our original basis is a Frobenius basis and $\avg{\gamma_{m_i}, }: \Gamma \to p_i \ZZ$ and because $\Omega$ is constant, positive, and integral on $\widetilde{\Gamma}_{\mathrm{light}} \supset \widetilde{S} \supset S$, by Assumption \ref{it:omega}.
This action is triholomorphic and has moment map \be\mu_\gamma = \half q_\gamma i \bar q_\gamma + \sum_i \Omega(\gamma) p_i^{-1} c_{\gamma,i} y_i.\ee Thus, $H_S = \cap_{\gamma\in S} \mu_\gamma^{-1}(0)/U(1)^{|S|}.$\footnote{ 
Physicists may appreciate the following explanation of where this hyper-K\"ahler quotient comes from. Locally, our manifold is the Coulomb branch of a 3d $\N=4$ $U(1)^r$ gauge theory at finite coupling with $|S|$ hypermultiplets such that the $\gamma$-th hyper has charge $c_{\gamma,i}$ under the $i$-th $U(1)$. The UV 3d mirror symmetry of \cite{kapustin:mirror3} shows that this is also the Higgs branch of a 3d $\N=4$ gauge theory with action
\be \sum_{\gamma\in S} \brackets{ S_H(\Q_\gamma, \tilde\V_\gamma) + \sum_i p_i^{-1} c_{\gamma,i} S_{BF}(\V_i, \tilde\V_\gamma) } + \frac{1}{g^2} \sum_i S_V(\V_i) \ , \ee
in the notation of \cite{kapustin:mirror3}. Dualizing the vector multiplets $\V_i$ to hypermultiplets $\Q'_i$ with periodic real parts (if we regard their scalars as quaternions) --- which really means dualizing the $\N=2$ vector multiplets $V_i$ inside of $\V_i$ to chiral multiplets $\chi_i$ with periodic real parts --- yields the hyper-K\"ahler quotient described in the text. It is satisfying to verify that the condition needed for this theory to have no Coulomb branch, namely that the vectors $(p^{-1}c)_\gamma$ be linearly independent, or equivalently that the matrix $p_i^{-1}c_{\gamma,i}$ have rank $|S|$, coincides with the condition for the original theory to have no Higgs branch. For, a $U(1)^{\rank c}$ subgroup of $U(1)^r$ acts faithfully on the hypermultiplets, so that the dimension of the Higgs branch is $4(|S|-\rank c)$.}
 The hyperk\"ahler quotient $H_S$ is independent of the choice of representatives in $S\subset \tilde S$, since if we replace $q_\gamma$ by $q_{-\gamma} = q_\gamma j$ and $\alpha_\gamma$ by $-\alpha_{-\gamma}$ then $q_\gamma\mapsto q_\gamma e^{i\alpha_\gamma}$ becomes $q_{-\gamma} \mapsto q_{-\gamma} e^{i\alpha_{-\gamma}}$, while $x_i\mapsto x_i+\Omega(\gamma)p_i^{-1}c_{\gamma,i}\alpha_\gamma$ becomes $x_i \mapsto x_i + \Omega(-\gamma)p_i^{-1} c_{-\gamma,i} \alpha_{-\gamma}$, noting $\Omega(\gamma)=\Omega(-\gamma)$.

In total, enumerating $S=\{\gamma_\sigma\}_{\sigma=1}^s$ the $G_S$ action is 
\begin{align} e^{i \sum_{\sigma=1}^s\alpha_\sigma}: q_\varsigma &\mapsto q_\varsigma e^{i\alpha_\varsigma}\nonumber \\ 
y_i &\mapsto y_i \nonumber \\
x_i&\mapsto x_i+\sum_{\sigma=1}^{s} \Omega(\gamma_\sigma) p_i^{-1}
\avg{\gamma_{m_i}, \gamma_\sigma} \alpha_{\sigma}.\end{align}
Then $G_S$ acts freely on $\cap_{\gamma \in S} \mu_\gamma^{-1}(0)/U(1)^{|S|}$, thanks to Assumption \ref{it:smooth}. 
In particular, suppose that $(q_{\gamma'}, y_i, x_i) \in \cap_{\gamma \in S} \mu_\gamma^{-1}(0)$ has a non-trivial stabilizer, i.e. there is some element $e^{i \sum_{\sigma=1}^s \alpha_\sigma}$ fixing the point. We will rule out the existence of such an element. Looking at the action on $x_i$, for each $i=1, \cdots, r$
\be\sum_{\sigma=1}^s \Omega(\gamma_\sigma) p_i^{-1} \avg{\gamma_{m_i},\gamma_\sigma} \alpha_\sigma  \in 2 \pi \ZZ.\ee
First, we  
take
\be a_\sigma = \frac{ 2 \pi \prod_{\varsigma=1}^s \Omega(\gamma_\varsigma)}{\Omega(\gamma_\sigma)}.\ee Then this element generates a non-trivial stabilizer, if and only if, there is some element $\varsigma \in \{1, \cdots, s\}$ such that $\Omega(\gamma_\varsigma) \neq 1$. Our Assumption \ref{it:smooth} rules this out; in the rest of this proof, we keep track of the $\Omega(\gamma_\sigma)$, but remind the reader that they are all equal $1$. 
Since the elements of $S$ are linearly independent, the rank of the matrix $\Omega(\gamma_\sigma) p_i^{-1} \avg{\gamma_{m_i}, \gamma_\sigma}$ is $r$ and consequently, the stabilizer of a point doesn't contain a continuous $U(1)$ subgroup. 
Now, suppose 
\be \alpha_\sigma = \frac{2 \pi n_\sigma}{p} \qquad (n_1, \cdots, n_s) \not\equiv (0, \cdots, 0) \in (\ZZ/ p \ZZ)^{s},\ee
i.e. 
$e^{i \sum_{\sigma=1}^s \alpha_\sigma}$ generates a cyclic subgroup of prime order $p>1$. Then for each $i=1, \cdots, r$
\be\sum_{\sigma=1}^s \Omega(\gamma_\sigma) p_i^{-1} \avg{\gamma_{m_i}, \gamma_\sigma} n_\sigma \in  p\ZZ.\ee
Consequently, consider \begin{align}\gamma = \sum_{\sigma=1}^s\sum_{i=1}^{r} \Omega(\gamma_\sigma) p_i^{-1} \avg{\gamma_{m_i}, \gamma_\sigma} \frac{n_\sigma}{p} \gamma_{e_i}.\end{align}
Inside of $\mathrm{Span}(\gamma_1, \cdots, \gamma_s) \otimes_{\ZZ} \QQ,$ $\gamma = \sum_{\sigma=1}^s \Omega(\gamma_\sigma) \frac{n_\sigma}{p} \gamma_\sigma$, but this is not inside of 
$\mathrm{Span}(\gamma_1, \cdots, \gamma_s)$
since $\Omega(\gamma_\sigma) =1$ and the integers $n_\sigma$ are not all divible by $p$.
However, we check that 
\begin{align}
    p \gamma &= \sum_{\sigma=1}^s \sum_{i=1}^r \Omega(\gamma_\sigma) p_i^{-1} \avg{\gamma_{m_i}, \gamma_\sigma} n_\sigma  \gamma_{e_i}\nonumber \\
    &=\sum_{\sigma=1}^s \Omega(\gamma_\sigma)n_\sigma \gamma_s  \in \mathrm{Span}(\gamma_1,\cdots, \gamma_s).
\end{align}
Since the lattice $\mathrm{Span}(\gamma_1, \cdots, \gamma_s)$ is primitive, we have a contradiction.

Thus, $H_S$ is a smooth manifold. 
Indeed, we claim that $H_S$ is diffeomorphic to $\RR^{4|S|}\times (\RR^3\times \RR/2\pi\ZZ)^{r-|S|}$.
We permute the basis elements $\{\gamma_{e_i}\}$ (and correspondingly permute the basis elements $\{\gamma_{m_i}\}$)  
so that $M$, the first $s$ columns of the $s \times r$ matrix $\widetilde{M}=\Omega(\gamma_\sigma) p_i^{-1} c_{\sigma, i}$,
is invertible (hence $\det M = \pm 1$). This is possible since $\mathrm{Span}(\gamma_1, \cdots, \gamma_s)$ is primitive.  
Write $\tilde{M} = (M |N)$, where $N$ is $s \times r-s$ matrix of the last $r-s$ columns of $\tilde{M}$.

It will be convenient to introduce coordinates $y_{\gamma} =-\frac{1}{2} q_\gamma i \bar q_\gamma$ for $\gamma \in S$, since these coordinates naturally appear in the moment map.
Then, the moment maps  $-\sum_{i=1}^s \Omega(\gamma_\sigma) p_i^{-1}c_{\gamma_\sigma,i} y_{i} = y_{\gamma_\sigma} + \sum_{i=s+1}^r \Omega(\gamma_\sigma) p_i^{-1} c_{\gamma_\sigma,i} y_i$ allow us to solve for $y_1, \cdots, y_s$ in terms of the other variables $q_{\gamma_1}, \cdots, q_{\gamma_s}$ (via $y_{\gamma_1}, \cdots y_{\gamma_s})$) and $ y_{s+1}, \cdots, y_r$ as 
\be \begin{pmatrix} y_1 \\ \vdots \\ y_s \end{pmatrix} = M^{-1} \left( \begin{pmatrix}y_{\gamma_1} \\ \vdots \\ y_{\gamma_s} \end{pmatrix} - N \begin{pmatrix} y_{s+1} \\ \vdots \\ y_{r} \end{pmatrix} \right). \label{eq:y} \ee
Furthermore, we can use the $U(1)^{|S|}$ freedom to set $x_{1}, \cdots, x_s=0$.

\bigskip
On the other hand, $H_S$ admits a triholomorphic $U(1)^r$ action.  The $i$-th factor implements $x_i\mapsto x_i+\beta_i$ while leaving all other coordinates invariant and has moment map $y_i$. (One can check that $e^{i \beta_i}$ preserves $\mu_{\gamma}^{-1}(0)$ and that this action commutes with the residual $G_S$ action on $\mu_{\gamma}^{-1}(0)$, hence descends to a $U(1)$ action on $H_S$. Note that on $H_S$ in the coordinates $q_{\gamma_1}, \cdots, q_{\gamma_s}, y_{s+1}, \cdots, y_r, x_{s+1}, \cdots, x_r$ this $U(1)^r$ action will be non-trivial on $q_{\gamma_\sigma}$.)  It follows (see, e.g., \cite{gibbons:hkMonopoles}) that $H_S$ admits the structure of a fibered manifold $H_S\to  \Imag(\HH)^r$, where the base is parametrized by the moment maps $y_1, \cdots, y_r$, which is a principal $U(1)^r$-bundle on the complement of the locus with a non-trivial stabilizer in $U(1)^r$. Explicitly, this locus is where one or more of the coordinates $q_{\gamma}$ vanishes, i.e.  one of more of the functions $\sum_i \Omega(\gamma) p_i^{-1} c_{\gamma,i} y_i$ vanishes. 
\bigskip

Finally, having studied the geometry of $H_S$, we return to the manifold of interest, and show that $\M$ extends $\widetilde{\M'}$. Near the locus in $H_S$ where the principal bundle structure degenerates, we identify the coordinates \be y_1=\theta_{\gamma_{e_1}}i+2Z_{\gamma_{e_1}}j, \qquad \ldots, \qquad y_r=\theta_{\gamma_{e_r}}i+2Z_{\gamma_{e_r}}j\ee
and the coordinates \be x_i=\tilde\theta_{\gamma_{m_i}}'.\ee
(The factor of $2=2 \pi R$ appears due to our choice of $R= \frac{1}{\pi}$.)
\end{proof}

\begin{lemma}\label{lem:Jacobians}
  The $1$-forms $\{(da_i,  d \bar a_i, d\theta_{\gamma_{e_i}})\}_{i=1, \cdots r}$ are smooth on $\M$. 
Moreover, the absolute value of the Jacobian determinant is 
\be
\abs{\frac{\partial( \{ a_{i}\}_{i=1}^r,\{ \bar a_{i}\}_{i=1}^r,  \{ \theta_{e_i}\}_{i=1}^r, \{ \tilde\theta'_{m_i}\}_{i=1}^r)}{\partial(\{(w_{\gamma_i,1},\bar w_{\gamma_i, 1} w_{\gamma_i, 2}, \bar w_{\gamma_i, 2})\}_{i=1}^s, \{ a_{i}\}_{i=s+1}^r,\{ \bar a_{i}\}_{i=s+1}^r,  \{ \theta_{e_i}\}_{i=s+1}^r, \{ \tilde\theta'_{m_i}\}_{i=s+1}^r)}}= \prod_{i=1}^s\frac{1}{8} |q_{\gamma_i}|^2 
\ee
for $q_\gamma = w_{\gamma, 1} + w_{\gamma, 2} j$.
\end{lemma}
\begin{proof}
    Fix $u'' \subset U \cap \B'', \theta \in \Theta_{\rm light}$ such that $S_{u'', \theta} \neq \emptyset$. Let 
    $Y'$ be the open neighborhood of $(u'', \theta)$ in $U \times \Theta$, as in the previous proposition, and consider $p^{-1}(Y')$.
We relate the smooth coordinates on $p^{-1}(Y') \subset \M$ to the smooth coordinates on $\tilde\M'$.
The usual smooth coordinates on $\tilde \M'$
\be \label{eq:coords1}\{(a_i,  \theta_{\gamma_{e_i}}, \tilde\theta'_{\gamma_{m_i}})\}_{i=1, \cdots r}\ee
are related to Gibbons-Hawking-Ansatz-like coordinates $\{(x_i, y_i)\}_{i=1, \cdots, r}$ by 
$y_i = \theta_{\gamma_{e_i}}i + 2a_i j$ and $x_i=\tilde \theta'_{\gamma_{m_i}}$.
The hyper-K\"ahler quotient in Propositon \ref{lem:SmoothII} lets us write the coordinates $y_1, \cdots, y_s, x_1, \cdots, x_s$ in terms of terms of $\{q_{\gamma_1}, \cdots, q_{\gamma_s},y_{s+1}, \ldots, y_r, x_{s+1}, \ldots, x_r\}.$
Our preferred smooth coordinates on $p^{-1}(Y') \subset \M$ will be
\be \{w_{\gamma_1, 1}, \cdots, w_{\gamma_s,1},w_{\gamma_1, 2}, \cdots, w_{\gamma_s,2}, a_{s+1}, \ldots, a_r, \theta_{\gamma_{e_{s+1}}}, \ldots, \theta_{\gamma_{e_r}}, \tilde \theta'_{\gamma_{m_{s+1}}}, \ldots, \tilde \theta'_{\gamma_{m_r}}\},\ee
using the following identification of $\HH_{q_\gamma} \simeq \CC_{w_{\gamma,1}} \times \CC_{w_{\gamma_2}}$ via 
$q_\gamma = w_{\gamma, 1} + w_{\gamma, 2}j$.

\medskip

It is convenient to introduce intermediate coordinates \begin{align} y_{\gamma} &= \theta_{\gamma} i + 2Z_{\gamma} j = -\frac{1}{2} q_{\gamma} i \bar q_{\gamma} \nonumber \\
\chi_\gamma &= - \mathrm{arg}(w_{\gamma,2})=-\frac{1}{2i} \log \left( \frac{w_{\gamma_\sigma,2}}{\bar w_{\gamma_\sigma, 2}}\right)
\label{eq:gammacoords}
\end{align}
 In particular, $\theta_\gamma, Z_\gamma$ are related to $w_{\gamma,1}, w_{\gamma, 2}$ by 
\begin{equation} \label{eq:gammacoords}
\theta_{\gamma} = - \frac{1}{2} \left(|w_{\gamma,1}|^2 - |w_{\gamma,2}|^2 \right), \qquad  \qquad Z_{\gamma} = \frac{i}{2} w_{\gamma,1} w_{\gamma,2}.
\end{equation}
These coordinates $y_\gamma, \chi_\gamma$ are related to coordinates on $\tilde \M'$ via \eqref{eq:y}, while
\begin{equation}\label{eq:triv1totriv2}
    \begin{pmatrix}\chi_{\gamma_1} \\ \vdots \\ \chi_{\gamma_s} \end{pmatrix}  = -(M^T)^{-1} \begin{pmatrix} x_1 \\ \vdots \\ x_s \end{pmatrix}.
    \end{equation}
    To see \eqref{eq:triv1totriv2}, consider the coordinates of the $U(1)^r$ bundle. 
Trivialization \#1 is given by setting $\tilde \theta_{m_i}'=0$ for $i=1, \cdots, s$ and 
varying $\tilde \theta_{m_i}'$ for $i=s+1, \cdots, r$ and $\chi_{\gamma_\sigma}$ for $\sigma=1, \cdots, s$;
trivialization \#2 is given by setting $\chi_{\gamma_\sigma} =0$ for $\sigma=1, \cdots, s$, and 
varying $\tilde \theta_{m_i}'$ for $i=1, \cdots, r$.
To go back and forth between these two trivializations we observe that the $U(1)^{|S|}$ action on $w_{\gamma_\sigma, 2}$ is $w_{\gamma_\sigma, 2} \mapsto w_{\gamma_\sigma, 2} e^{-i \alpha_{\gamma_\sigma}}$, i.e. $\chi_{\gamma_\sigma} \mapsto \chi_{\gamma_\sigma} + \alpha_{\gamma_\sigma}$. The following $s$ quantities are preserved by the $U(1)^{|S|}$ hence give coordinates on $H_S$:
\begin{equation}
\begin{pmatrix}\chi_{\gamma_1} \\ \vdots \\ \chi_{\gamma_s} \end{pmatrix} -(M^T)^{-1} \begin{pmatrix} x_1 \\ \vdots \\ x_s \end{pmatrix}.
\end{equation}
Thus, to translate between the two trivializations, we take \eqref{eq:triv1totriv2}.

\medskip

With all this, we can compute the relationship between the $1$-forms and the Jacobian determinants.

Differentiating \eqref{eq:gammacoords}, we have 
\begin{equation}
    \begin{pmatrix} dZ_{\gamma} \\ d\bar Z_{\gamma} \\ d\theta_{\gamma} \\d \chi_{\gamma} \end{pmatrix}= \underbrace{\begin{pmatrix}  \frac{i}{2} w_{\gamma, 2} & 0 & \frac{i}{2} w_{\gamma,1} & 0 \\
    0 & -\frac{i}{2} \bar w_{\gamma,2}  & 0 & -\frac{i}{2} \bar w_{\gamma, 1} \\ - \frac{1}{2} \bar w_{\gamma, 1} & -\frac{1}{2} w_{\gamma, 1} & \frac{1}{2} \bar w_{\gamma, 2} &\frac{1}{2} w_{\gamma,2}\\ 0 & 0  & - \frac{1}{2i} \frac{1}{w_{\gamma,2}} & \frac{1}{2i} \frac{1}{\bar w_{\gamma,2}} \end{pmatrix}}_{P_\gamma} \begin{pmatrix}dw_{\gamma,1}\\ d\bar w_{\gamma,1}\\dw_{\gamma,2}\\d\bar w_{\gamma,2} \end{pmatrix}
\end{equation}
We observe that $\det P_\gamma =-\frac{i}{8}(|w_{\gamma,1}|^2 + |w_{\gamma, 2}|^2)=-\frac{i}{8}|q_\gamma|^2$. Thus, 
\be
\abs{\frac{\partial(\{(Z_{\gamma_i}, \bar Z_{\gamma_i}, \theta_{\gamma_i}, \chi_{\gamma_i})\}_{i=1}^s, \{ a_{i}\}_{i=s+1}^r,\{ \bar a_{i}\}_{i=s+1}^r,  \{ \theta_{e_i}\}_{i=s+1}^r, \{ \tilde\theta'_{m_i}\}_{i=s+1}^r)}{\partial(\{(w_{\gamma_i,1},\bar w_{\gamma_i, 1} w_{\gamma_i, 2}, \bar w_{\gamma_i, 2})\}_{i=1}^s, \{ a_{i}\}_{i=s+1}^r,\{ \bar a_{i}\}_{i=s+1}^r,  \{ \theta_{e_i}\}_{i=s+1}^r, \{ \tilde\theta'_{m_i}\}_{i=s+1}^r)}}= \prod_{i=1}^s \det P_{\gamma_i}.
\ee

Differentiating \eqref{eq:y} and \eqref{eq:triv1totriv2}, we see that the coordinates 
\be \{\{(\theta_{\gamma_i}, Z_{\gamma_i}, \chi_{\gamma_i})\}_{i=1}^s, a_{s+1}, \ldots, a_r, \theta_{\gamma_{e_{s+1}}}, \ldots, \theta_{\gamma_{e_r}}, \tilde \theta'_{\gamma_{m_{s+1}}}, \ldots, \tilde \theta'_{\gamma_{m_r}}\},\ee
are related to the coordinates \eqref{eq:coords1} by an affine transformation. Moreover, since
\begin{align} \abs{\frac{\partial(a_1, \cdots, a_r)}{\partial(Z_{\gamma_1}, \cdots Z_{\gamma_s}, a_{s+1}, \cdots a_r)}} &= \det \begin{pmatrix}M^{-1} & -M^{-1} N \\ 0 & I_{r-s} \end{pmatrix}= \det(M^{-1}),\\ \nonumber
   \abs{\frac{\partial( \theta_{e_1}, \cdots, \theta_{e_r})}{\partial(\theta_{\gamma_1}, \cdots, \theta_{\gamma_s}, \theta_{e_{s+1}}, \cdots, \theta_{e_r})}} &= \det \begin{pmatrix}M^{-1} & -M^{-1} N \\ 0 & I_{r-s} \end{pmatrix}= \det(M^{-1}),\\ \nonumber 
    \abs{\frac{\partial( \tilde \theta'_{m_1}, \cdots, \tilde \theta'_{m_r})}{\partial(\chi_{\gamma_1}, \cdots, \chi_{\gamma_s}, \tilde \theta'_{m_{s+1}}, \cdots, \tilde \theta'_{m_r})}} &= \det\begin{pmatrix}-M^T & 0 \\ 0 & I_{r-s} \end{pmatrix}= \det(M^{-1}),
\end{align}
and $\det M = \pm 1$, the Jacobian determinant is simply
\be
\abs{\frac{\partial(\{a_i\}_{i=1}^r, \{\bar a_i\}_{i=1}^r, \{\theta_{e_i}\}_{i=1}^r, \{\tilde \theta'_{m_i}\}_{i=1}^r)}{\partial(\{Z_{\gamma_i}\}_{i=1}^s, \{ a_{i}\}_{i=s+1}^r, \{\bar Z_{\gamma_i}\}_{i=1}^s, \{ \bar a_{i}\}_{i=s+1}^r,  \{\theta_{\gamma_i }\}_{i=1}^s, \{ \theta_{e_i}\}_{i=s+1}^r,  \{\chi_{\gamma_i }\}_{i=1}^s, \{ \tilde\theta'_{m_i}\}_{i=s+1}^r)}}= 1.
\ee
\end{proof}

\begin{rem}[Smoothness of Vector Fields]\label{rem:vectorfields}
Note that the vector fields are related by the transpose of the inverse so that 
\begin{equation}
    \begin{pmatrix}\partial_{w_{\gamma,1}}\\ \partial_{\bar w_{\gamma,1}}\\\partial_{ w_{\gamma,2}}\\ \partial_{\bar w_{\gamma,2}} \end{pmatrix}
    = \begin{pmatrix} \frac{i}{2} w_{\gamma, 2} & 0 & - \frac{1}{2} \bar w_{\gamma, 1} \\ 0 & -\frac{i}{2} \bar w_{\gamma, 2} & - \frac{1}{2} w_{\gamma,1} & 0 \\ \frac{i}{2} w_{\gamma, 1} & 0 & \frac{1}{2} \bar w_{\gamma,2} & - \frac{1}{2} \frac{1}{w_{\gamma,2}} \\ 0 & -\frac{i}{2} \bar w_{\gamma,1} & \frac{1}{2} w_{\gamma, 2} & \frac{1}{2} \frac{1}{\bar w_{\gamma,2}} \end{pmatrix}\begin{pmatrix} \partial_{Z_{\gamma}}  \\ \partial_{\bar Z_{\gamma}} \\ \partial_{\theta_{\gamma}} \\\partial_{ \chi_{\gamma}} .\end{pmatrix}
\end{equation}
\end{rem}

\begin{proposition} \label{prop:TN} Fix some $u''\in U'\cap\B''$ and $\theta\in \Theta_{\rm light}$ such that $S_{u'', \theta} \neq \emptyset$.  Take the  standard hyper-K\"ahler metric on  $\HH^s \times (\RR^3 \times \RR/ 2 \pi \ZZ)^r$.
In the coordinates, $\theta_{\gamma_{e_i}}, a_i=Z_{\gamma_{e_i}}, \widetilde{\theta}'_{\gamma_{m_i}}$ on an open neighborhood
$Y'$ of $(u'',\theta)$ in $U\times\Theta$, the induced hyper-K\"ahler structure has associated holomorphic symplectic forms 
\begin{align}
    \varpi^{\mathrm{TN}}(\zeta) &= - \frac{i}{2 \zeta} \omega_+^{\mathrm{TN}} + \omega_3^{\mathrm{TN}} - \frac{i}{2}\zeta \omega_-^{\mathrm{TN}}
    \end{align}
where\footnote{\label{footnote:TNholcoords}
It is straightforward to check that 
\[\omega_+^{\rm TN} = \sum_{i=1}^r da_i \wedge dz_i^{\rm TN}\]
for 
\[z^{\rm TN}_i = \frac{1}{2 \pi} \left( \theta'_{\gamma_{m_i}} + \theta_{\gamma_{e_i}}+ \sum_{\gamma \in S} \left(\frac{p_i^{-1} c_{\gamma, i}}{2i} \mathrm{arctanh} \left(\frac{\theta_\gamma}{\sqrt{4 Z_\gamma \bar Z_\gamma + \theta_\gamma^2}}\right) + \frac{1}{4i} d \log \bar Z_\gamma \right)\right).\]
In this computation, we have chosen $z^{\rm TN}_i$ so that the $dz_i^{\rm TN}$ has the correct coefficients for the $1$-forms $d\theta_{\gamma_{e_j}}$ and $d \bar Z_{\gamma_{e_j}}$; however, we do not match the coefficients for $d Z_{\gamma_{e_j}}$; instead we use that $\sum_{i,j=1}^r f_{ij} da_i \wedge da_j=0$ if $f_{ij}=f_{ji}$.
}
\begin{align}
\omega^{\rm TN}_+ &= \frac{1}{2\pi} \sum_{i=1}^r da_i\wedge (p_i^{-1} 
d\theta'_{\gamma_{m_i}}+{A'_i}^{\rm TN}-i\sum_j V^{\rm TN}_{ij} d\theta_{\gamma_{e_j}}) \ , \nonumber \\
\omega_3^{\rm TN} &= \frac{i}{2\pi} \sum_{i,j} V^{\rm TN}_{ij} da_i\wedge d\bar a_j + \frac{1}{4\pi} \sum_i d\theta_{\gamma_{e_i}}\wedge (p_i^{-1}
d\theta'_{\gamma_{m_i}}+{A'_i}^{\rm TN}) \ , \nonumber \\
\omega_-^{\rm TN} &= \overline{\omega_+^{\rm TN}} \ ,
\end{align}
where 
\begin{align}
    V^{\mathrm{TN}}_{ij} &=  
    \delta_{ij} + \sum_{\gamma \in S} \frac{1}{|q_\gamma|^2} p^{-1}_i p^{-1}_j c_{\gamma, i} c_{\gamma, j} \\ \nonumber 
   {A'_i}^{\mathrm{TN}}&=  
    \sum_{\gamma' \in S} p_i^{-1} c_{\gamma', i} \frac{1}{2} d\arg Z_{\gamma'}\left( \frac{\theta_{\gamma'}}{\sqrt{4|Z_{\gamma'}|^2 + \theta_{\gamma'}^2}} -1 \right).
\end{align}
(Here, $\mathrm{TN}$ is used since this generalizes the Taub-NUT geometry.)
\end{proposition}

\begin{proof}

   Take the standard hyper-K\"ahler metric on  $\HH^s \times (\RR^3 \times \RR/ 2 \pi \ZZ)^r$, corresponding to the 
 the standard K\"ahler forms
\be\frac{1}{4\pi} \left(\sum_{\gamma \in S} -\frac{1}{2} dq_\gamma \wedge d \bar q_{\gamma} + \sum_{i=1}^r  -\frac{1}{2} d(x_i + y_i) \wedge d \overline{(x_i + y_i)} \right)= i \omega_I + j \omega_J + j \omega_K.\ee
The induced K\"ahler forms on $H_s$ at $(q_\gamma, y_i, x_i)$ are simply obtained by restriction\footnote{See the extended discussion around \cite[(3.35)]{mz:K3HK}. Since $d \mu_{\gamma} = \iota_{X_{\gamma}} (i \omega_I + j \omega_J + k \omega_K)$, on $\cap^{-1}\mu^{-1}_{\gamma}(0)$, any vector $X_{\gamma}$ is in the kernel of $\omega_I, \omega_J, \omega_K$. Consequently, while one does need to find representatives of the tangent space that are orthogonal to the $G$-orbit when computing the induced metric in a hyper-K\"ahler quotient; the same caution is unneccessary for the induced symplectic forms.}. 
We note a few preliminary facts:
First, we observe that 
\be -\frac{1}{2} d(x_i + y_i) \wedge d(x_i-y_i) =dx_i \wedge (dy_{i,1} i + dy_{i,2} j + dy_{i,3} k) + \star dy_{i, 1} i + \star dy_{i, 2} j +\star dy_{i, 3} k,\ee
hence
\begin{align*}
    - \frac{1}{2} d(x_i + y_i) \wedge d(x_i-y_i) &=  \underbrace{\left(dx_i \wedge d \theta_{e_i} + 4 \cdot \frac{i}{2} da_i \wedge d\bar a_i\right)}_{\omega_{I, i}} i +  \underbrace{\left(2dx_i \wedge da_i  + 2i d \theta_{e_i} \wedge da_i \right)}_{\Omega_{I, i}} j;
\end{align*}
also
\[-\frac{1}{2} dq_\gamma \wedge d \bar q_{\gamma} = \underbrace{\frac{i}{2}  \left(dw_{\gamma, 1} \wedge d \bar w_{\gamma, 1} +  dw_{\gamma, 2} \wedge d \bar w_{\gamma, 2}\right)}_{\omega_{I, \gamma}} i +  \underbrace{dw_{\gamma, 1} \wedge dw_{\gamma, 2}}_{\Omega_{I, \gamma}} j 
\]
where 
a tedious computation gives the expected shape 
\begin{align*}
\frac{i}{2} \left( dw_{\gamma, 1} \wedge d \bar w_{\gamma, 1} + d w_{\gamma, 2} \wedge d \bar w_{\gamma,2} \right) &=   
V_\gamma \underbrace{\frac{i}{2} dZ_{\gamma} \wedge d\bar Z_{\gamma}}_{\star d \theta_{\gamma}} + 
\Theta_\gamma \wedge d \theta_\gamma\\
    dw_{\gamma, 1} \wedge d w_{\gamma, 2}  
    &= \left(i \frac{d \theta_\gamma}{|q_\gamma|^2} + \Theta_\gamma \right) \wedge dZ_\gamma
\end{align*}
in terms of the 4d Gibbons-Hawking potential and connection:
\begin{align}
    V_\gamma&=\frac{1}{|q_\gamma|^2}=\frac{1}{2 |y_\gamma|} \nonumber\\
    \Theta_\gamma&=  
    d \chi_\gamma + \frac{|w_{\gamma, 1}|^2}{|q_\gamma|^2} d \mathrm{arg} Z_{\gamma} = d \chi_\gamma + \frac{1}{2}\left(-\frac{\theta}{|y_{\gamma}|} +1 \right)d \mathrm{arg} Z_{\gamma}.
\end{align}
Having done that, we see that 
\begin{align*}
    \omega_I &= \frac{1}{4 \pi} \left( \sum_{\gamma \in S} \frac{1}{|q_\gamma|^2} 4 \cdot \frac{i}{2} dZ_{\gamma} \wedge d\bar Z_{\gamma} +\Theta_{\gamma_i}  \wedge d \theta_\gamma  + \sum_{i=1}^r dx_i \wedge d \theta_{e_i} + 4 \cdot \frac{i}{2} da_i \wedge d \bar a_i  \right)\\
    &= \frac{1}{4\pi} \left(\sum_{\gamma \in S} \frac{1}{|q_\gamma|^2} \sum_{ij} p^{-1}_i p^{-1}_j c_{\gamma,i} c_{\gamma,j} \frac{i}{2} \cdot 4 da_i \wedge d \bar a_j + \Theta_\gamma \wedge \left(\sum_{i=1}^r p_i^{-1} c_{\gamma, i} d \theta_{e_i} \right) \right.\\
    & \left.\qquad + \sum_{i=1}^r dx_i \wedge d \theta_{e_i} + \frac{i}{2} \cdot 4 da_i \wedge d \bar a_i \right)\\
    & = \frac{i}{2 \pi} \sum_{ij} V_{ij}da_i \wedge d \bar a_j + \frac{1}{4\pi} \sum_i d \theta_{e_i} \wedge (d \theta'_{\gamma_{m_i}} + A'_i) \\
     \Omega_I &= \frac{1}{4 \pi} \left( \sum_{\gamma \in S} \left(i \frac{d \theta_\gamma}{|q_\gamma|^2} + \Theta \right) \wedge 2 dZ_\gamma+ \sum_{i=1}^r (dx_i + i d \theta_{e_i})\wedge 2 da_i \right)  \\
     &= \frac{1}{4 \pi} \left( \sum_{\gamma \in S} \left(i \frac{d \theta_\gamma}{|q_\gamma|^2} + \Theta \right) \wedge  \left(\sum_{i=1}^r p_i^{-1} c_{\gamma, i} 2 da_i \right)+ \sum_{i=1}^r (dx_i + i d \theta_{e_i})\wedge 2  da_i  \right) \\
     &= \frac{1}{2 \pi} \sum_{i=1}^r  da_i \wedge \left(d\widetilde{\theta}'_{m_i} + {A_i'}^{\rm TN} -i \sum_{j=1}^r V^{\rm TN}_{ij} d \theta_{e_i} \right). 
\end{align*}
where the potential and connections $d \theta_{m_i} + A_i = d \tilde{\theta}_{m_i}'$ are
\begin{align*}
    V^{\rm TN}_{ij} &= \delta_{ij} + \sum_{\gamma \in S} \frac{1}{2\sqrt{4|Z_{\gamma}^2 + \theta_\gamma^2}} p^{-1}_i p^{-1}_j c_{\gamma, i} c_{\gamma, j} \\
   {A'_i}^{\rm TN}&=   \sum_{\gamma' \in S} p_i^{-1} c_{\gamma', i} \frac{1}{2} d\arg Z_{\gamma'}\left( \frac{\theta_{\gamma'}}{\sqrt{4|Z_{\gamma'}|^2 + \theta_{\gamma'}^2}} -1 \right).
\end{align*}
as claimed.
\end{proof}

\section{Model geometry}\label{sec:modelgeom}
We now proceed to describe the relevant model geometry associated to all of the above data. This was first constructed by Seiberg and Shenker in \cite{seiberg:mirrorT}; a special case was determined from its extra symmetries by Ooguri and Vafa in \cite{vafa:spacetimeInsts} and was studied mathematically rigorously in 
\cite[\S3]{gross:OV}. We will first describe these Seiberg-Shenker geometries using the technology of \cite{GMN:walls}, as opposed to the Gibbons-Hawking ansatz \cite{hawking:multiAn,rocek:tensorDuality,poon:hk,anderson:gh,goto:hk,gibbons:hkMonopoles} as in \cite{gross:OV}.

We begin with a definition:

\begin{figure}[h!]
\begin{centering} 
\includegraphics[height=1.5in]{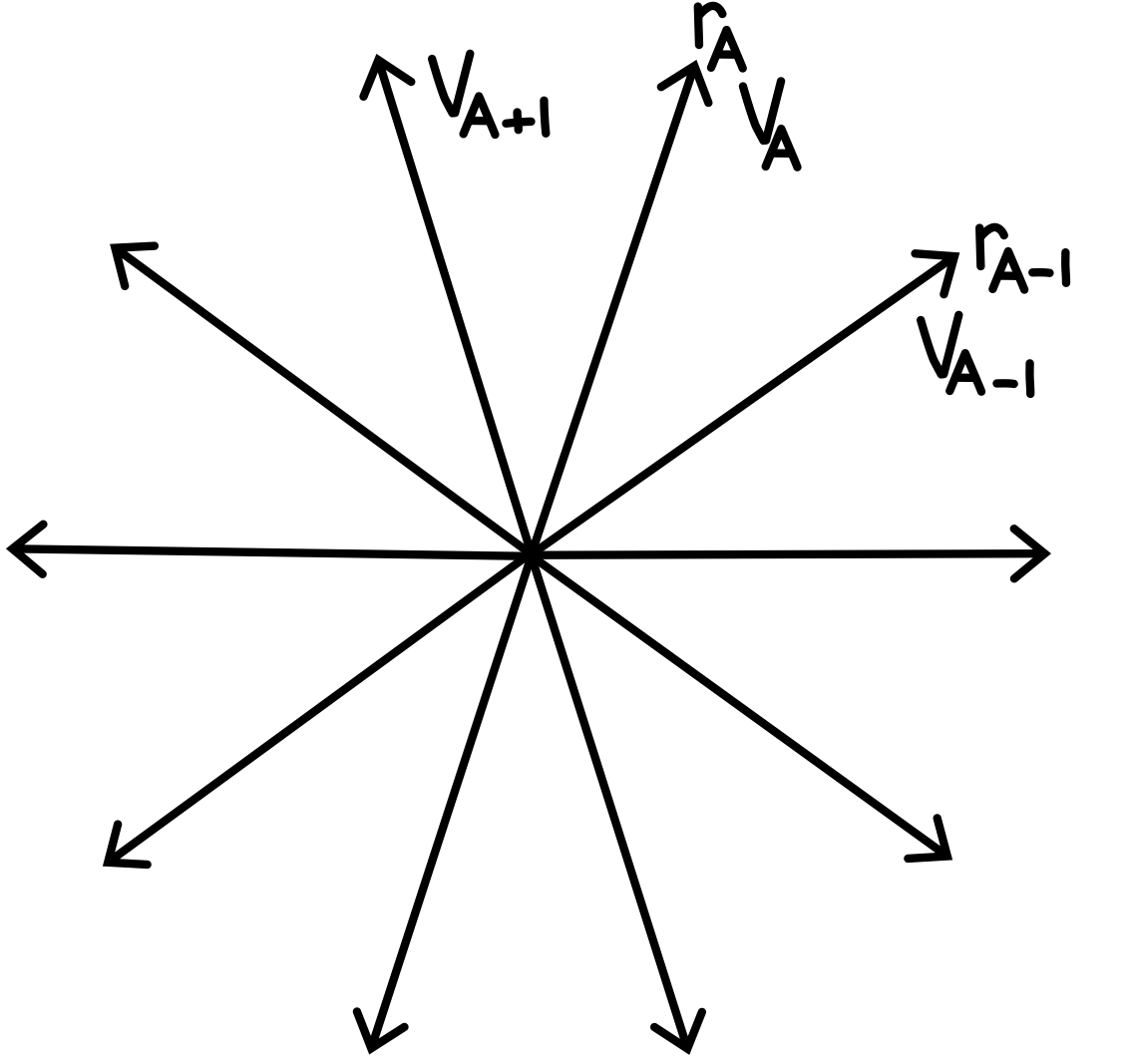} \qquad \qquad  \qquad 
\includegraphics[height=1.5in]{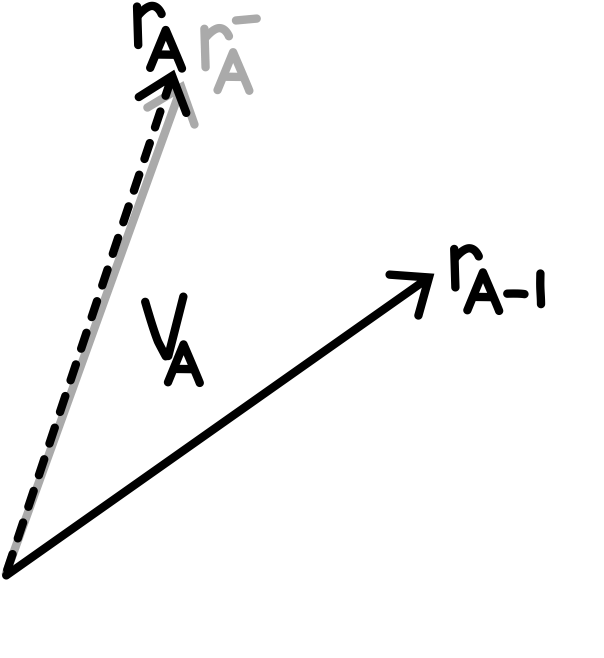}
\caption{\label{fig:sectors}\textsc{(Left)} Sectorial decomposition $\mathcal{V}=\{V_A\}_{A=1}^{2K}$ of $\CC_\zeta^\times$ with rays $r_A$ labelled \textsc{(Right)} $V_A$ contains $r_{A-1}$ but not $r_A$}
\end{centering}
\end{figure}

\begin{defn}
A \emph{sectorial decomposition} (see Figure \ref{fig:sectors}) $\mathcal{V} = \{V_A\}_{A=1}^{2K}$ of $\mb{C}^{\times}_{\zeta}$ is a partition of $\mb{C}^{\times}_{\zeta}$ into $2K$ disjoint half-open sectors ordered counterclockwise and all of opening angle $\phi = \pi/K$, with $K\ge 5$. 
\end{defn}

\noindent\underline{Notation:} Let $r_A$ be the ray separating $V_A$ from $V_{A+1}$, oriented from $0$ to $\infty$; we take $V_A$ to contain $r_{A-1}$ but not $r_{A}$ (see Figure \ref{fig:sectors}).
We also introduce the notation $r_A^-$ to denote a ray infinitesimally clockwise of $r_A$, which we shall use as a contour of integration (see Figure \ref{fig:sectors}). We consider the indices modulo $2K$ so that, e.g., $V_{A+K} = -V_A$. Given $\zeta\in\CC^\times$, $A(\zeta)$ is the index $A$ for which $\zeta\in V_A$.
Similarly, given $\gamma\in\tilde\Gamma$, $A(\gamma)$ is the index $A$ for which $\ell_{\gamma}(u) \subset V_A$. Finally, define
\be \label{eq:Vopen}\mc{V}^{\circ} := \Big(\mb{C}^{\times} \setminus \bigcup_A r_A\Big) \ . \ee

\begin{defn}
Next, given a point $u' \in \mc{B}'$ and a subset $\tilde\Gamma\subset \hat\Gamma_u$, we say that a sectorial decomposition is $(u',\tilde\Gamma)$-\emph{good} if for no $\gamma\in\tilde\Gamma$ does some ray $r_A$ coincide with the ray $\ell_{\gamma}(u)=-Z_{\gamma}(u)\RR^{+}$ defined in Lemma \ref{lem:gwOV}.\end{defn}

\bigskip

Fix $u \in \B''$ and an open set $U$ as in \ref{it:A1}. 
We now cover $U\cap\B'$ by finitely many\footnote{We need at most $(2 \times 2K)^{|\widetilde{\Gamma}_{\mathrm{light}}|/2}$ open sets. Fix a sectorial decomposition $\mathcal{V}_1$ and another sectorial decomposition $\mathcal{V}_2$ obtained by a small rotation of $\mathcal{V}_1$ so that $\mathcal{V}_1^{\circ} \cup \mathcal{V}_2^\circ$ cover $\CC^\times_\zeta$. Track the finitely many $\gamma \in \widetilde{\Gamma}_{\mathrm{light}}$ (up to $\pm$) with respect to these sectors.  
}
contractible open neighborhoods $U_a$ for which there exists a sectorial decomposition $\V_a=\{V_{a,A}\}$ with boundary rays $\{r_{a,A}\}$ which is $(u,\tilde\Gamma_{\rm light})$-good for all $u\in U_a$. 
On $\pi^{-1}(U_a)$, we define the following twisted character:
\begin{defn}[$\X^{{\rm model},a}$]\label{defn:Xmodel} For each $\zeta \in \CC^\times$ and $m \in \M'$, define 
\be \X^{{\rm model},a}_\gamma(\zeta) = \X^{\rm sf}_\gamma(\zeta) \exp\parens{-\frac{1}{4\pi i} \sum_{\gamma'\in \tilde\Gamma_{\rm light}} \Omega(\gamma') \avg{\gamma,\gamma'} \int_{r^-_{a,A(\gamma')}} \frac{d\zeta'}{\zeta'} \frac{\zeta'+\zeta}{\zeta'-\zeta} \log(1-\X^{\rm sf}_{\gamma'}(\zeta')) } \ . \ee
The minus superscript handles the possibility that $\zeta\in r_{a,A(\gamma')}$.
\end{defn}

\begin{rem}
Note that  $\X^{{\rm model},a}_\gamma(\zeta) = \X^{\rm sf}_\gamma(\zeta)$ if $\gamma\in \Gamma_{\rm local}$.
\end{rem}

\begin{rem}[Riemann-Hilbert Problem]
As discussed in \S\ref{sec:strategy}, the goal of this integral relation is to produce a map  $\mathcal{X}: \M'_{U_a} \times \CC^\times \to \mathcal{T}'_{U_a}$ with the following properties:
\begin{enumerate}[(1)]
\item $\X$ depends piecewise holomorphically on $\zeta \in \CC^\times$, with discontinuities at a subset of the boundary rays $\cup_{A=1}^{2K} r_{a,A}$  of the sectorial decomposition $\mathcal{V}_a$. In particular note that the value of $\X$ at $r_A$ is the limit from the clockwise adjacent $V_{A+1}$ sector. 
\item The limits $\X^\pm$ of $\X$ as $\zeta$ approaches $r_A$ from both sides exist and are related by $\X^{+} = S_A^{-1} \circ \X^{-}$ where $S_A = \prod_{\gamma \in \widetilde{\Gamma}_{\mathrm{light}}: \ell_{\gamma}(u) \in V_A} \mathcal{K}_{\gamma}^{\Omega(\gamma; u)}$. (We note that because the bracket is trivial on $\widehat{\Gamma}_{\mathrm{light}}$, there is no subtlety about order. Moreover, $\widetilde{\Gamma}_{\mathrm{light}}$ is finite, so there are no concerns about infinite sums and products.)
\item $\X$ obeys the reality condition $\X(-1/\bar \zeta) = \rho^* \X(\zeta)$ where $\rho: \T'_u \to \T'_u$ is an antiholomorphic involution of $\T'_u$ defined by $\rho^* X_\gamma = \bar \X_{-\gamma}$.
\item For any $\gamma$, $\lim_{\zeta \to 0} \X_{\gamma}(\zeta)/ \X_{\gamma}^{\semif}(\zeta)$ exists.
\end{enumerate}
\end{rem}
\begin{rem}[Comments on Integral]\label{rem:commentsonintegral}
    Integrals like 
    \begin{equation}I_\Psi(\zeta)=\int_{r^-} \frac{d\zeta'}{\zeta'} \frac{\zeta'+\zeta}{\zeta'-\zeta} \Psi(\zeta'), \end{equation}
     where $\Psi(\zeta')$ is a holomorphic function of $\zeta'$, appear throughout this series of papers. Since this is the first appeance, we make a few initial comments. These will be discussed in much greater detail in follow up papers in their series.
     
     First, we observe that the integrand potentially has first order poles at $\zeta'=0, \zeta, \infty$. Supposing that the $\Psi(\zeta')$ vanishes at $\zeta'=0, \infty$ in order to cancel the poles, then the only remaining possible simple pole is at $\zeta'=\zeta$. Supposing that the $\Psi(\zeta')$ vanishes to second order at $\zeta'=0, \infty$, 
    and $\zeta \notin r$, then this is absolutely integrable. 
     The function $I_\Psi(\zeta)$ is a holomorphic function of $\zeta$ away from $r \backslash \{0, \infty\}$. 
     
     The residue theorem tells us that the jump at $\zeta \in r$ is $4 \pi i \Psi(\zeta)$. The integral along the ray $r$ makes sense in a regularized or distributional sense, albeit not in a absolutely integrable sense. 

We note that if $\Psi(\zeta')$ 
has the form $\log(1-\mathcal{X}^{\mathrm{sf}}_\gamma(\zeta'))$, then $\Psi(\zeta')$ has a zero of order $m$ at $0, \infty$ if $\mathcal{X}^{\semif}_\gamma(\zeta')$ has a zero of order $m$ at $0, \infty$. 

Given $r$ making an acute angle with $\ell_{\gamma}(u)$, define a decomposition of $Z_{\gamma}(u) = Z^\parallel_{\gamma}(u) + Z^\perp_{\gamma}(u)$, where $Z^\parallel_{\gamma}(u)$ is parallel to $r$. (See Figure \ref{fig:moveray}.)
 \begin{figure}[h!]
  \begin{centering}
 \includegraphics[width=2.5in]{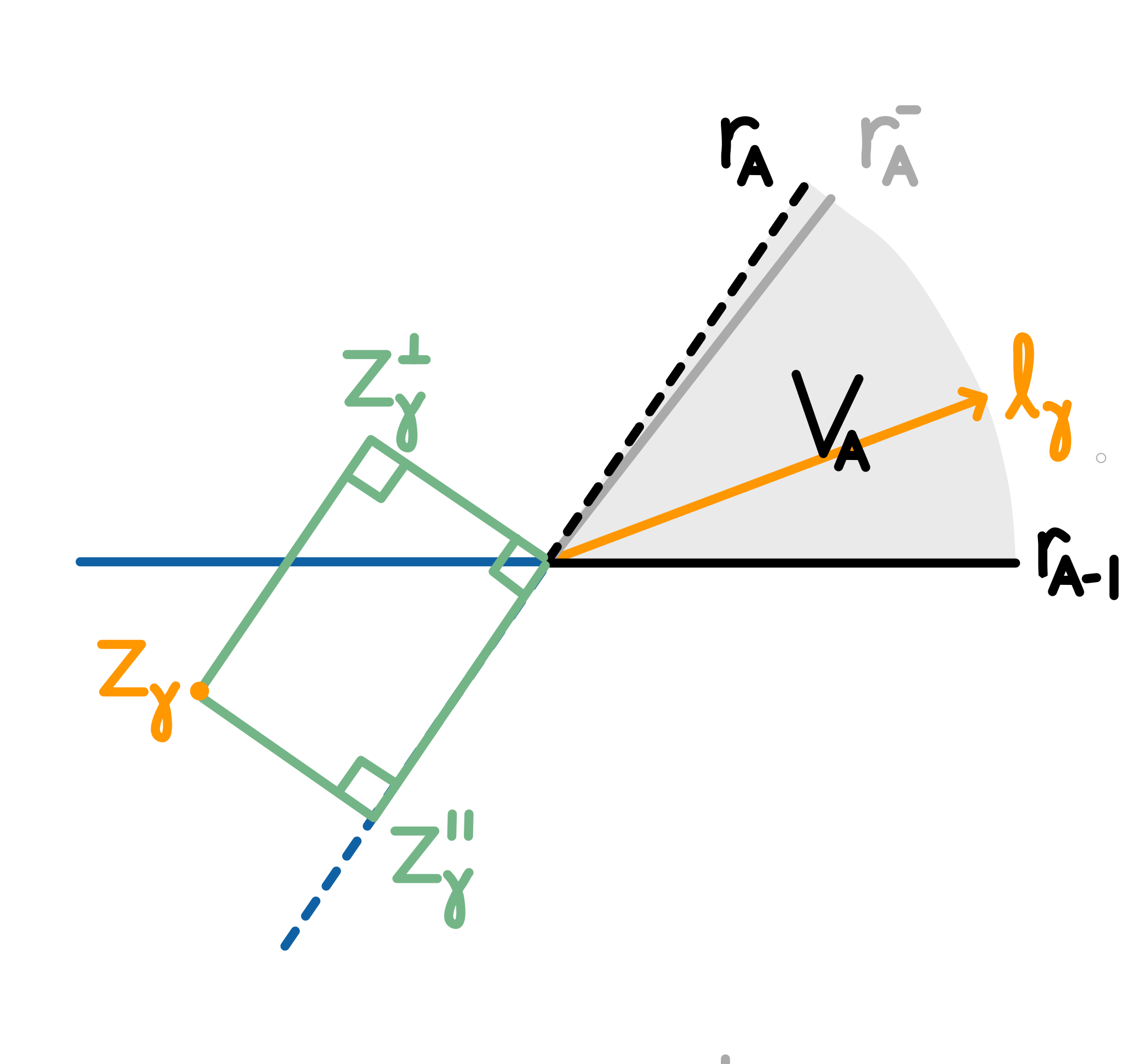}
 \caption{\label{fig:moveray}$Z_{\gamma}(u) = Z^\parallel_{\gamma}(u) + Z^\perp_{\gamma}(u)$} 
 \end{centering}
 \end{figure}
We note that if $\Psi(\zeta')$ has the form $\mathcal{X}^{\semif}_\gamma(\zeta')= \exp(\zeta'^{-1} Z_{\gamma'}(u) + i \theta_{\gamma'} + \zeta \overline{Z_{\gamma'}(u)}),$ then $\Psi(\zeta')$ vanishes to any order at $0, \infty$ along $r$ if $Z^\parallel_{\gamma'}(u)>0$.

Moreover, we note that these integrals are typically highly oscillatory. The polar decomposition of the function is
\be \mathcal{X}^{\semif}_\gamma(\zeta') = \exp \left( - (|\zeta'| + |\zeta'|^{-1}) |Z^\parallel_{\gamma}(u)| - \mathrm{Im}(\theta_\gamma')\right)  \cdot \exp\left( \zeta'^{-1} Z^\perp_{\gamma}(u) + \zeta' \overline{Z^\perp_{\gamma}(u)} +i \mathrm{Re}(\theta_\gamma')   \right),
\ee
where we note that the second term has unit norm. In particular, if $Z^\perp_{\gamma}(u)$ is non-zero, as is the case if $r \neq \ell_{\gamma}(u)$ then the oscillation near $\zeta'=0$ is dominated by $\exp(\zeta'^{-1} Z_{\gamma}^\perp (u))$ up to a phase, and similarly at infinity.  The integral defining $I_\Psi(\zeta')$ is only converging because the real part of $\Psi$ vanishes sufficiently quickly at $0, \infty$.

Lastly, we note that one can analytically continue $I_\Psi(\zeta)$ counterclockwise or clockwise past the jump discontinity along $r$ simply by rotating the ray that we integrate along, provided that the new ray continues to make an acute angle with $\ell_{\gamma}(u)$.
\end{rem}

\bigskip

We analytically continue these functions clockwise and counterclockwise. Note that here we are using that the angle between $\ell_{\gamma'}(u)$ and $r_{a, A(\gamma') \pm 1}$ is strictly acute since it is bounded above by $2 \phi$ and we took $K \geq 5$.
\begin{defn}[$\X^{{\rm model}, \pm, a}$]\label{defn:Xmodelpm} For each $\zeta \in \CC^\times$ and $m \in \M'$, define 
\be \X^{{\rm model}, \pm,a}_\gamma(\zeta) = \X^{\rm sf}_\gamma(\zeta) \exp\parens{-\frac{1}{4\pi i} \sum_{\gamma'\in \tilde\Gamma_{\rm light}} \Omega(\gamma') \avg{\gamma,\gamma'} \int_{r^-_{a,A(\gamma') \pm 1}} \frac{d\zeta'}{\zeta'} \frac{\zeta'+\zeta}{\zeta'-\zeta} \log(1-\X^{\rm sf}_{\gamma'}(\zeta')) } \ . \ee
The minus superscript handles the possibility that $\zeta\in r_{a,A(\gamma') \pm 1}$.
\end{defn}

\begin{rem}[Reality Condition]
The functions $\mathcal{X}^{\mathrm{model}, (\pm), a}$ satisfy the reality condition (analogous to \eqref{eq:xSFReality}):
    \be \overline{\X^{{\rm model},a}_\gamma(\zeta)} = 1/\X^{{\rm model},a}_\gamma(-1/\bar \zeta) \ , \label{eq:xModelReality} \ee
\end{rem}

\begin{proposition}\label{prop:dominatedconvergence}
    The integrals appearing  in the definition, Definition \ref{def:Xmodel} of $\X^{{\rm model},a}_\gamma(\zeta)$ converge. The integral $\X^{{\rm model},a}_\gamma(\zeta)$ is equal to the sum of the following absolutely convergent series
    \be \label{eq:modelvssf2}\X^{{\rm model},a}_\gamma(\zeta) = \X^{\rm sf}_\gamma(\zeta) \exp\parens{\frac{1}{4\pi i} \sum_{\gamma'\in \Gamma_{\rm light}} \bar\Omega(\gamma') \avg{\gamma,\gamma'} \int_{r^-_{a,A(\gamma')}} \frac{d\zeta'}{\zeta'} \frac{\zeta'+\zeta}{\zeta'-\zeta} \X^{\rm sf}_{\gamma'}(\zeta') } \ . \ee
Define $\X^{{\rm model},a}_{\CC,\gamma}$ using the formula in Definition \ref{defn:Xmodel}, but allow $\theta_\gamma$ to be complex-valued. The above result holds provided we assume  $-2|Z_{\gamma}|\cos\phi < \mathrm{Im}(\theta_{\gamma}(u))$.
A similar result is true of the analytic continuations $\X^{{\rm model},\pm, a}_{\gamma_i}(\zeta)$.
\end{proposition}

 \begin{proof}[Proof of Proposition \ref{prop:dominatedconvergence}] 
 Since $|\X_{(\CC),\gamma'}^{\rm sf}(\zeta')|\le e^{-2|Z_{\gamma'}|\cos\phi} e^{-\mathrm{Im}(\theta_{\gamma'}(u))} < 1$ 
 on $V_{a,A(\gamma')}$, $\log(1-\X^{\rm sf}_{(\CC), \gamma'}(\zeta'))$ is unambiguously defined on $V_{a,A(\gamma')}$ 
 as the uniformly and absolutely convergent series $-\sum_{n>0} \frac{1}{n}(\X^{\rm sf}_{(\CC),\gamma'}(\zeta'))^n$. The dominated convergence theorem gives
 \be \label{eq:modelvssf2}\X^{{\rm model},a}_{(\CC),\gamma}(\zeta) = \X^{\rm sf}_{(\CC),\gamma}(\zeta) \exp\parens{\frac{1}{4\pi i} \sum_{\gamma'\in \Gamma_{\rm light}} \bar\Omega(\gamma') \avg{\gamma,\gamma'} \int_{r^-_{a,A(\gamma')}} \frac{d\zeta'}{\zeta'} \frac{\zeta'+\zeta}{\zeta'-\zeta} \X^{\rm sf}_{(\CC),\gamma'}(\zeta') } \ , \ee
 where this series converges absolutely. In particular, $\X^{{\rm model},a}_{(\CC),\gamma}(\zeta) = \X^{\rm sf}_{(\CC),\gamma}(\zeta)$ if $\gamma\in \Gamma_{\rm local}$. We also have
 \be \overline{\X^{{\rm model},a}_{(\CC),\gamma}(\zeta)} = 1/\X^{{\rm model},a}_{(\CC),\gamma}(-1/\bar \zeta) \ , \label{eq:xModelReality} \ee
  which is analogous to \eqref{eq:xSFReality}.
 \end{proof}

\bigskip

We now provide an analogue of Lemma \ref{lem:nondegen}:
\begin{lemma}[Nondegeneracy] \label{lem:nondegenModel} 
Let $\{\gamma_{e_i}, \gamma_{m_i}\}$ be a lift to $\hat\Gamma$ of a Frobenius basis of sections of $\Gamma$ over $U_a$.  
Then, for every $\zeta\in \CC^\times$, $\{\X^{{\rm model},a}_{\gamma_i}(\zeta)\}$ provide valid coordinates on $\pi^{-1}(U_a) \subset \M'$.
A similar result is true of the analytic continuations $\{\X^{{\rm model},\pm, a}_{\gamma_i}(\zeta)\}$.
\end{lemma}
\begin{proof}
We follow the proof of Lemma \ref{lem:nondegen}.  Define $\tilde \Y^{\rm sf} = \log \tilde\X^{\rm sf}$ and 
$\tilde \Y^{\rm model} = \log \tilde\X^{\rm model}$ so 
\begin{align*}
\Imag \tilde \Y^{\rm sf}_{\gamma_{e_i}} &= \frac{1}{2i}(\zeta^{-1} - \bar \zeta) Z_{\gamma_{e_i}} + \tilde \theta_{\gamma_{e_i}} + \frac{1}{2i}(\zeta - \bar \zeta^{-1}) \bar Z_{\gamma_{e_i}} \\
\Real \tilde \Y^{\rm sf}_{\gamma_{e_i}} &= \frac{1}{2}(\zeta^{-1} +\bar \zeta) Z_{\gamma_{e_i}}  + \frac{1}{2}(\zeta + \bar \zeta^{-1}) \bar Z_{\gamma_{e_i}} \\
\frac{1}{p_i} \Imag \tilde \Y^{\rm model}_{\gamma_{m_i}} &= \frac{1}{2i}(\zeta^{-1} - \bar \zeta) a_i^D + \frac{1}{p_i} \tilde \theta_{\gamma_{m_i}} + \frac{1}{2i}(\zeta - \bar \zeta^{-1})  \bar a_i^D+ \\
& \qquad  \frac{1}{p_i} \Imag \parens{\frac{1}{4\pi i} \sum_{\gamma'\in \Gamma_{\rm light}} \bar\Omega(\gamma') c_{\gamma', i} \int_{r^-_{a,A(\gamma')}} \frac{d\zeta'}{\zeta'} \frac{\zeta'+\zeta}{\zeta'-\zeta} \X^{\rm sf}_{\gamma'}(\zeta') } \\
\frac{1}{p_i} \Real \tilde\Y^{\rm model}_{\gamma_{m_i}} &= \frac{1}{2}(\zeta^{-1} +\bar \zeta) a_i^D  + \frac{1}{2}(\zeta + \bar \zeta^{-1})  \bar a_i^D  +\\
& \qquad \frac{1}{p_i}  \Real \parens{\frac{1}{4\pi i} \sum_{\gamma'\in \Gamma_{\rm light}} \bar\Omega(\gamma') c_{\gamma', i} \int_{r^-_{a,A(\gamma')}} \frac{d\zeta'}{\zeta'} \frac{\zeta'+\zeta}{\zeta'-\zeta} \X^{\rm sf}_{\gamma'}(\zeta') }
\end{align*}
As before, the proof reduces to the computation of the Jacobian 
\be \abs{\frac{\partial(\{\Imag \tilde\Y^{\rm sf}_{\gamma_{e_i}}\}_{i=1}^r 
, \{\frac{1}{p_i}\Imag \tilde\Y^{{\rm model},a}_{\gamma_{m_i}} \}_{i=1}^r, 
\{\Real \tilde\Y^{\rm sf}_{\gamma_{e_i}}\}_{i=1}^r,
\{\frac{1}{p_i}\Real \tilde\Y^{{\rm model},a}_{\gamma_{m_i}}\}_{i=1}^r
}{\partial(\tilde\theta_{\gamma_{e_1}},\ldots,\tilde\theta_{\gamma_{e_r}},\frac{1}{p_1}\tilde\theta_{\gamma_{m_1}},\ldots,\frac{1}{p_r}\tilde\theta_{\gamma_{m_r}},a_1,\ldots,a_r,\bar a_1,\ldots,\bar a_r)}} \ , \ee
where $\gamma_{e_1},\ldots,\gamma_{e_r},\gamma_{m_1},\ldots,\gamma_{m_r}$ is a lift to $\hat\Gamma\otimes\QQ$ of a symplectic basis of sections of $\Gamma\otimes\QQ$ over $U_a$ such that $\gamma_{e_i}\in \Gamma_{\rm local}$, and the notation otherwise parallels that from the earlier proof. The Jacobian matrix is of the form
\begin{align}
\begin{pmatrix}
I_{r} & 0_r & \frac{1}{2i}(\zeta^{-1}-\bar\zeta) I_{r} & \frac{1}{2i}(\zeta-\bar\zeta^{-1}) I_{r} \\
* & I_r & * & * \\
0_r & 0_r & \half (\zeta^{-1}+\bar\zeta) I_r & \half(\zeta+\bar\zeta^{-1}) I_r \\
N & 0_r & \half (\zeta^{-1}+\bar\zeta) (\tau + M) & \half (\zeta+\bar\zeta^{-1}) (\bar\tau + \bar M)
\end{pmatrix} \ , \label{eq:modelJacobianMat}
\end{align}
where the $*$ terms are irrelevant. As before, the
Jacobian is found to equal
\begin{align}
&\begin{vmatrix}
\half(\zeta^{-1}+\bar\zeta)I_r & \half(\zeta+\bar\zeta^{-1})I_r \\
\half(\zeta^{-1}+\bar\zeta)(\tau+M)-\frac{1}{2i}(\zeta^{-1}-\bar\zeta) N & \half(\zeta+\bar\zeta^{-1})(\bar\tau+\bar M)-\frac{1}{2i}(\zeta-\bar\zeta^{-1}) N
\end{vmatrix} \nonumber \\
&= \abs{\frac{\zeta^{-1}+\bar\zeta}{2}}^{2r} (-2i)^r \det\parens{\Imag(\tau+M) + \frac{1-|\zeta|^2}{1+|\zeta|^2} N} \ .
\end{align}
We will now show that $\Imag (\tau+M)+\frac{1-|\zeta|^2}{1+|\zeta|^2}N$ is precisely the ($\zeta$-independent!) positive-definite matrix $V$ from Lemma \ref{lem:gwOV}\ref{it:Vij}, which we remind the reader is 
\begin{align}
V_{ij}(u,\theta) 
&= \Imag\tau_{ij}(u) + \frac{1}{4\pi} \sum_{\gamma\in\tilde\Gamma_{\rm light}} \Omega(\gamma) p_i^{-1} p_j^{-1} c_{\gamma,i} c_{\gamma,j} \int_{\ell_\gamma(u)} \frac{d\zeta}{\zeta} \frac{\X^{\rm sf}_\gamma(\zeta)}{1-\X^{\rm sf}_\gamma(\zeta)} \label{eq:Vnewagain},
\end{align}
where $\ell_\gamma(u)$ is the contour from 0 to $\infty$ along which $\arg(-Z_\gamma(u)/\zeta)=0$.
To this end, we first write 
\be \frac{1}{p_i}\tilde\Y^{{\rm model},a}_{\gamma_{m_i}}(\zeta) = \frac{a^D_i}{\zeta}+i\frac{1}{p_i}\tilde\theta_{\gamma_{m_i}}+\zeta \overline{a^D_i} - \frac{1}{4\pi i} \frac{1}{p_i}\sum_{\gamma'\in \tilde\Gamma_{\rm light}} \Omega(\gamma') c_{\gamma',i} \int_{r^-_{a,A(\gamma')}} \frac{d\zeta'}{\zeta'} \frac{\zeta'+\zeta}{\zeta'-\zeta} \log(1-\X^{\rm sf}_{\gamma'}(\zeta')) \ . \ee
Using \eqref{eq:xSFReality}, we compute
\begin{align}
    \frac{1}{p_i}\overline{\tilde \Y^{{\rm model},a}_{\gamma_{m_i}}} &= \frac{\overline{a_i^D}}{\bar\zeta} - i \frac{1}{p_i} \tilde\theta_{\gamma_{m_i}} + \bar\zeta a_i^D - \frac{1}{4\pi i} \frac{1}{p_i} \sum_{\gamma'\in\tilde\Gamma_{\rm light}} \Omega(-\gamma') c_{-\gamma',i} \int_{r^-_{a,A(-\gamma')}} \frac{d\zeta'}{\zeta'} \frac{-\frac{1}{\zeta'}+\bar\zeta}{-\frac{1}{\zeta'}-\bar\zeta} \log(1-\X^{\rm sf}_{-\gamma'}(\zeta')) \nonumber \\
&= \frac{\overline{a_i^D}}{\bar\zeta} - i \frac{1}{p_i} \tilde\theta_{\gamma_{m_i}} + \bar\zeta a_i^D - \frac{1}{4\pi i} \frac{1}{p_i}\sum_{\gamma'\in\tilde\Gamma_{\rm light}} \Omega(\gamma') c_{\gamma',i} \int_{r^-_{a,A(\gamma')}} \frac{d\zeta'}{\zeta'} \frac{-\frac{1}{\zeta'}+\bar\zeta}{-\frac{1}{\zeta'}-\bar\zeta} \log(1-\X^{\rm sf}_{\gamma'}(\zeta')) \ .
\end{align}
So,
\begin{align}\frac{1}{p_i} \Real \tilde\Y^{{\rm model},a}_{\gamma_{m_i}}(\zeta) &= \Real(a_i^D (\zeta^{-1}+\bar\zeta)) \nonumber \\
    & \quad - \frac{1}{8\pi i} \frac{1}{p_i} \sum_{\gamma'\in\tilde\Gamma_{\rm light}} \Omega(\gamma') c_{\gamma',i} \int_{r^-_{a,A(\gamma')}} \frac{d\zeta'}{\zeta'} \brackets{\frac{\zeta'+\zeta}{\zeta'-\zeta}+\frac{-\frac{1}{\zeta'}+\bar\zeta}{-\frac{1}{\zeta'}-\bar\zeta}} \log(1-\X^{\rm sf}_{\gamma'}(\zeta')) \ . \end{align}
Differentiating this with respect to $a_j,\bar a_j,\tilde\theta_{\gamma_{e_j}}$ gives
\begin{align}
M_{ij} &= - \frac{1}{4\pi i (\zeta^{-1}+\bar\zeta)} \sum_{\gamma'\in\tilde\Gamma_{\rm light}} \Omega(\gamma') p_i^{-1} p_j^{-1} c_{\gamma',i} c_{\gamma',j} \int_{r^-_{a,A(\gamma')}} \frac{d\zeta'}{\zeta'} \brackets{\frac{\zeta'+\zeta}{\zeta'-\zeta}+\frac{-\frac{1}{\zeta'}+\bar\zeta}{-\frac{1}{\zeta'}-\bar\zeta}} \frac{-\frac{1}{\zeta'} \X_{\gamma'}^{\rm sf}(\zeta')}{1-\X^{\rm sf}_{\gamma'}(\zeta')} \ , \nonumber \\
\overline{M_{ij}} &= -\frac{1}{4\pi i (\zeta+\bar\zeta^{-1})} \sum_{\gamma'\in \tilde\Gamma_{\rm light}} \Omega(\gamma') p_i^{-1} p_j^{-1} c_{\gamma',i} c_{\gamma',j} \int_{r^-_{a,A(\gamma')}} \frac{d\zeta'}{\zeta'} \brackets{\frac{\zeta'+\zeta}{\zeta'-\zeta}+\frac{-\frac{1}{\zeta'}+\bar\zeta}{-\frac{1}{\zeta'}-\bar\zeta}} \frac{-\zeta' \X_{\gamma'}^{\rm sf}(\zeta')}{1-\X^{\rm sf}_{\gamma'}(\zeta')} \ , \nonumber \\
N_{ij} &= - \frac{1}{8\pi i} \sum_{\gamma'\in \tilde\Gamma_{\rm light}} \Omega(\gamma') p_i^{-1} p_j^{-1} c_{\gamma',i} c_{\gamma',j} \int_{r^-_{a,A(\gamma')}} \frac{d\zeta'}{\zeta'} \brackets{\frac{\zeta'+\zeta}{\zeta'-\zeta}+\frac{-\frac{1}{\zeta'}+\bar\zeta}{-\frac{1}{\zeta'}-\bar\zeta}} \frac{-i\X^{\rm sf}_{\gamma'}(\zeta')}{1-\X^{\rm sf}_{\gamma'}(\zeta')} \ .
\end{align}
Putting these together, we find that
\be \Imag(M_{ij})+\frac{1-|\zeta|^2}{1+|\zeta|^2} N_{ij} = \frac{1}{4\pi} \sum_{\gamma\in \tilde\Gamma_{\rm light}} \Omega(\gamma) p_i^{-1} p_j^{-1} c_{\gamma,i} c_{\gamma,j} \int_{r^-_{a,A(\gamma)}} \frac{d\zeta}{\zeta} \frac{\X_{\gamma}^{\rm sf}(\zeta)}{1-\X^{\rm sf}_\gamma(\zeta)} \ . \ee
After rotating the contour to $\ell_\gamma(u)$, this equals the second term in \eqref{eq:newV}, and so we conclude that, indeed
\begin{equation} 
V = \Imag (\tau+M)+\frac{1-|\zeta|^2}{1+|\zeta|^2}N.
\end{equation}
With this, we are done.
\end{proof}

\begin{definition}[model holomorphic symplectic forms]
    For every $\zeta\in\CC^\times$, we now define a closed 2-form on $\pi^{-1}(U_a)$ 
\begin{equation}
    \varpi^{{\rm model},(\pm), a}(\zeta) := \frac{1}{8\pi} \avg{d\log \X^{{\rm model},(\pm), a}(\zeta) \wedge d\log \X^{{\rm model},(\pm) a}(\zeta)}. 
\end{equation}
\end{definition}
By Lemma \ref{lem:nondegenModel}, this is a holomorphic symplectic form on the underlying real manifold $\M'$.

\begin{cor}[Corollary to \ref{lem:nondegenModel}]\label{cor:patch} The family $\varpi^{{\rm model},(\pm), a}(\zeta)$ gives a pseudo-hyper-K\"ahler structure on $\M'_{U_a}$.
\end{cor}
\begin{proof}
By Corollary \ref{cor:twistor}, this follows from Corollary \ref{lem:nondegenModel}.
\end{proof}

\begin{lemma}\label{lem:coordind}The holomorphic symplectic form is equal to
    \begin{align}\label{eq:modelv2}
\varpi&^{{\rm model}, (\pm), a}(\zeta) = \varpi^{\rm sf}(\zeta) + \frac{1}{16\pi^2 i} \sum_{\gamma'\in\tilde\Gamma_{\rm light}} \Omega(\gamma') d\log \X^{\rm sf}_{\gamma'}(\zeta) \nonumber \\
&\qquad \wedge \sum_{n\not=0} e^{in\theta_{\gamma'}} \brackets{K_0(2|nZ_{\gamma'}|)(\zeta^{-1}dZ_{\gamma'}-\zeta d\bar Z_{\gamma'}) - |Z_{\gamma'}| \sgn(n) K_1(2|nZ_{\gamma'}|)d\log(Z_{\gamma'}/\bar Z_{\gamma'})} \ .
\end{align}
\end{lemma}
\begin{proof}
We give the proof for $\varpi^{{\rm model},a}(\zeta)$.
We compute
\begin{align}
\varpi&^{{\rm model},a}(\zeta) := \frac{1}{8\pi} \avg{d\log \X^{{\rm model},a}(\zeta) \wedge d\log \X^{{\rm model},a}(\zeta)} \nonumber \\
&= \varpi^{\rm sf}(\zeta)-\frac{1}{4\pi} \sum_{i=1}^r d\log \X^{\rm sf}_{\gamma_{e_i}}(\zeta) \wedge \frac{1}{p_i} d\log (\X^{{\rm model},a}_{\gamma_{m_i}}(\zeta)/\X^{\rm sf}_{\gamma_{m_i}}(\zeta)) \nonumber \\
&= \varpi^{\rm sf}(\zeta) - \frac{1}{16\pi^2 i} \sum_{i=1}^r d\log \X^{\rm sf}_{\gamma_{e_i}}(\zeta)\wedge \sum_{\gamma'\in\tilde\Gamma_{\rm light}} \Omega(\gamma') p_i^{-1} c_{\gamma',i} \int_{r^-_{a,A(\gamma')}} \frac{d\zeta'}{\zeta'} \frac{\zeta'+\zeta}{\zeta'-\zeta} d\log \X^{\rm sf}_{\gamma'}(\zeta') \frac{\X^{\rm sf}_{\gamma'}(\zeta')}{1-\X^{\rm sf}_{\gamma'}(\zeta')} \nonumber \\
&= \varpi^{\rm sf}(\zeta) - \frac{1}{16\pi^2 i} \sum_{\gamma'\in\tilde\Gamma_{\rm light}} \Omega(\gamma') \int_{r^-_{a,A(\gamma')}} \frac{d\zeta'}{\zeta'} \frac{\zeta'+\zeta}{\zeta'-\zeta} \frac{\X^{\rm sf}_{\gamma'}(\zeta')}{1-\X^{\rm sf}_{\gamma'}(\zeta')} d\log \X^{\rm sf}_{\gamma'}(\zeta)\wedge d\log \X^{\rm sf}_{\gamma'}(\zeta') \nonumber \\
&= \varpi^{\rm sf}(\zeta) - \frac{1}{16\pi^2 i} \sum_{\gamma'\in\tilde\Gamma_{\rm light}} \Omega(\gamma') \int_{r^-_{a,A(\gamma')}} \frac{d\zeta'}{\zeta'} \frac{\zeta'+\zeta}{\zeta'-\zeta} \frac{\X^{\rm sf}_{\gamma'}(\zeta')}{1-\X^{\rm sf}_{\gamma'}(\zeta')} \nonumber \\
&\qquad \times d\log \X^{\rm sf}_{\gamma'}(\zeta)\wedge (d\log \X^{\rm sf}_{\gamma'}(\zeta')-d\log\X^{\rm sf}_{\gamma'}(\zeta)) \nonumber \\
&= \varpi^{\rm sf}(\zeta) + \frac{1}{16\pi^2 i} \sum_{\gamma'\in\tilde\Gamma_{\rm light}} \Omega(\gamma') d\log \X^{\rm sf}_{\gamma'}(\zeta)\wedge \int_{r^-_{a,A(\gamma')}} \frac{d\zeta'}{\zeta'} \frac{\X^{\rm sf}_{\gamma'}(\zeta')}{1-\X^{\rm sf}_{\gamma'}(\zeta')}  \nonumber \\
&\qquad\times \brackets{ (\zeta^{-1}+\zeta'^{-1}) dZ_{\gamma'} - (\zeta+\zeta') d\bar Z_{\gamma'} } \nonumber \\
&= \varpi^{\rm sf}(\zeta) + \frac{1}{16\pi^2 i} \sum_{\gamma'\in\tilde\Gamma_{\rm light}} \Omega(\gamma') d\log\X^{\rm sf}_{\gamma'}(\zeta) \wedge \brackets{\I_{1,\gamma'} (\zeta^{-1} dZ_{\gamma'} - \zeta d\bar Z_{\gamma'}) + \I_{2,\gamma'}dZ_{\gamma'}-\I_{0,\gamma'} d\bar Z_{\gamma'} } \ ,
\end{align}
where
\be \I_{\nu,\gamma'} = \int_{r^-_{a,A(\gamma')}} \frac{d\zeta'}{\zeta'^{\nu}} \frac{\X^{\rm sf}_{\gamma'}(\zeta')}{1-\X^{\rm sf}_{\gamma'}(\zeta')} \ . \ee
We have already computed $\I_{1,\gamma'}$, and the other integrals may be computed identically:
\be -\frac{\bar Z_{\gamma'}}{|Z_{\gamma'}|} \I_{0,\gamma'} = -\frac{Z_{\gamma'}}{|Z_{\gamma'}|} \I_{2,\gamma'} = 2 \sum_{n>0} e^{in\theta_{\gamma'}} K_1(2|nZ_{\gamma'}|) \ , \ee
where as before this converges absolutely and uniformly on compact subsets of $\pi^{-1}(U_a)$. Combining terms from $\pm\gamma'$ gives equation \eqref{eq:modelv2}
\end{proof}

\begin{cor}[Corollary to Lemma \ref{lem:coordind}]\label{cor:modind}
 On $U_a\cap U_b$, $\varpi^{{\rm model},a}=\varpi^{{\rm model},b}$! Moreover, $\varpi^{{\rm model},a} = \varpi^{{\rm model}, +, a}=\varpi^{{\rm model}, -, a}$ on $U_a$.
\end{cor}
\begin{rem}[Explanation of Corollary \ref{cor:modind}]
First, for any $u\in U_a\cap U_b$ and $\zeta\in\CC^\times$, let $\T'_{\zeta,u,{\rm light},a,b}$ be the open subset of $\T'_{\zeta,u}$ with $|X_\gamma|<1$ for all $\gamma\in \tilde\Gamma_{\rm light}$ such that $\ell_\gamma$ and $\zeta$ are contained in the same sector in either the decomposition $\V_a$ or $\V_b$. Note that $\X(\zeta,u)$ is valued in this space on $\pi^{-1}(U_a\cap U_b)$. Then $\mathcal{X}^a_\gamma(\zeta)$ and $\mathcal{X}^b_\gamma(\zeta)$ are related by some particular symplectomorphism of $\T'_{\zeta,u,{\rm light},a,b}$.
\end{rem}

We exploit this by dropping the $a$ superscript, and $\pm$ superscripts for the analytic continuations.

\begin{proposition}\label{prop:extendssmoothly} The $\mathbb{P}^1$-family of closed $2$-forms 
 $\varpi^{\rm model}(\zeta)$ on $\M'_U$ extends to a $\mathbb{P}^1$-family of smooth closed $2$-forms on all of $\M_U$! 
\end{proposition}

\begin{rem} Note that we never extended $\mathcal{X}^{\mathrm{model}, a}$ over $\mathcal{B}''$. This was by design. As discussed in Remark \ref{rem:commentsonintegral}, the integrals defining $\mathcal{X}^{\mathrm{model}, a}$ do not converge when $Z_{\gamma'}(u)=0$. The modified Bessel functions of the second kind, $K_0(2|nZ_{\gamma'}|)$ and $K_1(2 |nZ_{\gamma'}|)$, appearing Lemma \ref{lem:coordind} similarly aren't defined.
One can make sense of  the integral defining $\mathcal{X}_{\gamma}^{\mathrm{model}, a}$ over $\mathcal{B}''$ when $r_{a, A(\gamma')}=\ell_{\gamma'}(u),$ by doing a regularized integral that cancels\footnote{This is the key idea in \cite[Lemma 3.1]{garza:singular1}. } the simple pole at $0$ with the simple pole at $\infty$; however, this is not necessary.
\end{rem}
\begin{proof}
The proof proceeds in two parts. First, we check that that the holomorphic symplectic form extends to $\widetilde{\M}'$; then we check that the holomorphic symplectic form extends to $\tilde{\M}$.

We first show that the holomorphic symplectic form extends to $\widetilde{\M}$, starting with the equality in Lemma \ref{lem:coordind}. 
Within $\varpi^{\mathrm{model}}(\zeta)$, there are $\log|Z_{\gamma'}|$ singularities in both $\varpi^{\rm sf}(\zeta)$ (thanks to the presence of $\Imag\tau$ --- recall \eqref{eq:imTildeTau}) and the $K_0$ sum, but miraculously these cancel out! By Lemma \ref{lem:GHsf}, we can write 
\begin{align}
\varpi^{\rm sf}(\zeta)
&= \frac{1}{4\pi i} \sum_{i=1}^r d\log \X^{\rm sf}_{\gamma_{e_i}}(\zeta)\wedge \brackets{ p_i^{-1} d\theta_{\gamma_{m_i}}+A_i^{\mathrm{sf}} + \sum_{j=1}^r V^{\mathrm{sf}}_{ij}(\zeta^{-1} da_j-\zeta d\bar a_j) } \ ,
\end{align}
where
\begin{align*}
A_i^{\mathrm{sf}} &= \sum_{j=1}^r- \Real\tau_{ij} d\theta_{\gamma_{e_j}}\\
V_{ij}^{\semif} &=\Imag\tau_{ij},
\end{align*}
and similarly, we discern that
\begin{align}
\varpi^{{\rm model}}(\zeta) &= \frac{1}{4\pi i} \sum_{i=1}^r d\log \X^{\rm sf}_{\gamma_{e_i}}(\zeta)\wedge \brackets{p_i^{-1}d\theta_{\gamma_{m_i}} + A_i + \sum_{j=1}^r V_{ij}(\zeta^{-1} da_j-\zeta d\bar a_j)} \ , \label{eq:varpiModelExplicit}
\end{align}
where $V_{ij}$ is given in \eqref{eq:newV} and the connection is
\be
A_i = - \sum_{j=1}^r \Real \tau_{ij}d\theta_{\gamma_{e_j}} - \frac{1}{4\pi} \sum_{\gamma'\in\tilde\Gamma_{\rm light}} \Omega(\gamma') p_i^{-1} c_{\gamma',i} |Z_{\gamma'}| d\log(Z_{\gamma'}/\bar Z_{\gamma'}) \sum_{n\not=0} e^{in\theta_{\gamma'}} \sgn(n) K_1(2|nZ_{\gamma'}|) \label{eq:ai2}
\ee
equals \eqref{eq:ai}. (Again, see pages 36-37 of \cite{mz:K3HK} for an explanation of this Poisson resummation identity.) With this, all of the singularities are in the connection $p_i(p_i^{-1}\de \theta_{m_i} + A_i)$. In the proof of Lemma \ref{lem:gwOV}\ref{it:bundle},
we explicitly showed that the possibly problematic connection  $p_i(p_i^{-1}\de \theta_{m_i} + A_i)$ appearing \eqref{eq:varpiModelExplicit} extends smoothly to $\widetilde{\M}'$. 

\bigskip

Now, we show, that the holomorphic symplectic form $\varpi^{\mathrm{model}}$ extends smoothly to all of $\mathcal{M}$. Let $u'' \in U' \cap \mathcal{B}''$ and let $\theta \in \Theta_{\mathrm{light}}$ such that $\widetilde{S}_{u'', \theta} \neq \emptyset$. Then, the structure of the smooth manifold is constructed by Proposition \ref{lem:SmoothII} and Proposition \ref{prop:TN} moreover constructs a smooth hyper-K\"ahler metric in a neighborhood.

Since we know $\varpi^{\mathrm{TN}}(\zeta)$ is smooth on $\M$,
we will show that the difference of $\varpi^{\mathrm{model}}(\zeta)$  and $\varpi^{\mathrm{TN}}(\zeta)$ is smooth on $\M'$.

 We begin with the expression for $\varpi^{\mathrm{model}}(\zeta)$ in \eqref{eq:varpiModelExplicit}, but we take a modified form 
\begin{align}
\varpi^{{\rm model}}(\zeta) &= \frac{1}{4\pi i} \sum_{i=1}^r d\log \X^{\rm sf}_{\gamma_{e_i}}(\zeta)\wedge \brackets{p_i^{-1} d\theta'_{\gamma_{m_i}} + A'_i + \sum_{j=1}^r V_{ij}(\zeta^{-1} da_j-\zeta d\bar a_j)} \ , 
\end{align}
where the connection is 
\begin{align} A'_i =& - \sum_{j=1}^r \Real\tilde\tau_{ij} d\theta_{\gamma_{e_j}}\\ \nonumber
&  + \frac{1}{4}\sum_{\gamma'\in\tilde\Gamma_{\rm light}} \Omega(\gamma') p_i^{-1} c_{\gamma',i} d\arg Z_{\gamma'} \sum_{m\in\ZZ} \parens{\frac{2\pi m+\theta_{\gamma'}}{\sqrt{4|Z_{\gamma'}|^2+(2\pi m +\theta_{\gamma'})^2}} - \sgn(2\pi m+\theta_{\gamma'})} \  \end{align}
and the potential is 
\be V_{ij}(u,\theta) = \Imag\tilde\tau_{ij}(u) + \frac{1}{4\pi} \sum_{\gamma\in\tilde\Gamma_{\rm light}} \Omega(\gamma)p_i^{-1} p_j^{-1} c_{\gamma,i} c_{\gamma,j} T(Z_\gamma(u),\theta_\gamma) \ee
where 
\be T(w,\theta) = \sum_{m=-\infty}^\infty \parens{\frac{\pi}{\sqrt{4|w|^2+(2\pi m+\theta)^2}} - \kappa_m} \ , \quad \kappa_m = \piecewise{\half \log \frac{2|m|+1}{2|m|-1}}{m\not=0}{0}{m=0}.
\ee
As we observed in Proposition \ref{prop:TN}, the Taub-NUT metric is given by 
\begin{align}
\varpi^{{\rm TN}}(\zeta) &= \frac{1}{4\pi i} \sum_{i=1}^r d\log \X^{\rm sf}_{\gamma_{e_i}}(\zeta)\wedge \brackets{p_i^{-1}d\theta'_{\gamma_{m_i}} + A'^{\mathrm{TN}} _i + \sum_{j=1}^r V^{\mathrm{TN}}_{ij}(\zeta^{-1} da_j-\zeta d\bar a_j)} \ . 
\end{align}
Here,
one can obtain the expression for ${A'_i}^{\rm TN}$ in the following predicted way (1) summing over $\gamma' \in \widetilde{S}_{u'', \theta},$ (2) taking only the $m=0$ piece, and (3) setting $\widetilde{\tau}=0$. Note that we will use our assumption $\theta_{\gamma'} \in (0, 2 \pi)$.  I.e.,   
we write the connection and potential as a sum over the singleton $m \in \{0\}$: 
\begin{align} A'^{\mathrm{TN}}_i =& \frac{1}{4}\sum_{\gamma'\in\tilde S_{u'', \theta}} \Omega(\gamma') p_i^{-1} c_{\gamma',i} d\arg Z_{\gamma'}  \sum_{m \in 
\{0\}}\parens{\frac{2\pi m+\theta_{\gamma'}}{\sqrt{4|Z_{\gamma'}|^2+(2\pi m +\theta_{\gamma'})^2}} - \sgn(2\pi m+\theta_{\gamma'})} \ . \end{align}
and the potential is 
\be V^{\mathrm{TN}}_{ij}(u,\theta) = \frac{1}{4\pi} \sum_{\gamma\in\tilde\Gamma_{\rm light}} \Omega(\gamma)p_i^{-1} p_j^{-1} c_{\gamma,i} c_{\gamma,j} T_0(Z_\gamma(u),\theta_\gamma) \ee
where 
\be T_0(w,\theta) = \sum_{m \in \{0\}}\parens{\frac{\pi}{\sqrt{4|w|^2+(2\pi m+\theta)^2}} - \kappa_m} \ , \quad \kappa_m = \piecewise{\half \log \frac{2|m|+1}{2|m|-1}}{m\not=0}{0}{m=0}
\ee

The difference is 
\begin{align}
&\varpi^{\mathrm{model}}(\zeta)-\varpi^{\mathrm{TN}}(\zeta) :=\nonumber \\ & \qquad \frac{1}{4 \pi i} \sum_{i=1}^r d \log \X^{\semif}_{\gamma_{e_i}}(\zeta) \wedge \left[(A'_i - {A'_i}^{\mathrm{TN}}) + (V_{ij}-V_{ij}^{\mathrm{TN}}) (\zeta^{-1} d a_j - \zeta d \bar a_j)  \right].
\end{align}
Now, we refer the reader to the change of coordinate formulas appearing in Lemma \ref{lem:Jacobians}. In smooth coordinates on $\M$, the $1$-form $d\theta'_{m_i}$ is singular; however, it does not appear in the above difference. 

The second most problematic term is the $ d \mathrm{arg} Z_{\gamma'}=d \mathrm{arg} w_{\gamma',1} + d \mathrm{arg} w_{\gamma',2}$ term which appears in $A_i'- {A_i'}^{\rm TN}$; this is not defined when $w_{\gamma',1}=0$ or $w_{\gamma', 2}=0$. However, writing  
\begin{align}
A_i'- {A_i'}^{\rm TN}&= - \sum_{j=1}^r \mathrm{Re} \widetilde \tau_{ij} d \theta_{\gamma_{e_j}} + \frac{1}{2} \sum_{\gamma' \in S} \Omega(\gamma') p_i^{-1} c_{\gamma', i} d \mathrm{arg} Z_{\gamma'}  \nonumber\\
& \qquad  \times \underbrace{\sum_{m \in \ZZ-\{0\}} \parens{\frac{2\pi m+\theta_{\gamma'}}{\sqrt{4|Z_{\gamma'}|^2+(2\pi m +\theta_{\gamma'})^2}} - \sgn(2\pi m+\theta_{\gamma'})}}_{Q(|Z_{\gamma'}|, \theta_{\gamma'})},
\label{eq:Adiff}
\end{align}
one observes that for each $\gamma' \in S$
 \begin{align*}
 &\sum_{i=1}^r d \log \mathcal{X}_{\gamma_{e_i}}^{\semif}(\zeta) \wedge \frac{1}{2} \Omega(\gamma') p_i^{-1} c_{\gamma', i} d \mathrm{arg} Z_{\gamma'} Q(|Z_{\gamma'}|, \theta_{\gamma'})\\
      &=\sum_{i=1}^r  \Omega(\gamma') p_i^{-1} c_{\gamma', i}  (\zeta^{-1} d a_i + i d \theta_{e_i} + \zeta d \bar a_i) \wedge   \frac{1}{2}  d \mathrm{arg} Z_{\gamma'} Q(|Z_{\gamma'}|, \theta_{\gamma'})\\
            &= (\zeta^{-1} d Z_{\gamma'} + i d \theta_{e_i} + \zeta d \bar Z_{\gamma'}) \wedge   \frac{1}{2} \sum_{\gamma' \in S} d \mathrm{arg} Z_{\gamma'} f(|Z_{\gamma'}|, \theta_{\gamma'}).
      \end{align*}
     The $1$-form $(\zeta^{-1} d Z_{\gamma'} + i d \theta_{e_i} + \zeta d \bar Z_{\gamma'})$ vanishes at $w_{\gamma',1}=0, w_{\gamma', 2}=0$, so the wedge product is continuous on $\M$.
     
     All of the other $1$-forms appearing in the difference are individually continuous, and all of the other functions, like $V_{ij}-V_{ij}^{\rm TN}$ are continuous on $\M$.  

     Lastly, to talk about the smoothness of the difference---not just the continuity of the difference---on $\M$, we need to show that the derivatives with respect to the smooth coordinates on $\M$ are continuous. The possibly problematic derivatives are $\frac{\partial}{\partial w_{\gamma', 1}}, \frac{\partial}{\partial \bar w_{\gamma', 1}},\frac{\partial}{\partial w_{\gamma', 2}}, \frac{\partial}{\partial \bar w_{\gamma', 2}}$, since in the expansion in terms of the  original smooth coordinates on $\M'$, the coefficients of the fiber partial derivatives  $\frac{\partial}{\partial \tilde \theta'_{m_i}}$ are singular (see Remark \ref{rem:vectorfields} and the computations in the proof of Lemma \ref{lem:Jacobians}). However, since the difference is independent of the fiber coordinates, smoothness in the coordinates on $(\Imag \HH)^{r}$ implies smoothness in the coordinates on $\M$.

     \bigskip

Lastly, since $d\varpi^{\mathrm{model}}(\zeta)=0$ on a dense subset $\M'_U$ of $\M_U$, $\varpi^{\mathrm{model}}(\zeta)$ is closed on $\M_U$.
\end{proof}
\begin{rem}[Explicit form] 
We write
\be \varpi^{\rm model}(\zeta) = -\frac{i}{2\zeta} \omega^{\rm model}_+ + \omega^{\rm model}_3 - \frac{i}{2} \zeta \omega^{\rm model}_- \ , \ee
noting in particular that the $\zeta^{\pm 2}$ terms vanish. Explicitly,
\begin{align}
\omega^{\rm model}_+ &= \frac{1}{2\pi} \sum_{i=1}^r da_i\wedge (p_i^{-1} d\theta_{\gamma_{m_i}}+A_i-i\sum_j V_{ij} d\theta_{\gamma_{e_j}}) \ , \nonumber \\
\omega_3^{\rm model} &= \frac{i}{2\pi} \sum_{i,j} V_{ij} da_i\wedge d\bar a_j + \frac{1}{4\pi} \sum_i d\theta_{\gamma_{e_i}}\wedge (p_i^{-1}d\theta_{\gamma_{m_i}}+A_i) \ , \nonumber \\
\omega_-^{\rm model} &= \overline{\omega_+^{\rm model}} \ ,
\end{align}
where $A_i$, $i=1, \cdots, r$ describe the $U(1)^r$ connection described in \eqref{eq:ai}.
\end{rem}

Lastly, we determine holomorphic Darboux coordinates for $\omega_+^{\rm model}$ corresponding to $\zeta=0$. 

\begin{lemma}
The closed $2$-form  $\omega_+$ is equal to \be \omega_+^{\rm model} = \sum_{i=1}^r da_i \wedge d  z^{\rm model}_i, \ee 
for 
\be z^{\rm model}_i = \underbrace{\frac{p_i^{-1}  \theta_{\gamma_{m_i}} - \sum_{j=1}^r \tau_{ij} \theta_{\gamma_{e_j}}}{2 \pi}}_{z_i}+ 
\left( -\frac{1}{8\pi^2} \sum_{\gamma\in\tilde\Gamma_{\rm light}} \Omega(\gamma) p_i^{-1}  c_{\gamma,i}  \sum_{n\not=0} \frac{1}{n} e^{in\theta_\gamma} K_0(2|n Z_\gamma(u)|)
    \right)
\ee 
The determinant of 
\be\label{eq:Jac1} \abs{\frac{\partial(\{a_i\}_{i=1}^r 
, \{\bar a_i \}_{i=1}^r, 
\{z_i^{\rm model}\}_{i=1}^r \{\bar z_i^{\rm model}\}_{i=1}^r 
)}
{\partial(\tilde\theta_{\gamma_{e_1}},\ldots,\tilde\theta_{\gamma_{e_r}},
\frac{1}{p_1}\tilde\theta'_{\gamma_{m_1}},\ldots,\frac{1}{p_r}\tilde\theta'_{\gamma_{m_r}},
a_1,\ldots,a_r,\bar a_1,\ldots,\bar a_r)}} = \ . \ee

\end{lemma}

\begin{proof}
 We begin with the expression for $V_{ij}$ in \eqref{eq:newV} of Lemma \ref{lem:gwOV}\ref{it:Vij}:
 \begin{align}
    V_{ij}(u,\theta)&= \Imag \tau_{ij}(u) + \frac{1}{4\pi} \sum_{\gamma\in\tilde\Gamma_{\rm light}} \Omega(\gamma) p_i^{-1} p_j^{-1} c_{\gamma,i} c_{\gamma,j} \sum_{n\not=0} e^{in\theta_\gamma} K_0(2|n Z_\gamma(u)|)
    \end{align}
    and the expression for $A_i$ in \eqref{eq:a12}:
    \be
    A_i = - \sum_{j=1}^r \Real \tau_{ij}d\theta_{\gamma_{e_j}} - \frac{1}{4\pi} \sum_{\gamma'\in\tilde\Gamma_{\rm light}} \Omega(\gamma') p_i^{-1} c_{\gamma',i} |Z_{\gamma'}| d\log(Z_{\gamma'}/\bar Z_{\gamma'}) \sum_{n\not=0} e^{in\theta_{\gamma'}} \sgn(n) K_1(2|nZ_{\gamma'}|).
    \ee
    Then we note that $\omega_+^{\rm model} = \sum_{i=1}^r da_i \wedge \beta_i$ for 
    \begin{align}
    \beta_i &=\frac{1}{2 \pi}\left(p_i^{-1} d \theta_{\gamma_{m_i}} + A_i - i \sum_{j=1}^r V_{ij} d \theta_{\gamma_{e_j}}\right)\nonumber \\
     &= \frac{1}{2 \pi}\left(p_i^{-1} d \theta_{\gamma_{m_i}}  
     - \sum_{j=1}^r \Real \tau_{ij}d\theta_{\gamma_{e_j}} - i \sum_{j=1}^r 
     \Imag \tau_{ij}(u) d \theta_{\gamma_{e_j}}
    \right) \nonumber \\
     & \qquad + \frac{1}{2 \pi}\left( 
     \frac{1}{4\pi} \sum_{\gamma'\in\tilde\Gamma_{\rm light}} \Omega(\gamma') p_i^{-1} c_{\gamma',i} |Z_{\gamma'}| d\log(Z_{\gamma'}/\bar Z_{\gamma'}) \sum_{n\not=0} e^{in\theta_{\gamma'}} \sgn(n) K_1(2|nZ_{\gamma'}|) \right. \nonumber \\
     & \qquad  \left.
      - 
       \frac{i}{4\pi} \sum_{\gamma\in\tilde\Gamma_{\rm light}} \Omega(\gamma) p_i^{-1}  c_{\gamma,i}  \sum_{n\not=0} e^{in\theta_\gamma} K_0(2|n Z_\gamma(u)|)
      d \theta_{\gamma}\right)\nonumber \\
     &= \frac{1}{2 \pi}\left(p_i^{-1} d \theta_{\gamma_{m_i}} 
      - \sum_{j=1}^r \Real \tau_{ij}d\theta_{\gamma_{e_j}} - i \sum_{j=1}^r 
      \Imag \tau_{ij}(u) d \theta_{\gamma_{e_j}}
     \right) \nonumber  \\
     &  \qquad +  d \left( -\frac{1}{8\pi^2} \sum_{\gamma\in\tilde\Gamma_{\rm light}} \Omega(\gamma) p_i^{-1}  c_{\gamma,i}  \sum_{n\not=0} \frac{1}{n} e^{in\theta_\gamma} K_0(2|n Z_\gamma(u)|)
    \right) \nonumber \\
     & \qquad - \frac{1}{4 \pi^2}  
     \sum_{\gamma'\in\tilde\Gamma_{\rm light}} \Omega(\gamma') p_i^{-1} c_{\gamma',i} |Z_{\gamma'}| d\log(Z_{\gamma'}) \sum_{n\not=0} e^{in\theta_{\gamma'}} \sgn(n) K_1(2|nZ_{\gamma'}|). 
    \end{align}
Recalling that for the semi-flat we have holomorphic Darboux coordinates $\omega_+^{\rm sf} = \sum_{i=1}^r da_i \wedge dz_i$ for 
\be z_i = \frac{p_i^{-1} \theta_{\gamma_{m_i}} - \sum_{j=1}^r \tau_{ij} \theta_{\gamma_{e_j}}}{2 \pi},\ee
we take 
\be z^{\rm model}_i = \underbrace{\frac{p_i^{-1}  \theta_{\gamma_{m_i}} - \sum_{j=1}^r \tau_{ij} \theta_{\gamma_{e_j}}}{2 \pi}}_{z_i}+ 
\left( -\frac{1}{8\pi^2} \sum_{\gamma\in\tilde\Gamma_{\rm light}} \Omega(\gamma) p_i^{-1}  c_{\gamma,i}  \sum_{n\not=0} \frac{1}{n} e^{in\theta_\gamma} K_0(2|n Z_\gamma(u)|)
    \right)
\ee 
and check that  $\omega_+^{\rm model} = \sum_{i=1}^r da_i \wedge dz^{\rm model}_i$, as claimed, precisely because 
\begin{align}
&\sum_{i=1}^r d Z_{\gamma_{e_i}} \wedge  -\frac{1}{4 \pi^2}  
\sum_{\gamma'\in\tilde\Gamma_{\rm light}} \Omega(\gamma') p_i^{-1} c_{\gamma',i}  \frac{|Z_{\gamma'}|}{Z_{\gamma'}}  dZ_{\gamma'}\sum_{n\not=0} e^{in\theta_{\gamma'}} \sgn(n) K_1(2|nZ_{\gamma'}|)\nonumber \\
&=  -\frac{1}{4 \pi^2}  
\sum_{\gamma'\in\tilde\Gamma_{\rm light}} \Omega(\gamma') dZ_{\gamma'} \wedge dZ_{\gamma'} \frac{ |Z_{\gamma'}|}{Z_{\gamma'}} \sum_{n\not=0} e^{in\theta_{\gamma'}} \sgn(n) K_1(2|nZ_{\gamma'}|) \nonumber \\
&=0
\end{align}

Lastly, we compute the determinant, making use of our previous computations in the proof of Lemma \ref{lem:omega+sf}.
The relevant Jacobian matrix this time is
\be
\frac{\partial(a_1,\ldots,a_r,\bar a_1,\ldots,\bar a_r,z^{\rm model}_1,\ldots,z^{\rm model}_r,\bar z^{\rm model}_1,\ldots,\bar z^{\rm model}_r)}{\partial(a_1,\ldots,a_r,\bar a_1,\ldots,\bar a_r,\frac{1}{p_1}\tilde\theta_{\gamma_{m_1}},\ldots,\frac{1}{p_r} \tilde\theta_{\gamma_{m_r}},\tilde\theta_{\gamma_{e_1}},\ldots,\tilde\theta_{\gamma_{e_r}})} \ . \ee
Noting that 
\be \frac{\partial z_i^{\rm model}}{\partial \theta_{\gamma_{e_j}}} = \frac{\partial z_i}{\partial \theta_{\gamma_{e_j}}}
-\frac{i}{2 \pi}   (V_{ij}- \Imag \tau_{ij}) 
     \ee 
     and 
    \be \frac{\partial z_i^{\rm model}}{\partial \theta_{\gamma_{m_j}}} = \frac{\partial z_i}{\partial \theta_{\gamma_{m_j}}},    \ee
 the Jacobian is 
\be \begin{pmatrix}
I_r & 0 & 0 & 0 \\
0 & I_r & 0 & 0 \\
* & * & \frac{1}{2\pi} I_r   & -\frac{1}{2\pi} \tau  - \frac{i}{2 \pi}   (V- \Imag \tau) \\
* &*& \frac{1}{2\pi} I_r  & -\frac{1}{2\pi} \bar\tau +  \frac{i}{2 \pi}   (V- \Imag \tau)
\end{pmatrix} \ ,
\ee
where the $*$ terms do not matter in the computation of the determinant.
As before, the determinant is found to be 
\be \begin{vmatrix}
\frac{1}{2\pi} I_r  & -\frac{1}{2\pi} \tau -  \frac{i}{2 \pi}   (V- \Imag \tau)\\
\frac{1}{2\pi} I_r  & -\frac{1}{2\pi} \bar\tau +  \frac{i}{2 \pi}   (V- \Imag \tau)
\end{vmatrix} =
(2\pi)^{-2r} (2i)^r \det V \ ,
\ee
which is nonzero. 

\end{proof}

\begin{theorem}[Model hyper-K\"ahler geometry]\label{thm:modelgeom}
    The $\mathbb{P}^1$-family of closed $2$-forms $\varpi^{\mathrm{model}}(\zeta)$  
    produce a hyper-K\"ahler structure on $\M_U$.
\end{theorem}

\begin{proof} 
    We will use Corollary \ref{cor:twistor}. From Proposition \ref{prop:extendssmootly}, for $\zeta \in \mathbb{P}^1$, $\varpi^{\mathrm{model}}(\zeta)$ is smooth on $\M$. In Lemma \ref{lem:nondegenModel}, we proved that the Jacobian determinant of the holomorphic Darboux coordinates with respect to the smooth coordinates on $\M'$ was non-singular. 

    While, we did not define the coordinates $\X^{\mathrm{model},a}$ over $\B''$, since we already showed that $\varpi^{\mathrm{model}}(\zeta)$,$ \omega^{\rm model}_+$, $\omega^{\rm model}_-$ are smooth on $\M|_U$, the non-vanishing of 
    $\varpi^{\mathrm{model}}(\zeta)^r \wedge \overline{\varpi^{\mathrm{model}}(\zeta)^r }$ ammounts to checking that limit of the following Jacobians as one approaches the a point in $\M_U - \M'_U$ are non-vanishing:
    \be\label{eq:Jac1} \abs{\frac{\partial(\{\Imag \tilde\Y^{\rm sf}_{\gamma_{e_i}}\}_{i=1}^r 
    , \{\frac{1}{p_i}\Imag \tilde\Y^{{\rm model},a}_{\gamma_{m_i}} \}_{i=1}^r, 
    \{\Real \tilde\Y^{\rm sf}_{\gamma_{e_1}}\},
    \{\frac{1}{p_i}\Real \tilde\Y^{{\rm model},a}_{\gamma_{m_i}}\}
    }{\partial(\tilde\theta_{\gamma_{e_1}},\ldots,\tilde\theta_{\gamma_{e_r}},\frac{1}{p_1}\tilde\theta'_{\gamma_{m_1}},\ldots,\frac{1}{p_r}\tilde\theta'_{\gamma_{m_r}},a_1,\ldots,a_r,\bar a_1,\ldots,\bar a_r)}} \ , \ee
    and 
    \be \label{eq:Jac2}\abs{\frac{\partial(\{\Imag \tilde\Y^{\rm sf}_{\gamma_{e_i}}\}_{i=1}^r 
    , \{\frac{1}{p_i}\Imag \tilde\Y^{{\rm model},a}_{\gamma_{m_i}} \}_{i=1}^r, 
    \{\Real \tilde\Y^{\rm sf}_{\gamma_{e_1}}\},
    \{\frac{1}{p_i}\Real \tilde\Y^{{\rm model},a}_{\gamma_{m_i}}\}
    }{\partial(\{w_{\gamma_i, 1}, \bar w_{\gamma_i, 1}, w_{\gamma_i, 2}, \bar w_{\gamma_i, 2}\}_{i=1}^r,  \tilde\theta_{\gamma_{e_{s+1}}},\ldots,\tilde\theta_{\gamma_{e_r}},\tilde\theta'_{\gamma_{m_{s+1}}},\ldots,\tilde\theta'_{\gamma_{m_r}},a_{s+1},\ldots,a_r,\bar a_{s+1},\ldots,\bar a_r)}} \ . \ee

  The non-vanishing of \eqref{eq:Jac1} follows from the non-vanishing with respect to the coordinates on $\M'$ in Lemma \ref{lem:nondegenModel} and the relation between local smooth coordinates on $\M_U'$ and $\tilde \M_U'$ in  Lemma \ref{lem:primedJacobian}
  since 
  \be \abs{\frac{\partial(\tilde\theta_{\gamma_{e_1}},\ldots,\tilde\theta_{\gamma_{e_r}},\frac{1}{p_1}\tilde\theta_{\gamma_{m_1}},\ldots,\frac{1}{p_r}\tilde\theta_{\gamma_{m_r}},a_1,\ldots,a_r,\bar a_1,\ldots,\bar a_r)}{\partial(\tilde\theta_{\gamma_{e_1}},\ldots,\tilde\theta_{\gamma_{e_r}},\frac{1}{p_1}\tilde\theta'_{\gamma_{m_1}},\ldots,\frac{1}{p_r}\tilde\theta'_{\gamma_{m_r}},a_1,\ldots,a_r,\bar a_1,\ldots,\bar a_r)}} =1. \ee

  To see the non-vanishing of \eqref{eq:Jac2} we again use the relation betwen local smooth coordinates in 
  on $\M_U'$ and $\M_U$ in Lemma \ref{lem:Jacobians}.
From our expression for Jacobian in new coordiantes with respect to old, we see that for   
\[ d \mu = \bigwedge_{i=1}^r \left( d w_{\gamma_i, 1} \wedge d \bar w_{\gamma_i, 1}  \wedge d w_{\gamma_i, 2} \wedge d \bar w_{\gamma_i, 2} \right) \wedge \bigwedge_{i=r+1}^s \left(d\tilde\theta_{\gamma_{e_i}}  \wedge p_i^{-1} d\tilde\theta'_{\gamma_{m_i}}\wedge da_i \wedge d\bar a_i \right) .\]
and then we compute 
\begin{align*}
\left|\iota_{d \mu}\varpi^{\mathrm{model}}(\zeta)\right| = \abs{\frac{\zeta^{-1}+\bar\zeta}{2}}^{2r} \det\parens{V}  \prod_{\gamma \in S}\frac{|q_\gamma|^2}{4}
\end{align*}
Consequently, we just need to examine the limit of this as $|q_\gamma| \to 0$.
 We see then that $T(Z_\gamma(u), \theta_\gamma) |q_\gamma|^2$ is non-zero and finite when $|q_\gamma| \to 0$. This verifies the non-degeneracy.  From Corollary \ref{cor:twistor}, 
 $\dim_{\CC}\ker \varpi^{\mathrm{model}}(\zeta) \geq 2r$ on $\M_U$. Hence for $\zeta \in \CC^\times$, $\varpi(\zeta)$ is a holomorphic symplectic form on the underlying real manifold $\M$.
 
 The proof that $\omega_+$ is a holomorphic symplectic form follows similarly from Lemma \ref{lem:omega+model}.
 
 Thus, by Theorem \ref{thm:twistor}, we have a pseudo-hyper-K\"ahler structure on $\M_U$. It is a hyper-K\"ahler structure since $V$ is positive definite.
\end{proof}

\begin{ex}[Multi-Ooguri-Vafa models with $p \neq 1$]\label{ex:OVagain}
As promised in Remark \ref{rem:pnot1}, we briefly discuss the model hyper-K\"ahler geometries constructed from the index $p$ sublattice $\mathrm{Span}\left(\gamma_1, \cdots, \gamma_N, \gamma_e, p \gamma_m\right)$ of the usual multi-Ooguri-Vafa lattice.  As Lemma \ref{lem:coordind} makes clear, the model geometry constructed from the sublattice is is the same as the model geometry constructed from the lattice itself; this is because the expression for $\varpi^{\rm model}(\zeta)$ includes $\varpi^{\rm sf}$ (which is unchanged) and a sum over the trivial $\widetilde{\Gamma}_{\rm light}$ lattice (which is unchanged); it makes no reference to the symplectic pairing. 

We briefly discuss this from the perspective of the 4d Gibbons-Hawking Ansatz. In the 4d Gibbons-Hawking Ansatz, the harmonic function $V=\frac{1}{4 \pi |y|}$ on $\RR^3_y$ with associated connection $d\theta_m + A$ corresponds to flat quaternion space; the harmonic function $V=\frac{p}{4 \pi |y|}$ with associated connection $d\theta_m + p A$ corresponds to $\RR^4/\ZZ_p$; dividing $V$ by $p$ corresponds to a cover of flat $\RR^4$ ramified at $q=0$. By analogy,  we take $V_{(p)}=p^{-1}\frac{1}{4 \pi |y|}$ and connection $\Theta_{(p)}=d\theta_m + A_{(p)}=d \theta_m + \frac{1}{p} A_{(1)}$, where $V_{(1)}, A_{(1)}$ are the Gibbons-Hawking data of the usual multi-Ooguri-Vafa metric. Then one notes that 
\begin{align}
pg_{(p)} &= p \left(V_{(p)} \norm{dy}^2 + V_{(p)}^{-1} \Theta_{(p)}^2 \right) \nonumber \\
&=V_{(1)}  \norm{dy}^2 + V_{(1)}^{-1} (pd \theta_m +  A_{(1)})^2 \nonumber \\
&= V_{(1)}  \norm{dy}^2 + V_{(1)}^{-1} (d \theta_{p \gamma_m} +A_{(1)})^2
\end{align}
is simply the usual Ooguri-Vafa model metric but with coordinate $\theta_{\gamma_m}$ replaced by $\theta_{p \gamma_m}$.
\end{ex}

\appendix

\section{Lattices and symplectic pairings} \label{sec:lattices}
We refer the reader to \cite[\S II.6]{GriffithsHarris} for more details.

\begin{definition}
A \emph{lattice} is a finitely-generated free $\ZZ$-module. 
\end{definition}

\begin{rem} In parts of the literature, a lattice is often equipped with an inner product. We do not require this. \end{rem}

\begin{definition}
A sublattice $M$ of a lattice $L$ is \emph{primitive} if $L/M$ is also a lattice.
\end{definition}

\begin{lemma} \label{lem:primitive}
For a sublattice $M$ of a lattice $L$, the following are equivalent:
\begin{enumerate}[(i)]
\item $M$ is primitive;
\item there is a sublattice $N$ of $L$ such that $L\cong M\oplus N$, where this isomorphism maps $M$ to $M\oplus 0$ and $N$ to $0\oplus N$;
\item $M=(M\otimes\RR)\cap L$.
\end{enumerate}
\end{lemma}
\begin{proof}\hfill

\noindent(i)$\Rightarrow$(ii): choosing a basis for $L/M$ and lifting it to $L$ gives the desired isomorphism.
\smallskip

\noindent(ii)$\Rightarrow$(iii) is clear.
\smallskip

\noindent(iii)$\Rightarrow$(i): if $k[\gamma]=0$ for some $\gamma\in L$ and associated $[\gamma]\in L/M$ and some nonzero $k\in \ZZ$, then $k\gamma\in M$ implies that $\gamma\in (M\otimes\RR)\cap L=M$, so that $[\gamma]=0$.
\end{proof}

\begin{definition}
Let $L$ be a rank $2r$ lattice with an antisymmetric $\ZZ$-valued pairing $\avg{,}$. This pairing is \emph{nondegenerate} (resp. \emph{unimodular}) if it defines an injection (resp. bijection) from $L$ to the dual lattice $\Hom(L,\ZZ)$. A \emph{symplectic pairing} on $L$ is a nondegenerate antisymmetric integral pairing.
\end{definition}

The following definitions will characterize the best bases of symplectic pairings on $L$:
\begin{definition}
\label{def:Frobeniusbasis}
Let $L$ be a rank $2r$ lattice with a symplectic pairing $\avg{ , }$.  
A \emph{Frobenius basis} is a basis 
$\alpha_1,\ldots,\alpha_r,\beta_1,\ldots,\beta_r\in L$ such that  $\avg{\alpha_i,\alpha_j}=\avg{\beta_i,\beta_j}=0$ and $\avg{\beta_i, \alpha_i}= \delta_{ij}p_i$ where the non-zero integers $p_i$ satisfy $p_1 | p_2| \cdots |p_r$. 
A \emph{symplectic basis} is a basis $\alpha_1,\ldots,\alpha_r,\beta_1,\ldots,\beta_r\in L$ such that $\avg{\alpha_i,\alpha_j}=\avg{\beta_i,\beta_j}=0$ and $\avg{\beta_i,\alpha_j}=\delta_{ij}$. 
\end{definition}

\begin{lemma}[Lemma 1 in VI.\S3 of \cite{LangAlgAbel}, Lemma on p. 304 \cite{GriffithsHarris}]\label{lem:Frobeniusbasis}
 A lattice $L$ with symplectic pairing $\avg{,}$ admits a Frobenius basis and the ideals $\ZZ p_i$ are uniquely determined. 
\end{lemma}
Following \cite[p. 306]{GriffithsHarris}, we will call $p_i$ \emph{elementary divisors}.
\begin{proof}
    The lemma is proved by induction. Among the set of all valued $\avg{\gamma, \gamma'}$ for $\gamma, \gamma' \in L$, let $d_1$ be the least positive one and take a pair $\alpha_1, \beta_1$ such that $\avg{\alpha_1, \beta_1}=p_1$. Let $L_1=\mathrm{Span}(\alpha_1, \beta_1)$ and let \[L_1^\perp=\{\gamma: \avg{\gamma, \alpha_1}=0, \avg{\gamma, \beta_1}=0 \}. \]
    The intersection $L_1 \cap L_1^\perp =\{0\}$. To see that $L=L_1 +L_1^\perp$, we use the standard Gram-Schmidt orthogonalization procees. Given $\gamma \in L$, we can find $a, b\in \ZZ$ and $\gamma^\perp \in L_1^\perp$ such that 
    \[\gamma = \gamma^\perp + a \alpha_1 + b \beta_1.\]
    In particular, letting 
    \[ 0=\avg{\gamma^\perp, \alpha_1}=\avg{\gamma - a \alpha_1 - b \beta_1, \alpha_1} = \avg{\gamma, \alpha_1} + b p_1,\]
    observe that since $\ZZ$ is a principle integral domain, $p_1$ is the positive generator of the ideal of all values of the bilinear form, hence $p_1$ divides $\avg{\gamma, \alpha_1}$, hence we can solve for $b$. We can similarly solve for $a$.

    We now repeat the process on the subbundle $L_1^\perp$.
\end{proof}

\begin{rem}
    Let $A$ be a non-principally-polarized abelian variety and let $\widehat{A}$ be its dual.
    Then, given a line bundle $\mathcal{L}$ on $A$,
    there is a map
    \begin{align}
    A &\to \widehat{A} \nonumber\\
    a &\mapsto t_a^* \mathcal{L} \otimes \mathcal{L}^*,
    \end{align}
    where $t_a$ is translation by $a$. In particular, this map is surjective, if and only if, $\mathcal{L}$ is ample. Taking $\mathcal{L}$ ample, the kernel of this map is $(\ZZ/p_1)^2 \times \cdots \times (\ZZ/p_r)^2$
    for some integers $p_1|p_2|\cdots |p_r$. 

    Given a symplectic lattice $(L, \avg{,})$, the space of twisted unitary characters on $L$ defined by
    \[ \]
    is an abelian variety $A$. 
\end{rem}

\begin{cor}
       Every non-unimodular lattice $(L, \avg{, })$ is a submodule of some unimodular lattice $(L', \avg{, })$. 
\end{cor}
\begin{proof}
Take a Frobenius basis $\alpha_1, \cdots, \alpha_r, \beta_1, \cdots, \beta_r$ of $(L, \avg{, })$. Define a superlattice $L'$ generated by $\alpha_1, \cdots, \alpha_r, \beta_1', \cdots, \beta_r'$ where $\beta_i' = \frac{\beta_i}{p_i}$. 
\end{proof}

\begin{defn}
We say that a sublattice $\Lambda$ of a rank $2r$ symplectic lattice $(L, \avg{,})$ is \emph{isotropic} if $\avg{\Lambda, \Lambda}=0$.
We say $\Lambda$ is \emph{maximal isotropic} if $\Lambda$ is of rank $r$.
\end{defn}

In the unimodular case, one can naturally extend certain sets to a symplectic basis. 

\begin{lemma}
Let $(L, \avg{,})$ a lattice with a symplectic unimodular pairing. Given $\alpha_1, \cdots, \alpha_r$ with $r<s$ generating a primitive isotropic sublattice, $\alpha_1, \cdots, \alpha_r$ can be extended to a symplectic basis.
\end{lemma}
\begin{proof}
See proof of Lemma \ref{lem:latBasis}.
\end{proof}
However, the same is not true in the non-unimodular case.
\begin{ex}
    Given $\alpha_1, \cdots, \alpha_s$ with $s \leq r$ generating a primitive isotropic sublattice. It is not true that, up to reordering, $\alpha_1, \cdots, \alpha_s$ can be extended to a Frobenius basis. 
    
    Consider the counterexample of $\ZZ^4$ with $\alpha_1=(1,0, 0, 0)^T$ and $\alpha_2= (0, 0, 1, 0)^T$ and $\avg{,}$ given by 
    \[\begin{pmatrix} 0  & 2 & 0 & 0 \\ -2 & 0 & 0 & 0 \\ 0 & 0 & 0 & 3\\ 0 & 0 & -3 & 0 \end{pmatrix}.\]
    One can see that constants for the Frobenius basis are $d_1=1$ and $d_2=6$, while $\mathrm{Im} \avg{\alpha_1, \cdot} = 2\ZZ$ and $\mathrm{Im} \avg{\alpha_2, \cdot} = 3\ZZ$. 

It is not even possible to do this if one defines a more general ``quasi-Frobenius basis" with no further assumptions on the integers $d_i$. In particular, note that  the Frobenius basis for the previous example is given by 
\[ \alpha'_1= \begin{pmatrix} 1 \\ 0 \\1\\0 \end{pmatrix} \beta'_1 = \begin{pmatrix} 0 \\ 2 \\ 0 \\-1 \end{pmatrix} \alpha'_2 = \begin{pmatrix} 3 \\ 0 \\ 2 \\0 \end{pmatrix}, \beta'_2 = \begin{pmatrix} 0 \\ 1 \\ 0 \\0 \end{pmatrix},\]
which indeed satisfy $\avg{\beta'_1, \alpha'_1}=1$ and $\avg{\beta'_2, \alpha'_2}=6$. 
Note that $\alpha_1$ and $\alpha_1'$ generate the same primitive isotropic lattice as $\alpha_1$ and $\alpha_2$, but it is not possible to complete $\alpha_1$ and $\alpha_1'$ to such a ``quasi-Frobenius basis'', essentially since if such a basis existed then a matrix of determinant $6^2$ would be similar to a matrix of determinant $2^2 \cdot 1^2$. More explicitly, $\ker \avg{\alpha_1', \cdot}=\ZZ  \cdot (0, 3, 0, -2)^T$, and the image of the pairing with $\alpha_1$ is $6\ZZ$; $\ker \avg{\alpha_1, \cdot} = \ZZ \cdot (0, 0, 0, 1)$, and the image of the pairing with $\alpha_1'$ is $3\ZZ$. So these 
 \[ \alpha_1, \alpha_1', (0, 3, 0, -2)^T, (0, 0, 0, 1)^T\]
are not a basis for the full lattice, but only a basis for the sublattice $\ZZ \oplus 3 \ZZ \oplus \ZZ \oplus \ZZ$.
\end{ex}

However, the following is true:

\begin{lemma}\label{lem:latBasis}
Let $L$ be a lattice with a symplectic pairing $\avg{,}$, and let $\Lambda_\alpha$ be a primitive maximal isotropic sublattice. Then $(L, \avg{, })$ admits a Frobenius basis such that $\alpha_i \in \Lambda_\alpha$.
\end{lemma}

\begin{proof}The proof is a simple modification of the proof in Lemma \ref{lem:Frobeniusbasis}.
    We simply observe that the generator $p_1$ of 
    the image of $\avg{,}: L \times L \to \ZZ$ 
    and the generator $p_1'$ of $\avg{,}: \Lambda_\alpha \times L \to \ZZ$ are related by $p_1=\pm p_1'$.
    Clearly $p_1' \in p_1 \ZZ$ since $\Lambda_\alpha$ is a sublattice of $L$.
    Since $\Lambda_\alpha$ is a primitive maximal coisotropic lattice, the quotient lattice lifts (non-uniquely) to a primitive maximal coisotripic lattice $\Lambda_\beta$ 
    and $L=\Lambda_\alpha \oplus \Lambda _\beta$.  Suppose $d_1=\avg{\gamma, \gamma'}$ for $\gamma, \gamma' \in L$. Then split $\gamma = \gamma_\alpha + \gamma_\beta, \gamma' = \gamma'_\alpha + \gamma'_\beta$. Then 
    \[p_1 = \avg{\gamma, \gamma'} = \avg{\gamma_\alpha, \gamma'_\beta} - \avg{\gamma'_\alpha, \gamma_\beta} \in p_1' \ZZ.\]
    Since $p_1\ZZ=p_1'\ZZ$, $p_1=\pm p_1'$. 
\end{proof}

\begin{definition}
    Given a lattice $L$, the dual lattice $L^*=\mathrm{Hom}(L, \ZZ)$.
\end{definition}
It is straightforward to check the following Lemma: 
\begin{lemma}\label{lem:dualsymp}
    Given a lattice $L$ with a symplectic pairing $\avg{,}_L$, the dual lattice inherits a (not necessarily integral) symplectic pairing $\avg{,}_{L^*}$ defined as follows:
    for every $a \in L^*$, there exists a unique $\alpha=\Psi(a) \in L$ such that
    \[a(\gamma) = \avg{\alpha, \gamma} \qquad \forall \gamma \in L;\]
    define 
    \[\avg{a, b}_{L^*}:= \avg{\Psi(a), \Psi(b)}_{L}\]
\end{lemma}

Concretely, given a Frobenius basis $\alpha_i, \beta_i$, let $\alpha_i^*, \beta_i^*$ be the dual basis. Then we note that $\beta_i^* = -\frac{1}{p_i} \alpha_i$ and $\alpha_i^*= \frac{1}{p_i} \beta_i$. Then,
\[\avg{\beta_i^*, \alpha_j^*}_{L^*} = \avg{-\frac{1}{p_i} \alpha_i, \frac{1}{p_j} \beta_j}_{L} = \delta_{ij}\frac{1}{p_i}.\]

\bibliography{refs}

\end{document}